\documentclass[11pt, draft]{amsart}
\usepackage{amssymb, amstext, amscd, amsmath, amssymb}
\usepackage{mathtools, xypic, paralist, color, dsfont, rotating}
\usepackage{verbatim}
\usepackage{enumerate}
\usepackage{enumitem}
\usepackage{calc}
\usepackage{float}

\numberwithin{equation}{section}

\usepackage[a4paper]{geometry}
\usepackage{quoting}
\quotingsetup{vskip=.1in}
\quotingsetup{leftmargin=.17in}
\quotingsetup{rightmargin=.17in}

\makeatletter
\def\@cite#1#2{{\m@th\upshape\bfseries%
[{#1\if@tempswa{\m@th\upshape\mdseries, #2}\fi}]}}
\makeatother
\theoremstyle{plain}
\newtheorem{theorem}{Theorem}[section]

\newtheorem{proposition}[theorem]{Proposition}
\newtheorem{lemma}[theorem]{Lemma}
\theoremstyle{definition}
\newtheorem{definition}[theorem]{Definition}

\newtheorem{remark}[theorem]{Remark}

\newtheorem*{acknow}{Acknowledgements}
\theoremstyle{remark}

\mathtoolsset{centercolon}
  \newcommand{\A}{{\mathcal{A}}}
  \newcommand{\B}{{\mathcal{B}}}

  \newcommand{\F}{{\mathcal{F}}}
  
\renewcommand{\H}{{\mathcal{H}}}
  \newcommand{\I}{{\mathcal{I}}}
  \newcommand{\J}{{\mathcal{J}}}
  \newcommand{\K}{{\mathcal{K}}}
\renewcommand{\L}{{\mathcal{L}}}
  
  \newcommand{\N}{{\mathcal{N}}}
\renewcommand{\O}{{\mathcal{O}}}
\renewcommand{\P}{{\mathcal{P}}}
  \newcommand{\Q}{{\mathcal{Q}}}
  
\renewcommand{\S}{{\mathcal{S}}}
  \newcommand{\T}{{\mathcal{T}}}

\newcommand{\eps}{\varepsilon}
\def\al{\alpha}
\def\be{\beta}
\def\ga{\gamma}
\def\de{\delta}
\def\ze{\zeta}

\def\ka{\kappa}
\def\la{\lambda}
\def\La{\Lambda}
\def\om{\omega}

\def\si{\sigma}

\newcommand\vphi{\varphi}

\newcommand{\bC}{\mathbb{C}}

\newcommand{\bF}{\mathbb{F}}
\newcommand{\bN}{\mathbb{N}}
\newcommand{\bT}{\mathbb{T}}
\newcommand{\bZ}{\mathbb{Z}}
\newcommand{\bR}{\mathbb{R}}

\newcommand{\fI}{{\mathfrak{I}}}

\newcommand{\Bi}{{\mathbf{i}}}

\newcommand{\FOR}{\text{ for }}

\newcommand{\foral}{\text{ for all }}
\newcommand{\qand}{\quad\text{and}\quad}
\newcommand{\qif}{\quad\text{if}\quad}
\newcommand{\qiff}{\quad\text{if and only if}\quad}
\newcommand{\qfor}{\quad\text{for}\; }
\newcommand{\qforal}{\quad\text{for all}\quad}

\newcommand{\ca}{\mathrm{C}^*}

\newcommand{\ol}{\overline}
\newcommand{\wt}{\widetilde}

\newcommand{\ad}{\operatorname{ad}}

\newcommand{\Aut}{\operatorname{Aut}}

\newcommand{\Avt}{\operatorname{AVT}}

\newcommand{\End}{\operatorname{End}}

\newcommand{\Eq}{\operatorname{E}}
\newcommand{\fty}{\operatorname{fin}}

\newcommand{\GEq}{\operatorname{G-E}}

\newcommand{\id}{{\operatorname{id}}}
\newcommand{\im}{{\operatorname{Im}}}

\newcommand{\mt}{\emptyset}

\newcommand{\spn}{\operatorname{span}}

\newcommand{\sumoplus}{\operatornamewithlimits{\sum \strut^\oplus}}
\newcommand{\supp}{\operatorname{supp}}
\newcommand{\Tr}{\operatorname{T}}

\newcommand{\sca}[1]{\langle#1\rangle} 
\newcommand{\nor}[1]{\Vert #1\Vert} 
\newcommand{\bo}[1]{\mathbf{#1}} 
\newcommand{\un}[1]{{\underline{#1}}} 
\newcommand{\umu}{\underline\mu}
\newcommand{\unu}{\underline\nu}

\newcommand{\fdn}{(\mathbb{F}_+^d)^N}

\addtocontents{toc}{\protect\setcounter{tocdepth}{1}}

\begin{document}

\title[Equilibrium states and entropy theory for Nica-Pimsner algebras]{Equilibrium states and entropy theory for \\ Nica-Pimsner algebras}

\author[E.T.A. Kakariadis]{Evgenios T.A. Kakariadis}
\address{School of Mathematics, Statistics and Physics\\ Newcastle University\\ Newcastle upon Tyne\\ NE1 7RU\\ UK}
\email{evgenios.kakariadis@ncl.ac.uk}

\thanks{2010 {\it  Mathematics Subject Classification.} 46L30, 46L55, 46L08, 58B34.}

\thanks{{\it Key words and phrases:} equilibrium states, product systems, Nica-Pimsner algebras.}

\begin{abstract}
We study the equilibrium simplex of Nica-Pimsner algebras arising from product systems of finite rank on the free abelian semigroup.
First we show that every equilibrium state has a convex decomposition into parts parametrized by ideals on the unit hypercube.
Secondly we associate every gauge-invariant part to a sub-simplex of tracial states of the diagonal algebra.
We show how this parametrization lifts to the full equilibrium simplices of non-infinite type.

The finite rank entails an entropy theory for identifying the two critical inverse temperatures: (a) the least upper bound for existence of non finite-type equilibrium states, and (b) the least positive inverse temperature below which there are no equilibrium states at all.
We show that the first one can be at most the strong entropy of the product system whereas the second is the infimum of the tracial entropies (modulo negative values).
Thus phase transitions can happen only in-between these two critical points and possibly at zero temperature.
\end{abstract}

\maketitle

\section{Introduction}

Kubo-Martin-Schwinger states encapture the properties of a quantum system at thermal equilibrium.
Taking motivation from the grand canonical form of Gibbs states, by now there is a well established theory of KMS-states for C*-dynamical systems over $\bR$.
Equilibrium states form an effective tool for $\bR$-equivariant isomorphisms and the community has focused on their parametrization.
It has become evident that they are tractable in particular for algebras arising from Fock type constructions.
Complete results have been established for $\bZ_+$-systems and algebras of semigroups of particular type, while steps have been taken forward in specific multivariable contexts.
In these cases the goal is to attain a full decomposition and parametrization of the equilibrium $\Eq$-simplex.
Surprisingly however much less is known for $\bZ_+^N$, even when finite unit decompositions are available.
The overarching aim of this paper is to tackle this class.
Namely, we establish the multivariable Wold decomposition and parametrization combined with (several notions of) entropy for Nica-Pimsner algebras of finite rank $\bZ_+^N$-product systems.

\subsection{Background}

In the early 2000s Laca-Neshveyev \cite{LacNes04} identified the scaling property that characterizes the $\Eq$-simplex for a large class of C*-algebras, namely the Pimsner algebras.
Introduced in \cite{Pim97} and finessed by Katsura \cite{Kat04}, Pimsner algebras are naturally constructed from a single Hilbert bimodule $X$ over a C*-algebra $A$, better known as a C*-correspondence.
They encompass a great variety of constructs, e.g., dynamical systems, transfer operators, graphs, and there are two main variants.
The Toeplitz-Pimsner algebra $\T(X)$ is generated by the left creation operators, while the Cuntz-Pimsner algebra $\O(X)$ is a quotient by an appropriate sub-ideal of the compacts $\K(\F X)$, the smallest possible so that $\O(X)$ attains an isometric copy of $X$ and $A$.
One of the breakthrough findings in \cite{LacNes04} is that the simplex of equilibrium states  $\Eq_\be(\T(X))$ at inverse temperature $\be$ has a rather rich structure that stabilizes after a critical temperature $\be_c$.
Moreover, every equilibrium state of $\T(X)$ has a unique convex decomposition in a finite part (arising from $\K(\F X))$ and in an infinite part (inducing on $\T(X)/ \K(\F X)$).
The latter quotient is not always the anticipated $\O(X)$, but they do coincide when $X$ is regular.

Following the work of Bost-Connes \cite{BosCon95}, Laca-Raeburn \cite{LacRae10} applied the framework of \cite{LacNes04} to study C*-algebras of the $(ax+b)$-semigroup.
In the process Laca-Raeburn \cite{LacRae10} formalized an algorithm for parametrizing equilibrium states of finite type that has been widely applicable.
A great volume of subsequent works for specific one-variable constructs emerged from \cite{LacRae10}.
One of the common characteristics is the identification of two critical inverse temperatures $\be_c'$ and $\be_c$: there are no equlibrium states below $\be_c'$ while it is just the finite part that survives above $\be_c$.
Hence phase transitions can only happen in-between $\be_c'$ and $\be_c$.
A second type of phase transition may occur at $\be = \infty$ between ground states and KMS${}_\infty$-states.

Although the C*-algebra $\T(\bN \rtimes \bN^\times)$ of \cite{LacRae10} does not fit in the Pimsner class, it accommodates several of the arguments of \cite{LacNes04}.
As shown by Brownlowe-an Huef-Laca-Raeburn \cite{BHLR12} it is a Toeplitz-Nica-Pimsner algebra in the sense of Fowler \cite{Fow02}.
In this case the product system $X = \{X_p\}_{p \in P}$ is a family of C*-correspondences (along with multiplication rules) over a quasi-lattice group $(G,P)$.
In addition to Fowler's Toeplitz-Nica-Pimsner algebra $\N\T(X)$ of the Fock representation, Carlsen-Larsen-Sims-Vittadello \cite{CLSV11} proved the existence of a Cuntz-Nica-Pimsner algebra $\N\O(X)$ as co-universal with respect to the Gauge-Invariant-Uniqueness-Theorem, for rather general pairs $(G,P)$.

Nica-Pimsner algebras encode in turn a variety of semigroup transformations.
Consequently the community envisaged a program that builds on \cite{LacNes04, LacRae10} and aspires to understand the $\Eq$-structure for Nica-Pimsner algebras and their quotients.
There is a number of worked out cases that shapes the framework in the following sense:

\vspace{2pt}

(Q.1) What is a canonical Wold decomposition for $(G,P)$?

\vspace{2pt}

(Q.2) How does it characterize the types of the equilibrium states?

\vspace{2pt}

(Q.3) How does the product system data affect the critical temperatures?

\vspace{2pt}
\noindent
Results in this direction can then be used for rigidity or classification purposes.
For example equilibrium states have been used recently for reconstruction of graphs \cite{CM12}.

Let us review some of the obtained parametrizations.
Brownlowe-an Huef-Laca-Raeburn \cite{BHLR12} studied the $\Eq$-structure for quotients of $\T(\bN \rtimes \bN^\times)$ of \cite{LacRae10} in conjuction with previous work of Larsen \cite{Lar10} on multivariable transfer operators.
These constructs of product systems attain orthonormal bases and Hong-Larsen-Szymanski \cite{HLS12a, HLS12b} also studied equilibrium states in such generality.
Under some extra conditions they identified a finite-type only structure at suitably high inverse temperatures.
However the critical points $\be_c'$ and $\be_c$, or states of non-finite type are not computed in \cite{HLS12b}.
Afsar-Brownlowe-Larsen-Stammeier \cite{ALS18} recently extended this by identifying the equilibrium simplex for the Toeplitz-type C*-algebras of right LCM semigroups, with a description for the critical interval.
In the particular case where the product system comes from a $\bZ_+^N$-dynamical system it has been shown in \cite{Kak14} that $\be_c' = \be_c = \log N$ and the states are either finite or infinite.
Moving beyond orthonormal bases, Afsar-an Huef-Raeburn \cite{AHR18} considered product systems arising from local homeomorphisms.
An upper bound for $\be_c$ is given but it is not shown if it is optimal.
Considerably more has been achieved for higher-rank graphs that admit orthogonal but not orthonormal bases.
If the graph is irreducible then it is shown in \cite{HLRS14} that $\be_c'$ and $\be_c$ coincide with the logarithm of the common Perron-Frobenius eigenvalue.
Fletcher-an Huef-Raeburn \cite{FHR18} later formalized an algorithm for computing the $\Eq$-structure of their Toeplitz algebras.
However the lowest inverse temperature $\be_c'$ is not computed.
Very recently Christensen \cite{Chr18} answered questions (Q.1) and (Q.2) and provided the full parametrization for higher rank graphs without any of the assumptions of \cite{FHR18}.
His independent approach fits exactly the framework we have been developing for the current paper.

There are two additional works of the author that inform our direction, in particular with respect to (Q.3).
In \cite{Kak17} we revisited the description of \cite{LacNes04} for finite rank C*-corres\-pondences.
By using several notions of entropy inspired by Pinzari-Watatani-Yonetani \cite{PWY00} we identified the critical values through the structural data of $X$.
Moreover we identified the $\Eq$-structure not just for $\T(X)$ but also for all relative Cuntz-Toeplitz algebras, including a full parametrization for $\O(X)$ itself (without assuming injectivity of $X$).
At the same time finite rank product systems are strong compactly aligned product systems as in \cite{DK18}.
We can thus use the analysis of \cite{DK18} for the realization of ideals of generalized compacts and of $\N\O(X)$.

\subsection{Main results}

The class of all product systems and quasi-lattices is extremely vast to hope to treat all cases in considerable depth in one stroke.
In this paper we wish to answer Questions (1)--(3) for finite rank product systems over $\bZ_+^N$, a class that encompasses a number of cases, e.g., \cite{AHR18, Chr18, FHR18, HKR17, HLRS14, HLRS15, Kak14, Lar10}.

The reader is advised to consult Section \ref{Ss:not} for the $\bZ_+^N$-notation we use.
Every $X_{\Bi}$ is of finite rank, i.e., it admits a family $\{x_{i,j} \mid j = 1, \dots, d_i\}$ of vectors in the unit ball such that
\begin{equation}
\xi_{\Bi} = \sum_{j=1}^{d_i} x_{i,j} \sca{x_{i,j}, \xi_{\Bi}} \foral \xi_{\Bi} \in X_\Bi.
\end{equation}
Thus so does every $X_{\un{n}}$ for the family $\{x_{\umu} \mid \ell(\umu) = \un{n}\}$, or equivalently $\K X_{\un{n}} = \L X_{\un{n}}$ for all $\un{n} \in \bZ_+^N$.
Note that we do not assume orthogonality of the $x_{i,j}$ and the decomposition may not be unique.
Such families are sometimes referred to as Parseval frames, e.g., \cite{HLS12b}, but here we will term $x = \{x_{i,j} \mid j=1, \dots, d_i, i=1, \dots, N\}$ as \emph{a unit decomposition of $X$}.

The Toeplitz-Nica-Pimsner algebra is the C*-algebra generated by the Fock left creation operators $\{\pi(a), t(\xi_{\un{n}}) \mid a \in A, \xi_{\un{n}} \in X_{\un{n}}, \un {n} \in \bZ_+^N\}$ on the full Fock space $\F X = \sum^{\oplus}_\un{m} X_{\un{m}}$.
The rotation along the $\bT^N$-gauge action corresponds to the multivariable Quantum Harmonic Oscillator as with the one-dimensional case, cf., \cite{Kak17}.
By passing to the unitization we may assume that $X$ is unital.
Since $X$ is of finite type, the projections
\begin{equation}
P_\Bi := \sum_{j=1}^{d_i} t(x_{i,j}) t(x_{i,j})^*
\qand
Q_F := \prod_{i \in F} (1 - P_\Bi)
\end{equation}
are in $\N\T(X)$ and commute with $\pi(A)$.
Suppose that $\{I_F \mid \mt \neq F \subseteq \{1, \dots, N\}\}$ is a lattice of $\perp$-invariant ideals in the sense that $I_F \subseteq I_{F'}$ when $F \subseteq F' \subseteq \{1, \dots, N\}$ and
\begin{equation}
\sca{X_{\un{n}}, I_F X_{\un{n}}} \subseteq I_F \foral \un{n} \perp F.
\end{equation}
Then we define $\N\O(I, X)$ be the quotient of $\N\T(X)$ by the ideal
\begin{equation}
\K_I := \sca{\pi(a) Q_F \mid a \in I_F, \mt \neq F \subseteq \{1, \dots, N\}}.
\end{equation}
By \cite{DK18} the Cuntz-Nica-Pimsner algebra $\N\O(X)$ corresponds to $\N\O(\I, X)$ for
\begin{equation}
\I_F := \{a \in \J_F \mid \sca{X_{\un{n}}, \J_F X_{\un{n}}} \subseteq \J_F \foral \un{n} \perp F\}
\; \text{ with } \;
\J_F := (\bigcap_{i \in F} \ker \phi_\Bi)^\perp.
\end{equation}
The Gauge-Invariant-Uniqueness-Theorem \cite{CLSV11} asserts that $\N\O(X)$ is the smallest Nica-Pimsner algebra that attains a faithful copy of $A$. 
In particular we define
\begin{equation}
\N\O(F, A, X) := \N\T(X) / \sca{Q_\Bi \mid i \in F}
\qand
\N\O(A,X) := \N\T(X) / \sca{Q_{\bo{1}}, \dots, Q_{\bo{N}}}.
\end{equation}
It follows that $\N\O(A,X)$ is $\N\O(X)$ when every $X_{\Bi}$ is injective.

For fixed $\be > 0$ the $F$-parts of the Wold decomposition of an equilibrium state relate to the $F$-projections
\begin{equation}
Q_F^{\un{n}} := \sum_{\ell(\umu) = \un{n}} t(x_{\umu}) Q_F t(x_{\umu})^*,
\end{equation}
so that $Q_{\{1, \dots, N\}}^{\un{n}} \equiv p_{\un{n}} \colon \F X \to X_{\un{n}}$.
For every $F \subseteq \{1, \dots, N\}$ we define
\begin{equation}
\Eq_\be^F(\N\T(X)) := \{\vphi \in \Eq_\be(\N\T(X)) \mid \sum_{\un{n} \in F} \vphi(Q_{F}^{\un{n}}) = 1 \text{ and } \vphi(Q_\Bi) = 0 \foral i \notin F\},
\end{equation}
with the understanding that for $F = \{1, \dots, N\}$ we write
\begin{equation}
\Eq_\be^{\{1, \dots, N\}}(\N\T(X)) \equiv 
\Eq_\be^{\fty}(\N\T(X)) := \{\vphi \in \Eq_\be(\N\T(X)) \mid \sum_{\un{n} \in \bZ_+^N} \vphi(p_{\un{n}}) = 1 \},
\end{equation}
and for $F = \mt$ we write
\begin{equation}
\Eq_\be^{\mt}(\N\T(X))
\equiv
\Eq_\be^\infty(\N\T(X)) := \{\vphi \in \Eq_\be(\N\T(X)) \mid \vphi(Q_\Bi) = 0 \foral i =1, \dots, N \}.
\end{equation}
We also consider the gauge invariant equilibrium states and we write
\begin{equation}
\GEq_\be^F(\N\T(X)) := \{\vphi \in \Eq_\be^F(\N\T(X)) \mid \vphi = \vphi E \},
\end{equation}
where $E \colon \N\T(X) \to \N\T(X)^\ga$ is the conditional expectation of the gauge action.
In Proposition \ref{P:gi} we show that any equilibrium state defines a gauge-invariant equilibrium state.
In particular we show that the finite-type equilibrium states have to be gauge-invariant.

Equivariance of the $Q_F^{\un{n}}$ allows to define $\Eq_\be^F(\N\O(I, X))$ in a similar manner.
Due to the KMS condition the infinite-type states are exactly the equilibrium states of $\N\O(A,X)$ and the finite-type equilibrium states correspond to unique extensions of states on $\K(\F X)$.
The $\Eq_\be^F$-states correspond to states of the ideal $\sca{Q_F}$ that annihilate $\sca{Q_{\Bi} \mid i \notin F}$.
Our first main result is that the $F$-parts are the building blocks of the $\Eq_\be$-simplex.

\begin{proof}[\bf Theorem A] 
\textit{(Theorem \ref{T:con}) Let $X$ be a product system of finite rank over $A$ and let $\be > 0$.
Then every $\vphi \in \Eq_\be(\N\T(X))$ admits a unique decomposition in a convex combination of $\vphi_F \in \Eq_\be^F(\N\T(X))$.
Moreover the $F$-part $\vphi_F$ is non-trivial if and only if $\vphi(Q_F) \neq 0$.
The same decomposition passes to the gauge-invariant equilibrium states and also to relative Cuntz-Nica-Pimsner algebras.}
\end{proof}

Our next step is to parametrize every $\GEq_\be^F(\N\T(X))$ by specific subsets of tracial states of $A$.
For every $\mt \neq F \subseteq \{1, \dots, N\}$ and $\tau \in \Tr(A)$ we define
\begin{equation}
c_{\tau, \be}^F := \sum \{ e^{- |\umu| \be} \tau(\sca{x_{\umu}, x_{\umu}}) \mid \ell(\umu) \in F \}.
\end{equation}
Then the associated $F$-set of tracial states of $A$ is given by
\begin{equation}
\Tr_{\be}^F(A)
:=
\{ \tau \in \Tr(A) \mid c_{\tau,\be}^F < \infty
\text{ and }
e^{\be}\tau(a) = \sum_{j=1}^{d_i} \sca{x_{i,j}, a x_{i,j}} \foral i \notin F\}.
\end{equation}
In particular for $F = \{1, \dots, N\}$ we write
\begin{equation}
\Tr_\be^{\fty}(A) := \{\tau \in \Tr(A) \mid c_{\tau, \be}^{\{1, \dots, N\}} = \sum_{\ell(\umu) \in \bZ_+^N} e^{- |\umu| \be} \tau(\sca{x_{\umu}, x_{\umu}}) < \infty \}.
\end{equation}
The case of $F = \mt$ is captured in the averaging traces, namely
\begin{equation}
\Avt_\be(A) := \{\tau \in \Tr(A) \mid e^{\be}\tau(a) = \sum_{j=1}^{d_i} \tau(\sca{x_{i,j}, a x_{i,j}}) \foral i = 1, \dots, N\}.
\end{equation}
Nevertheless these parts contain more states than what we want and we need to restrict further to traces that annihilate the kernel
\begin{equation}
\fI_{F^c} := \ker\{ A \to \N\O(F^c, A, X) \} = \{a \in A \mid \lim_{k} \vphi_{k \cdot \un{1}_{F^c}}(a) = 0 \}.
\end{equation}
When $X_\Bi$ is injective for every $i \in F^c$ then $\fI_{F^c} = (0)$. 
Passing further to states that annihilate $I_{F}$ we obtain the full $\Eq$-structure for $\N\O(I, X)$ at inverse temperature $\be >0$.
By setting $I_F = (0)$ for every $F$ gives the $\Eq$-structure of $\N\T(X)$.
Setting $I_F = \I_F$ for every $F$ we obtain the $\Eq$-structure for the important quotient $\N\O(X)$.

\begin{proof}[\bf Theorem B]
\textit{(Theorem \ref{T:para} and Theorem \ref{T:para 2}) 
Let $X$ be a product system of finite rank over $A$ and $\be > 0$.
Let $\{I_F \mid \mt \neq F \subseteq \{1, \dots, N\}\}$ be a lattice of $\perp$-invariant ideals of $A$.
Then: \\
$(1)$ The infinite-type simplex $\GEq_\be^{\infty}(\N\O(I,X))$ is weak*-homeomorphic onto 
\[
\{\tau \in \Avt_\be(A) \mid \tau|_{\fI_{\{1, \dots, N\}}} = 0\}.
\]
$(2)$ For every $F \neq \mt$ there is a continuous bijection from $\GEq_\be^{F}(\N\O(I, X))$ onto 
\[
\{ \tau \in \Tr_\be^F(A) \mid \tau|_{\fI_{F^c} + I_{F}} = 0 \}.
\]
$(3)$ These parametrizations respect convex combinations and thus the extreme points of the simplices.
}
\end{proof}

If $\Tr_\be^F(A)$ or $\GEq_\be^F(\N\O(I, X))$ is weak*-closed then the map in item (3) is a weak*-homeomorphism.
The proof is constructive and provides explicit formulas for these parametri\-zations.
To this end we use an induced product system $Z = \{Z_{\un{n}} \}_{\un{n} \in F}$ over $\B_{F^c}$ in the Fock space representation that is supported on $F$ and is given by
\begin{equation}\label{eq:z}
\B_{F^c} = \ol{\spn}\{t(X_{\un{k}}) t(X_{\un{w}})^* \mid \un{k}, \un{w} \perp F\}
\qand
Z_{\un{n}} = \ol{t(X_{\un{n}}) \B_{F^c}}
\qfor
\un{n} \in F.
\end{equation}
First we use a direct limit method to lift a $\tau \in \Tr_\be^F(A)$ to a KMS-state $\wt\tau$ on $\B_{F^c}$.
Then we use the finite-type statistical approximation on the GNS-representation of $\wt\tau$ to construct an equilibrium state on $\N\T(X)$ from $\sca{Q_F}$.
Finally we use the $\perp$-invariance of the $I_{F}$ to ensure this two-step process factorizes through $\N\O(I, X)$.
Gauge-invariance is used only to move from $\tau$ to $\wt\tau$.
Nevertheless with that given we can apply the same method to extend the parametrization to the entire $\Eq_\be^F(\N\T(X))$.

\begin{proof}[\bf Theorem C]
\textit{(Theorem \ref{T:para 3})
Let $X$ be a product system of finite rank over $A$ and $\be > 0$.
Then $\vphi \in \Eq_\be^{\infty}(\N\T(X))$ if and only if
\[
\vphi(f) = e^{-\be} \sum_{j=1}^{d_i} \vphi(t(x_{i,j})^* f t(x_{i,j})) 
\foral
f \in \N\T(X),
i=1, \dots, N.
\]
For $\mt \neq F \subseteq \{1, \dots, N\}$ the parametrization of Theorem \ref{T:para} lifts to a parametrization
\[
\{\wt\tau \in \Eq_\be(\B_{F^c}) \mid \wt\tau\pi \in \Tr_\be^F(A), \wt\tau|_{\fI_{F^c}'} = 0\}
\to 
\Eq_\be^F(\N\T(X))
\]
where
\[
\fI_{F^c}' := \ker\{ \B_{F^c} \to \N\O(F^c, A,X)\}.
\]
Likewise the parametrization of the relative Cuntz-Nica-Pimsner algebras from Theorem \ref{T:para 2} lifts to
\[
\{\wt\tau \in \Eq_\be(\B_{F^c}) \mid \wt\tau\pi \in \Tr_\be^F(A), \wt\tau|_{\fI_{F^c}' + I_F'} = 0\}
\to 
\Eq_\be^F(\N\O(I, X))
\]
where
\[
I_{F}' := \ol{\spn}\{t(X_{\un{m}}) \pi(I_F) t(X_{\un{w}})^* \mid \un{m}, \un{w} \perp F\}. \qedhere
\]
}
\end{proof}

In all results we consider the rotations to have the same (constant) rate.
In Remark \ref{R:ngi} we show that our methods extend to cover different weights as well.
When all weights are non-zero then this is just a scaling argument.
When some weights are zero then we just need to apply the general results to the Toeplitz-Nica-Pimsner algebra on the remaining fibers\footnote{\ It must be noted that this process is not class-preserving.
Nevertheless, Christensen \cite{Chr18} remarkably tackles the problem by remaining within the same class of higher rank graphs.}.

Next we turn our focus to the inverse temperatures $\be \geq 0$ for which equilibrium states exist.
The existence of a unit decomposition entails a notion of entropy; the interested reader may consult the monograph of Neshveyev-St\o rmer \cite{NesSto06} for a detailed discussion on the subject.
Such an idea is developed by Pinzari-Watatani-Yonetani \cite{PWY00} for Cuntz-Pimsner algebras of a single C*-correspondence.
Assuming the existence of both a left and a right unit decomposition they identify the critical minimal and maximal inverse temperatures of Cuntz-Pimsner algebras.
This is achieved by considering a Perron-Frobenius type theory for the induced completely positive maps which is applied to Cuntz-Krieger algebras.
Their notion of entropy is then compared to the Brown-Voiculescu topological entropy \cite{Voi95, Bro99} of the arising subshifts, which in turn finds its roots in the non-commutative entropy of Connes-St\o rmer \cite{CS75}.
Recently it has been discovered by an Huef-Laca-Raeburn-Sims \cite{HLRS15b} that not just the entropy of the adjacency matrix but also the entropies of the sink subgraphs are necessary to identify the phase transitions for the Toeplitz-Cuntz-Krieger algebras, and therefore of the Cuntz-Krieger algebras by descending appropriately; see also \cite{Kak17}.
The key idea is that the convergence of the statistical approximation over a state depends on the irreducible components it visits forwards.
This route also removes the assumption of a right unit decomposition from \cite{PWY00}.

In order to define the multivariable analogue in our context we need to consider all possible $F$-statistical approximations for the $\Eq_\be^F$-parts, and then discover their relations.
For every $\tau \in \Tr(A)$ we define the \emph{tracial entropy}
\begin{equation}
h_X^\tau := \limsup_k \frac{1}{k} \log \bigg[ \sum_{|\umu| = k} \tau(\sca{x_{\umu}, x_{\umu}}) \bigg],
\end{equation}
and for $F \subseteq \{1, \dots, N\}$ we define the \emph{$F$-tracial entropy}
\begin{equation}
h_X^{\tau, F} := \limsup_k \frac{1}{k} \log \bigg[ \sum_{|\umu| = k, \ell(\umu) \in F} \tau(\sca{x_{\umu}, x_{\umu}}) \bigg].
\end{equation}
We define the \emph{entropy} of a unit decomposition $x = \{x_{i,j} \mid j = 1, \dots, d_i, i=1, \dots, N\}$ by
\begin{equation}
h_X^x := \limsup_k \frac{1}{k} \log \| \sum_{|\umu| = k} \sca{x_{\umu}, x_{\umu}}\|_{A},
\end{equation}
and then the \emph{strong entropy} of $X$ by
\begin{equation}
h_X^s := \inf\{ h_X^x \mid \text{ $x$ is a unit decomposition of $X$}\}.
\end{equation}
Likewise for $F \subseteq \{1, \dots, N\}$ we define the \emph{$F$-entropy} of a unit decomposition $x$ by
\begin{equation}
h_X^{x, F} := \limsup_k \frac{1}{k} \log \| \sum_{|\umu| = k, \ell(\umu) \in F} \sca{x_{\umu}, x_{\umu}}\|_{A},
\end{equation}
and the \emph{$F$-strong entropy} of $X$ by
\begin{equation}
h_X^{s, F} := \inf\{ h_X^{x,F} \mid \text{ $x$ is a unit decomposition of $X$}\}.
\end{equation}
Moreover we define the \emph{entropy} of $X$ by
\begin{equation}
h_X := \inf\{\be >0 \mid \Eq_\be(\N\T(X)) \neq \mt \}.
\end{equation}
Due to our parametrization we are able to deduce $h_X$ from the tracial entropies.
This is quite pleasing as for the first time we are able to recognize the critical inverse temperatures for $\bZ_+^N$-product systems.

\begin{proof}[\bf Theorem D]
\textit{(Proposition \ref{P:entropy1}, Proposition \ref{P:entropy2}, Theorem \ref{T:entropy} and Proposition \ref{P:strong entropy}) 
Let $X$ be a product system of finite rank over $A$. \\
$(1)$ If $\tau \in \Tr(A)$ then
\[
h_X^\tau \leq h_X^s = \max \{ h_X^{s, i} \mid i =1, \dots, N \} \leq \max \{ \log d_i \mid i =1, \dots, N \}.
\]
$(2)$ If $\tau \in \Tr_\be^F(A)$ for some $F \subseteq \{1, \dots, N\}$ and $\be > 0$ then
\[
h_X^{\tau, F} \leq h_X^{\tau} \leq \be.
\]
$(3)$ The entropy of $X$ is the infimum of the tracial entropies modulo negative values, i.e., 
\[
h_X 
 = \max\{0, \inf \{h_X^{\tau} \mid \tau \in \Tr(A)\} \}.
\]
$(4)$ If in addition $h_X > 0$ then there exists a $\tau \in \Tr(A)$ such that $h_X = h_\tau$.\\
$(5)$ If $\be > h_X^{s,F}$ then $\Eq_\be^C(\N\T(X)) = \mt$ for all $C \subsetneq F$.
In particular, if $\be > h_X^s$ then $\Eq_\be(\N\T(X)) = \Eq_\be^{\fty}(\N\T(X)) \simeq \Tr(A)$.
}
\end{proof}

The strong entropy can be the greatest lower bound for the non-finite parts depending on the structure of the product system.
For example $h_X^s = \log \la$ for the Perron-Frobenius eigenvalue of a higher-rank graph \cite{HLRS14}.
Proposition \ref{P:strong entropy} also yields existence of KMS${}_\infty$-states as limits of finite-type states.
Thus another phase transition may occur at $\be = \infty$.

\begin{proof}[\bf Theorem E]
\emph{(Theorem \ref{T:ground})
Let $X$ be a product system of finite rank over $A$.
Then there exists an affine weak*-homeomoprhism $\Psi$ between the states $\tau \in \S(A)$ and the ground states of $\N\T(X)$ such that
\begin{equation*} 
\Psi_\tau(\pi(a)) = \tau(a) \foral a \in A
\qand
\Psi_\tau(t(\xi_{\un{n}}) t(\eta_{\un{m}})^*) = 0 \text{ when } \un{n} + \un{m} \neq \un{0}.
\end{equation*}
The restriction of $\Psi$ to the tracial states $\Tr(A)$ induces a weak*-homeomorphism onto the KMS${}_\infty$-states of $\N\T(X)$.}

\emph{If $\{I_F \mid \mt \neq F \subseteq \{1, \dots, N\}\}$ is a lattice of $\perp$-invariant ideals of $A$ then the corresponding weak*-homeomorphisms for $\N\O(I, X)$ arise by restricting on states that annihilate the ideal $I_{\{1, \dots, N\}}$.}
\end{proof}

We close with some examples to show how our theory covers some known results.
Apart from the $\bZ_+^N$-dynamics of \cite{Kak14} and irreducible higher rank graphs of \cite{HLRS14} we apply the entropy theory to product systems of multivariable factorial languages from \cite{DK18}.

\subsection{Organization of paper}

In Section \ref{S:pre} we fix notation and set up the terminology for Nica-Pimsner algebras of finite rank product systems.
In Section \ref{S:wold} we provide the convex decomposition of an equilibrium state on $\N\T(X)$ in its $F$-parts.
Then we proceed in Section \ref{S:para} to the parametrization of the $F$-simplices.
In Section \ref{S:rel} we provide the full decomposition/parametrization theorem for Nica-Pimsner algebras by using factorization through the canonical quotient map.
In Section \ref{S:entropy} we make the connection with the entropies: we show how they affect the $\Eq$-structure and we give the parametrization of the limit states.
The examples are presented in Section \ref{S:examples}.

\section{Preliminaries}\label{S:pre}

\subsection{Notation}\label{Ss:not}

The free generators of $\bZ^N$ for $N < \infty$ will be denoted by $\bo{1}, \dots, \bo{N}$.
We write
\[
|\un{n}| \equiv |\sum \{ n_i \Bi \mid i \in \{1, \dots, N\} \}| := \sum_{i =1}^N n_i
\]
for the \emph{length} of $\un{n} \in \bZ_+^N$.
We fix $\un{1} := (1, \dots, 1)$.
For $\mt \neq F \subseteq \{1, \dots, N\}$ and $\un{n} \in \bZ_+^N$ we write
\[
\un{n}_F := \sum_{i \in F} n_i \cdot \Bi.
\]
In this sense we write $\un{1}_F = \sum \{ \Bi \mid i \in F \}$.
We consider the usual lattice structure on $\bZ_+^N$.
We denote \emph{the support of $\un{n}$} by 
\[
\supp \un{n} := \{i \in \{1, \dots, N\} \mid n_i \neq 0\}
\]
and we write
\[
\un{n} \perp \un{m} \qiff \supp \un{n} \bigcap \supp \un{m} = \mt.
\]
Thus $\un{n} \perp F$ means that $\supp \un{n} \bigcap F = \mt$.
For simplicity we write
\[
\un{n} \in F \qiff \supp{\un{n}} \subseteq F.
\]
We will be making use of the alternating sums, i.e., for $F \subseteq \{1, \dots, N\}$ and $\{f_1, \dots, f_N\}$ commuting elements we have
\[
\prod_{i \in F} (1- f_i)
=
\sum \{(-1)^{|C|}  \prod_{i \in C} f_i \mid C \subseteq F \} = 0.
\]
We will consider multivariable words on different sets of symbols.
Let $\{d_1, \dots, d_N\}$ be a set of finite positive integers.
We write
\[
\umu = (\mu_1, \dots, \mu_N) \in \bF_+^{d_1} \times \cdots \times \bF_+^{d_N}
\]
for the tuple of $N$ words where each $\mu_i$ is a word over the symbol set $\{1, \dots, d_i\}$.
If $|\mu_i|$ denotes the length of each word then we define the multi-length $\ell(\umu)$ and the length $|\umu|$ by
\[
\ell(\umu) := (|\mu_1|, \dots, |\mu_N|) 
\qand 
|\umu| := |\ell(\umu)| = \sum_{i=1}^N |\mu_i|.
\]
We will be considering limits over $\bZ_+^N$ with the understanding of convergence over the directed family of finite sets in $\bZ_+^N$.
Therefore for a non-negative sequence $(a_{\un{n}})_{\un{n} \in \bZ_+^N}$ we have
\[
\lim_{\un{n} \in \bZ_+^N} a_{\un{n}}
=
\lim_{|\un{n}| = k \to \infty} a_{\un{n}}
=
\lim_{\un{n} \leq k \cdot \un{1}, k \to \infty} a_{\un{n}}.
\]

\subsection{C*-correspondences}

The reader should be well acquainted with the general theory of Hilbert modules and C*-correspondences.
For example, one may consult \cite{Kat04} and \cite{Lan95} which we follow for terminology.
Here we just wish to fix notation.

A \emph{C*-correspondence} $X$ over $A$ is a right Hilbert module over $A$ with a left action given by a $*$-homomorphism $\phi_X \colon A \to \L X$.
We write $\L X$ and $\K X$ for the adjointable operators and the compact operators of $X$, respectively.
We will write $a\xi$ for $\phi_X(a)\xi$ when it is clear from the context which left action we use.
The ``rank one compact operators'' $\zeta \mapsto \xi \sca{\eta, \zeta}$ are denoted by $\theta^X_{\xi, \eta}$.

For two C*-corresponden\-ces $X, Y$ over the same $A$ we write $X \otimes_A Y$ for the stabilized tensor product over $A$.
Moreover we say that $X$ is unitarily equivalent to $Y$ if there is a surjective adjointable $U \in \L(X,Y)$. 

A \emph{representation} $(\rho,v)$ of a C*-correspondence is a left module map that preserves the inner product.
Then $(\rho,v)$ is automatically a bimodule map.
Moreover there exists a $*$-homomorphism $\psi$ on $\K X$ such that $\psi(\theta^X_{\xi, \eta}) = v(\xi) v(\eta)^*$.
If $\rho$ is injective then so is $\psi$. 

\subsection{Product systems}

Fix a set $\{X_{\Bi} \mid i \in \{1, \dots, N\}\}$ of C*-correspondences over $A$, one for each generator of $\bZ_+^N$.
A \emph{product system $X$} is a family $\{X_{\un{n}} \mid \un{n} \in \bZ_+^N\}$ of C*-correspondences over $A$ such that 
\[
X_{\un{0}} = A \qand 
X_{\bo{i}_1}^{\otimes n_1} \otimes_A \cdots \otimes_A X_{\bo{i}_k}^{\otimes n_k} \simeq X_{\un{n}} \text{ whenever } \un{n} = \sum n_j \bo{i}_j \text{ and } n_1 \neq 0.
\]
We require $n_1 \neq 0$ so that these equivalences do not force non-degeneracy of the fibers.
Consequently $X$ comes with a family of product rules in the form of unitary equivalences
\[
u_{\un{n}, \un{m}} \colon X_{\un{n}} \otimes_A X_{\un{m}} \to X_{\un{n} + \un{m}}.
\]
We will suppress the use of the $u_{\un{n}, \un{m}}$ as much as possible by writing $\xi_{\un{n}} \xi_{\un{m}} \in X_{\un{n} + \un{m}}$ for the element $u_{\un{n}, \un{m}}(\xi_{\un{n}} \otimes \xi_{\un{m}})$.
Along with the system we have some canonical operations that respect these equivalences.
To this end we define the maps
\[
i_{\un{n}}^{\un{n} + \un{m}} \colon \L X_{\un{n}} \to \L X_{\un{n} + \un{m}}
\; \textup{ such that } \;
i_{\un{n}}^{\un{n} + \un{m}}(S) = u_{\un{n}, \un{m}}(S \otimes \id_{X_{\un{m}}}) u^{*}_{\un{n}, \un{m}}.
\]
It is clear that $i^{\un{n} + \un{m} + \un{r}}_{\un{n} + \un{m}} \, i^{\un{n} + \un{m}}_{\un{n}} = i^{\un{n} + \un{m} + \un{r}}_{\un{n}}$ and thus $i^{\un{n} + \un{m}}_{\un{n}}(\phi_{\un{n}}(a)) = \phi_{\un{n} + \un{m}}(a)$.
Following Fowler's work \cite{Fow02}, a product system is called \emph{compactly aligned} if it has the property:
\[
i_{\un{n}}^{\un{n} \vee \un{m}}(S) i_{\un{m}}^{\un{n} \vee \un{m}}(T) \in \K X_{\un{n} \vee \un{m}} \text{ whenever } S \in \K X_{\un{n}}, T \in \K X_{\un{m}}.
\]
If a compactly aligned product system satisfies $i_{\un{n}}^{\un{n} + \Bi}(\K X_{\un{n}}) \subseteq \K X_{\un{n} + \Bi}$ whenever $\un{n} \perp \Bi$ then it is called \emph{strong compactly aligned} \cite{DK18}.

\subsection{Finite rank} \label{Ss:finite rank}

A product system will be said to have finite rank $\{d_1, \dots, d_N\}$ if for every $i = 1, \dots, N$ there exists a family of vectors $\{x_{i,j} \mid j = 1, \dots, d_i\}$ such that
\[
1_{X_{\Bi}} = \sum_{j=1}^{d_i} \theta^{X_\Bi}_{x_{i,j}, x_{i,j}}.
\] 
We will say that a family $x = \{x_{i,j} \mid j=1, \dots, d_i, i=1, \dots, N\}$ is a unit decomposition for the C*-correspondence.
If $\mu_i = j_1 \cdots j_r$ is a word on the symbols $\{1, \dots, d_i\}$ then we write
\[
x_{i, \mu_i} = x_{i, j_1} \otimes \cdots \otimes x_{i, j_r}.
\]
We may extend this notation to the entire $\bZ_+^N$.
Let $\umu = (\mu_1, \dots, \mu_N)$ be a multivariable word such that each $\mu_i$ is a word on the symbols $\{1, \dots, d_i\}$.
If $\mu_{i_1}, \dots, \mu_{i_k} \neq \mt$ for $i_1 < \dots < i_k$ then we write
\[
x_{\umu} = x_{i_1, \mu_{i_1}} \cdots x_{i_k, \mu_{i_k}} \in X_{\ell(\umu)}.
\]
Then the family $\{x_{\umu} \mid \ell(\umu) = \un{n} \}$ is a unit decomposition for $X_{\un{n}}$.
Notice here that a re-ordering of the $x_{i, \mu_i}$ in $x_{\umu}$ gives another unit decomposition (of the same cardinality).
We have that $X$ is of finite rank if and only if $\K X_{\un{n}} = \L X_{\un{n}}$ for all $\un{n} \in \bZ_+^N$, and thus a product system of finite rank is automatically strong compactly aligned.

\subsection{Representations} \label{Ss:CNP-relations}

A \emph{Nica-covariant representation $(\rho, v)$} of a product system $X = \{X_{\un{n}} \mid \un{n} \in \bZ_+^N \}$ consists of a family of representations $(\rho, v_{\un{n}})$ of $X_{\un{n}}$ that satisfy the \emph{product rule}:
\[
v_{\un{n} + \un{m}}(\xi_{\un{n}} \xi_{\un{m}}) = v_{\un{n}}(\xi_{\un{n}}) v_{\un{m}}(\xi_{\un{m}})
\foral \xi_{\un{n}} \in X_{\un{n}}, \xi_{\un{m}} \in X_{\un{m}},
\]
and the \emph{Nica-covariance}:
\[
\psi_{\un{n}}(S) \psi_{\un{m}}(T) = \psi_{\un{n} \vee \un{m}} ( i_{\un{n}}^{\un{n} \vee \un{m}}(S) i_{\un{m}}^{\un{n} \vee \un{m}}(T)) \text{ whenever } S \in \K X_{\un{n}}, T \in \K X_{\un{m}}.
\]
We write $\rho \times v$ for the induced representation of a Nica-covariant pair $(\rho,v)$.
Henceforth we will suppress the use of the indices and write $v$ instead of $v_{\un{n}}$.

The \emph{Toeplitz-Nica-Pimsner} algebra $\N\T(X)$ is the universal C*-algebra generated by $A$ and $X$ with respect to the representations of $X$.
The Fock space provides an essential example of an injective Nica-covariant representation for $\N\T(X)$.
In short, let $\F(X) = \sumoplus \{ X_{\un{m}} \mid \un{m} \in \bZ_+^N\}$.
For $a \in A$ and $\xi_{\un{n}} \in X_{\un{n}}$ define
\[
\pi(a) \eta_{\un{m}} = \phi_{\un{m}}(a) \eta_{\un{m}}
\qand
t(\xi_{\un{n}}) \eta_{\un{m}} = \xi_{\un{n}} \eta_{\un{m}}
\qforal
\eta_{\un{m}} \in X_{\un{m}}.
\]
Then $(\pi, t)$ is Nica-covariant and it is called the \emph{Fock representation} of $X$ \cite{Fow02}.
By taking the compression at the $(\un{0}, \un{0})$-entry we see that $\pi$, and thus $t$, is injective.

Given a Nica-covariant representation $(\rho, v)$ and $\un{m}, \un{m}' \in \bZ_+^N$ we define the \emph{cores} of the representation $(\rho, v)$ by
\[
\B_{[\un{m}, \un{m} + \un{m}']} := \spn\{ \psi_{\un{n}}(k_{\un{n}}) \mid k_\un{n} \in \K X_{\un{n}}, \un{m} \leq \un{n} \leq \un{m} + \un{m}'\}.
\]
These $*$-algebras are closed in $\ca(\rho, v)$, e.g., \cite[Lemma 36]{CLSV11}.
It follows from the work of Fowler \cite[Proposition 5.10]{Fow02} that if $(\rho, v)$ is a Nica-covariant representation of a compactly aligned product system $X$ then
\begin{equation}\label{eq:Fow}
v(X_{\un{m}})^* v(X_{\un{n}}) \subseteq \ol{v(X_{\un{n} - \un{n} \wedge \un{m}}) v(X_{\un{m} - \un{n} \wedge \un{m}})^*}.
\end{equation}
Therefore the cores are $\perp$-stable in the sense that
\[
t_{\un{n}}(X_{\un{n}})^* \cdot \B_{[\un{m}, \un{m} + \un{m}']} \cdot t_{\un{n}}(X_{\un{n}}) \subseteq \B_{[\un{m}, \un{m} + \un{m}']} 
\foral \un{n} \perp \un{m} + \un{m}'.
\]
A Nica-covariant representation $(\rho, v)$ \emph{admits a gauge action} if there is a point-norm continuous family of $*$-automorphisms $\{\ga_{\un{z}}\}_{\un{z} \in \bT^N}$ such that
\[
\ga_{\un{z}}(\rho(a)) = \rho(a) \foral a \in A
\qand
\ga_{\un{z}}(v(\xi_{\un{n}})) = \un{z}^{\un{n}} \, v(\xi_{\un{n}}) 
\foral \xi_{\un{n}} \in X_{\un{n}}.
\]
In this case $\B_{[\un{0}, \infty]} = \ol{\cup_{\un{n}} \B_{[\un{0}, \un{n}]}}$ is the fixed point algebra of $\ca(\rho, v)$.
By universality $\N\T(X)$ (as well as any of the covariant C*-algebras we will define below) admits a gauge action.

The appropriate Cuntz-analogue of a C*-algebra should attain an isometric copy of $A$ and be co-universal with respect to a Gauge-Invariant-Uniqueness-Theorem.
Its existence has been established for compactly aligned product systems in \cite{CLSV11}.
The form of the representations makes use of the augmented Fock space construction of \cite{SY11}.
To get the Cuntz-Nica-Pimsner algebra for strong compactly aligned we require a little bit less work \cite{DK18}.
For a finite $\mt \neq F \subseteq \{1, \dots, N\}$ we form the ideal 
\begin{align*}
\J_F
& :=
(\bigcap_{i \in F} \ker\phi_{\Bi})^\perp \cap (\bigcap\{ \phi_{\un{n}}^{-1}( \K X_{\un{n}}) \mid \un{n} \leq \un{1} \} ),
\end{align*}
with the understanding that $\phi_{\un{0}} = \id_A$.
In particular when $X$ is strong compactly aligned, we have that
\[
\bigcap\{ \phi_{\un{n}}^{-1}( \K X_{\un{n}}) \mid \un{n} \leq \un{1} \} 
= 
\bigcap\{ \phi_{\Bi}^{-1}( \K X_{\Bi}) \mid i \in \{1, \dots, N\} \}. 
\]
Furthermore we define the ideal
\[
\I_F:= \{a \in \J_F \mid \sca{X_{\un{n}}, a X_{\un{n}}} \subseteq \J_F \foral \un{n} \perp F\}.
\]
Hence $\I_F$ is the biggest ideal in $\J_F$ that remains invariant under the ``action'' of $F^\perp$.
A Nica-covariant representation $(\rho, v)$ of $X$ is called \emph{Cuntz-Nica-Pimsner} (or a \emph{CNP-repre\-sentation}) if it satisfies
\[
\sum \{ (-1)^{|\un{n}|} \psi_{\un{n}}(\phi_{\un{n}}(a)) \mid \un{n} \leq \un{1}_F \} = 0 \foral a \in \I_F,
\]
where $\psi_{\un{0}}(\phi_{\un{0}}(a)) = \rho(a)$.
The \emph{Cuntz-Nica-Pimsner algebra} $\N\O(X)$ is the universal C*-algebra with respect to the CNP-representations.
One of the main results of \cite{DK18} is that this is the $*$-algebraic description of Cuntz-Nica-Pimsner algebra, i.e., this $\N\O(X)$ coincides with the one in \cite{CLSV11}.
From now on we pass to product systems of finite rank.
Notice that in this case $\phi_{\un{n}}(A) \subseteq \K X_{\un{n}}$ for all $\un{n} \in \bZ_+^N$.

\begin{definition}\label{D:rel}
Let $X$ be a product system of finite rank over $A$ and fix $\mt \neq F \subseteq \{1, \dots, N\}$.
A representation $(\rho, v)$ will be called \emph{$F$-covariant} if
\[
\rho(a) = \psi_\Bi(\phi_\Bi(a)) \foral a \in A \text{ and } i \in F.
\]
The \emph{$F$-Cuntz-Nica-Pimsner algebra} $\N\O(F,A,X)$ is the universal C*-algebra with respect to the $F$-covariant representations.
\end{definition}

It follows that $\N\O(F,A,X) = \N\T(X)/ \sca{\pi(a)Q_\Bi \mid a \in A, i \in F}$.
To this end we will write
\begin{equation*}
\fI_F := \ker \left\{ A \to \N\T(X)/\sca{\pi(a)Q_{\Bi} \mid a \in A, i \in F} \right\}.
\end{equation*}
Following the arguments of \cite[Corollary 5.1]{DK18} it can be shown that $(\rho, v)$ is $F$-covariant if and only if
\[
\rho(a) = \psi_{\un{n}}(\phi_{\un{n}}(a)) \foral a \in A \text{ and } \supp \un{n} \subseteq F.
\]
We write $\N\O(A, X)$ for the algebra related to the $\{1, \dots, N\}$-covariant pairs.
Sometimes $\N\O(A,X)$ appears as $\N\O(X)_{\textup{cov}}$ in the literature, but here we see it as an example in the bigger class of relative Nica-Pimsner algebras.
The latter is not always the Cuntz-analogue for product systems as it may not attain an isometric copy of $A$.
By the Gauge-Invariant-Uniqueness-Theorem for regular product systems of \cite{SY11} we get that $\N\O(A, X) = \N\O(X)$ if every $\phi_{\Bi}$ is injective.

\subsection{Kubo-Martin-Schwinger states}

Let $\si \colon \bR \to \Aut(\A)$ be an action on a C*-algebra $\A$.
Then there exists a norm-dense $\si$-invariant $*$-subalgebra $\A_{\textup{an}}$ of $\A$ such that for every $f \in \A_{\textup{an}}$ the function $\bR \ni r \mapsto \si_r(f) \in \A$ is analytically continued to an entire function $\bC \ni z \mapsto \si_z(f) \in \A$ \cite[Proposition 2.5.22]{BraRob87}.
If $\be > 0$, then a state $\vphi$ of $\A$ is called a \emph{$(\si,\be)$-KMS state} (or \emph{equilibrium state at $\be$}) if it satisfies the KMS-condition:
\begin{equation*}
\vphi(f g) = \vphi(g \si_{i\be}(f)) \, \text{ for all $f, g$ in a norm-dense $\si$-invariant $*$-subalgebra of $\A_{\text{an}}$}.
\end{equation*}
If $\be=0$ or if the action is trivial then a KMS-state is a tracial state on $\A$.
The KMS-condition follows as an equivalent for the existence of particular continuous functions \cite[Proposition 5.3.7]{BraRob97}. More precisely, a state $\vphi$ is an equilibrium state at $\be >0$ if and only if for any pair $f, g \in \A$ there exists a complex function $F_{f, g}$ that is analytic on $D = \{ z \in \bC \mid 0 < \im(z) < \be\}$ and continuous (hence bounded) on $\ol{D}$ such that
\[
F_{f, g}(r) = \vphi(f \si_r(g)) \text{ and } F_{f, g}(r + i \be) = \vphi(\si_r(g) f) \foral r \in \bR.
\]
A state $\vphi$ of $\A$ is called a \emph{KMS$_\infty$-state} if it is the weak*-limit of $(\si,\be)$-KMS states as $\be \uparrow \infty$.
A state $\vphi$ of a C*-algebra $\A$ is called a \emph{ground state} if the function $z \mapsto \vphi(f \si_{z}(g))$ is bounded on $\{z \in \bC \mid \text{Im}z >0\}$ for all $f, g$ inside a dense analytic subset of $\A$.
The distinction between ground states and KMS${}_\infty$-states is not apparent in \cite{BraRob97} and is coined in \cite{LacRae10}.

\begin{remark}
Henceforth we will focus on rotational actions in direct relation to Gibbs states.
Recall that in the classical case the $\bR$-action is induced by $r \mapsto e^{i r (H - \ka N)}$ where $H$ is the selfadjoint Hamiltonian, $N$ is the number operator and $\ka$ is the chemical potential.
When $H$ is the Quantum Harmonic Oscillator then $H - \ka N = h\om/2 + (h\om - \ka)N$ where $h\om/2$ is the zero point energy.
The Spectral Theorem then materializes the $\bR$-action as induced by the rotational number unitaries $r \mapsto u_{e^{irs}}$ for $s = h \om - \ka$ on the related Fock space, for example see \cite[Proposition 2.1]{Kak17}.
The same arguments apply for the multivariable version of the Quantum Harmonic Oscillator, i.e., the tensor product of number operators fiberwise.
We will be scaling $h \om - \ka = 1$; substituting $\be$ with $(h \om - \ka) \be$ unravels the general case.
\end{remark}

Let $\{\ga_{\un{z}}\}_{\un{z} \in \bT^n}$ be the gauge action on $\N\T(X)$ and define
\[
\si \colon \bR \to \Aut(\N\T(X)): r \mapsto \ga_{(\exp(i r), \dots, \exp(i r))}.
\]
The monomials of the form $t(\xi_{\un{n}}) t(\eta_{\un{m}})^*$ span a dense $*$-subalgebra of analytic elements of $\N\T(X)$ since the function
\[
\bR \to \N\T(X): r \mapsto \si_r(t(\xi_{\un{n}}) t(\eta_{\un{m}})^*) = e^{i |\un{n} - \un{m}| r} t(\xi_{\un{n}}) t(\eta_{\un{m}})^*
\]
is analytically extended to the entire function
\[
\bC \to \N\T(X) : z \mapsto e^{i |\un{n} - \un{m}| z} t(\xi_{\un{n}}) t(\eta_{\un{m}})^*.
\]
For $\be \in \bR$ the $(\si,\be)$-KMS condition is then equivalent to having
\begin{equation}\label{eq:kms}
\vphi(t(\xi_{\un{n}}) t(\eta_{\un{m}})^* \cdot t(\xi_{\un{k}}) t(\eta_{\un{w}})^*)
=
e^{-|\un{n} - \un{m}| \be} \vphi(t(\xi_{\un{k}}) t(\eta_{\un{w}})^* \cdot t(\xi_{\un{n}}) t(\eta_{\un{m}})^*)
\end{equation}
for all $\un{n}, \un{m}, \un{k}, \un{w} \in \bZ_+^N$.

\begin{definition}
Let $X$ be a compactly aligned product system.
For a fixed $\be >0$ we write $\Eq_\be(\N\T(X))$ for the $(\si,\be)$-KMS states on $\N\T(X)$ with respect to the action
\[
\si \colon \bR \to \Aut(\N\T(X)): r \mapsto \ga_{(\exp(i r), \dots, \exp(i r))}.
\]
If $E \colon \N \T(X) \to \N\T(X)^\ga$ is the conditional expectation of the gauge action, then we write
\[
\GEq_\be(\N\T(X)) := \{ \vphi \in \Eq_\be(\N\T(X)) \mid \vphi = \vphi E\}
\]
for the sub-simplex of the gauge-invariant equilibrium states.
\end{definition}

We note that it suffices to consider the unital case.
If there is at least one $X_{\Bi_{0}}$ that is not unital then consider the unitization $A^{1} = A + \bC$ of $A$ and extend the operations $\phi_{\un{n}}(1) \xi_{\un{n}} = \xi_{\un{n}} = \xi_{\un{n}} \cdot 1$.
Note here that $A^1 = A \oplus \bC$ if $A$ is unital but $\phi_{\un{n}}(1_A) \neq 1_{X_{\un{n}}}$.
We will write $X^{1} = \{X_{\un{n}}^{1}\}_{\un{n} \in \bZ_+^N}$.

\begin{proposition}\label{P:unital}
Let $X$ be a compactly aligned product system over $A$.
Then $\vphi$ is a $(\si,\be)$-KMS state for $\N\T(X^1)$ if and only if it restricts to a $(\si,\be)$-KMS state on $\N\T(X)$.
\end{proposition}

\begin{proof}
As $\si$ is the same action on both $\N\T(X^{1})$ and $\N\T(X)$ the proof follows by noting that 
$\N\T(X^{1})$ is the unitization of $\N\T(X)$.
See also \cite[Proposition 3.2]{Kak17} for more details from the $\bZ_+$-case.
\end{proof}

\begin{remark}
Due to Proposition \ref{P:unital} we assume that the product system is unital for our proofs and exposition.
However the results cover also non-unital product systems.
We will use the same symbol for the extension of a state from $\N\T(X)$ to $\N\T(X^{1})$ from Proposition \ref{P:unital}.
\end{remark}

We have an easier version of the KMS-condition that will be handy for our computations.

\begin{proposition}\label{P:char}
Let $X$ be a compactly aligned product system over $A$ and let $\be > 0$.
Then a positive functional $\vphi$ of $\N\T(X)$ satisfies the KMS-condition (\ref{eq:kms}) if and only if
\begin{equation}\label{eq:kms2}
\vphi(t(\xi_{\un{n}}) \cdot t(\xi_{\un{k}}) t(\eta_{\un{w}})^*) = e^{-|\un{n}| \be} \vphi(t(\xi_{\un{k}}) t(\eta_{\un{w}})^* \cdot t(\xi_{\un{n}}))
\foral
\un{n}, \un{k}, \un{w} \in \bZ_+^N.
\end{equation}
If in addition $\vphi = \vphi E$, then $\vphi$ satisfies the KMS-condition (\ref{eq:kms}) if and only if
\begin{equation}\label{eq:kms3}
\vphi(t(\xi_{\un{n}}) t(\eta_{\un{m}})^*)
= 
\de_{\un{n}, \un{m}} e^{-|\un{n}|\be} \vphi(t(\eta_{\un{m}})^* t(\xi_{\un{n}}))
\foral
\xi_{\un{n}} \in X_{\un{n}}, \eta_{\un{m}} \in X_{\un{m}}.
\end{equation}
Consequently, two gauge-invariant $(\si, \be)$-KMS states coincide if and only if they agree on $\pi(A)$.
\end{proposition}

\begin{proof}
For the first part, equation (\ref{eq:kms}) implies equation (\ref{eq:kms2}) for $\eta_{\un{m}} = 1_A$.
For the converse, we may use continuity to deduce that
\[
\vphi(t(\xi_{\un{n}}) f) = e^{-|\un{n}| \be} \vphi(f t(\xi_{\un{n}}))
\; \text{ for all } \;
f \in \ol{t(X_{\un{k}}) t(X_{\un{w}})^*}
\; \text{ and for all } \;
\un{n}, \un{k}, \un{w} \in \bZ_+^N.
\]
Since $\vphi$ is positive, by applying adjoints in equation (\ref{eq:kms2}) we get
\[
\vphi(t(\xi_{\un{n}})^* f)
=
e^{|\un{n}| \be} \vphi(f t(\xi_{\un{n}})^*)
\; \text{ for all } \;
f \in \ol{t(X_{\un{k}}) t(X_{\un{w}})^*}
\; \text{ and for all } \;
\un{n}, \un{k}, \un{w} \in \bZ_+^N.
\]
Now for $\un{m} \in \bZ_+^N$ we have that
\[
t(\eta_{\un{m}})^* t(\xi_{\un{k}}) t(\xi_{\un{w}})^*
\in
t(X_{\un{m}})^* t(X_{\un{k}}) t(X_{\un{w}})^*
\subseteq
\ol{t(X_{\un{k} - \un{k} \wedge \un{m}}) t(X_{\un{m} - \un{k} \wedge \un{m} + \un{w}})^*}.
\]
At the same time we have that
\[
t(\xi_{\un{k}}) t(\eta_{\un{w}})^* t(\xi_{\un{n}})
\in
t(X_{\un{k}}) t(X_{\un{w}})^* t(X_{\un{n}})
\subseteq
\ol{t(X_{\un{k} + \un{n} - \un{n} \wedge \un{w}}) t(X_{\un{w} - \un{n} \wedge \un{w}})^*}.
\]
Therefore we have that
\begin{align*}
\vphi(t(\xi_{\un{n}}) t(\eta_{\un{m}})^* \cdot t(\xi_{\un{k}}) t(\eta_{\un{w}})^*)
& =
\vphi(t(\xi_{\un{n}}) \cdot t(\eta_{\un{m}})^* t(\xi_{\un{k}}) t(\eta_{\un{w}})^*) \\
& =
e^{-|\un{n}| \be} \vphi(t(\eta_{\un{m}})^* \cdot t(\xi_{\un{k}}) t(\eta_{\un{w}})^* t(\xi_{\un{n}})) \\
& =
e^{-|\un{n}| \be} e^{|\un{m}| \be} \vphi(t(\xi_{\un{k}}) t(\eta_{\un{w}})^* t(\xi_{\un{n}}) t(\eta_{\un{m}})^*)
\end{align*}
which gives equation (\ref{eq:kms}).

Now suppose that in addition $\vphi = \vphi E$.
If $\vphi$ satisfies equation (\ref{eq:kms}) then it also implies equation (\ref{eq:kms3}).
Conversely we will show that if $\vphi$ satisfies equation (\ref{eq:kms3}) then it satisfies equation (\ref{eq:kms2}), and thus equation (\ref{eq:kms}).
To this end we consider two cases.
Suppose first that $\un{n} + \un{k} \neq \un{w}$.
Since $\vphi = \vphi E$ we thus have 
\[
\vphi(t(\xi_{\un{n}}) t(\xi_{\un{k}}) t(\eta_{\un{w}})^*)
=
\vphi E(t(\xi_{\un{n}}) t(\xi_{\un{k}}) t(\eta_{\un{w}})^*)
= 0.
\]
At the same time we have that
\[
t(\xi_{\un{k}}) t(\eta_{\un{w}})^* t(\xi_{\un{n}})
\in
\ol{t(X_{\un{k} + \un{n} - \un{n} \wedge \un{w}}) t(X_{\un{w} - \un{n} \wedge \un{w}})^*)}.
\]
Since $\un{k} + \un{n} - \un{n} \wedge \un{w} \neq \un{w} - \un{n} \wedge \un{w}$ we get that
\[
e^{-|\un{n}| \be} \vphi(t(\xi_{\un{k}}) t(\eta_{\un{w}})^* t(\xi_{\un{n}}))
=
e^{-|\un{n}| \be} \vphi E(t(\xi_{\un{k}}) t(\eta_{\un{w}})^* t(\xi_{\un{n}}))
=
0,
\]
so that equation (\ref{eq:kms2}) holds trivially.
Now suppose that $\un{n} + \un{k} = \un{w}$ so that $\un{w} \geq \un{n}$.
Then we have that $t(\eta_{\un{w}})^* t(\xi_{\un{n}}) \in t(X_{\un{w} - \un{n}})^* = t(X_{\un{k}})^*$ and thus
\begin{align*}
\vphi(t(\xi_{\un{n}}) \cdot t(\xi_{\un{k}}) t(\eta_{\un{w}})^*)
& =
\vphi(t(\xi_{\un{n}}) t(\xi_{\un{k}}) \cdot t(\eta_{\un{w}})^*)
 =
e^{-|\un{n} + \un{k}| \be} \vphi(t(\eta_{\un{w}})^* t(\xi_{\un{n}}) \cdot t(\xi_{\un{k}})) \\
& =
e^{-|\un{n} + \un{k}| \be} e^{|\un{k}| \be} \vphi(t(\xi_{\un{k}}) t(\eta_{\un{w}})^* t(\xi_{\un{n}}) ) 
 =
e^{-|\un{n}| \be} \vphi(t(\xi_{\un{k}}) t(\eta_{\un{w}})^* \cdot t(\xi_{\un{n}})),
\end{align*}
and the proof is complete.
\end{proof}

\section{Wold decomposition}\label{S:wold}

Suppose that $X$ is a unital product system with a unit decomposition $x = \{x_{i,j} \mid j=1, \dots, d_i, i=1, \dots, N\}$.
For every $i \in \{1, \dots, N\}$ we let the projections
\begin{equation*}
P_\Bi := \sum_{j=1}^{d_i} t(x_{i,j}) t(x_{i,j})^* \in \N\T(X).
\end{equation*}
For $F \subseteq \{1, \dots, N\}$ we define the projections
\begin{equation*}
P_F := \prod_{i \in F} P_{\Bi}
\qand 
Q_F := \prod_{i \in F} (1 - P_\Bi).
\end{equation*}
where $1 \equiv 1_{\F X}$.
It is clear that all these projections are in $\N\T(X)$.
For every $\un{n} \in \bZ_+^N$ let the projection $p_{\un{n}} \colon \F X \to X_{\un{n}}$.
It is straightforward to check that
\[
P_{F} = \sum\{ p_{\un{n}} \mid \un{n} \geq \un{1}_F\} = \sum_{\ell(\umu) = \un{1}_F} t(x_{\umu}) t(x_{\umu})^*,
\]
and therefore the $P_F$ and $Q_F$ all commute.
In particular they are in the center of the fixed point algebra $\N\T(X)^\ga$.
We will also write
\begin{equation*}
P_{k \cdot \Bi} := \sum\{ p_{\un{n}} \mid \un{n} \geq k \cdot \Bi\} = \sum_{\ell(\umu) = k \cdot \Bi} t(x_{\umu}) t(x_{\umu})^*.
\end{equation*}
For every $\un{n} \in \bZ_+^N$ we define the $\un{n}$-th ``translate'' of $Q_F \equiv Q_F^{\un{0}}$ by
\begin{equation*}
Q_{F}^{\un{n}} := \sum_{\ell(\umu) = \un{n}} t(x_{\umu}) Q_F t(x_{\umu})^*.
\end{equation*}
If in addition $\un{n} \in F$ then we can directly verify that
\[
Q_F^{\un{n}} \xi_{\un{w}} 
= \de_{\un{n}, \un{w}_F} \sum_{\ell(\umu) = \un{n}} x_{\umu} \sca{x_{\umu}, \xi_{\un{w}}}
= \de_{\un{n}, \un{w}_F} \xi_{\un{w}},
\]
and therefore
\begin{equation}\label{eq:p-r}
Q_F^{\un{n}} t(\xi_{\un{m}})
=
\begin{cases}
t(\xi_{\un{m}}) Q_F^{\un{n} - \un{m}_F} & \text{ if } \un{n} \geq \un{m}_F,\\
0 & \text{ otherwise}.
\end{cases}
\end{equation}

\begin{lemma}\label{L:proj}
With the aforementioned notation we have that
\[
Q_{F}^{\un{n}} \cdot Q_{C}^{\un{m}} = \de_{\un{n}, \un{m}_F} \cdot Q_{C}^{\un{m}} \foral \un{n} \in F, \un {m} \in C \text{ with } \mt \neq F \subseteq C.
\]
\end{lemma}

\begin{proof}
Let $\ell(\unu) = \un{n} \in F$ and $\ell(\umu) = \un{m} \in C$.
If $\un{n} - \un{n} \wedge \un{m} \neq \un{0}$ then it has an intersection with $F$, and thus with $C$, and so equation (\ref{eq:p-r}) yields
\[
Q_F t(x_{\unu})^* t(x_{\umu}) Q_C
\in
Q_F \ol{t(X_{\un{m} - \un{n} \wedge \un{m}}) \cdot t(X_{\un{n} - \un{n} \wedge \un{m}})^*} Q_C
=
(0).
\]
If $\un{n} = \un{n} \wedge \un{m}$, i.e., if $\un{n} \leq \un{m}$, but $\un{m}_F \neq \un{n}$ then $\un{m} - \un{n}$ has an intersection with $F$ and so again equation (\ref{eq:p-r}) yields
\[
Q_F t(x_{\unu})^* t(x_{\umu}) Q_C
\in
Q_F t(X_{\un{m} - \un{n}}) Q_C
=
(0).
\]
If $\un{n} = \un{m}_F$ then $t(x_{\unu})^* t(x_{\umu}) \in t(X_{\un{m} - \un{m}_F})$ and so equation (\ref{eq:p-r}) gives again
\[
Q_F t(x_{\unu})^* t(x_{\umu}) = t(x_{\unu})^* t(x_{\umu}) Q_F. 
\]
Therefore we get
\begin{align*}
Q_F^{\un{n}} \cdot Q_C^{\un{m}}
& =
\sum_{\ell(\unu) = \un{n}} \sum_{\ell(\umu) = \un{m}} t(x_{\unu}) Q_F t(x_{\unu})^* t(x_{\umu}) Q_C t(x_{\umu})^* \\
& =
\de_{\un{n}, \un{m}_F} \sum_{\ell(\unu) = \un{n}} \sum_{\ell(\umu) = \un{m}} t(x_{\unu}) t(x_{\unu})^* t(x_{\umu}) Q_F Q_C t(x_{\umu})^* \\
& =
\de_{\un{n}, \un{m}_F} \sum_{\ell(\umu'') = \un{m}_{F^c}} \sum_{\ell(\umu') = \un{m}_F} \sum_{\ell(\unu) = \un{m}_F} [t(x_{\unu}) t(x_{\unu})^* t(x_{\umu'})] t(x_{\umu''}) Q_F Q_C t(x_{\umu''})^* t(x_{\umu'})^* \\
& =
\de_{\un{n}, \un{m}_F} \sum_{\ell(\umu) = \un{m}} t(x_{\umu}) Q_C t(x_{\umu})^*
=
\de_{\un{n}, \un{m}_F} \cdot Q_C^{\un{m}}. \qedhere
\end{align*}
\end{proof}

\begin{definition}
Let $X$ be a product system of finite rank over $A$ and fix $\be > 0$.
For every $F \subseteq \{1, \dots, N\}$ we define
\begin{equation*}
\Eq_\be^F(\N\T(X)) := \{\vphi \in \Eq_\be(\N\T(X)) \mid \sum_{\un{n} \in F} \vphi(Q_{F}^{\un{n}}) = 1 \text{ and } \vphi(Q_\Bi) = 0 \foral i \notin F\},
\end{equation*}
with the understanding that for $F = \{1, \dots, N\}$ we write
\begin{equation*}
\Eq_\be^{\{1, \dots, N\}}(\N\T(X)) \equiv 
\Eq_\be^{\fty}(\N\T(X)) := \{\vphi \in \Eq_\be(\N\T(X)) \mid \sum_{\un{n} \in \bZ_+^N} \vphi(p_{\un{n}}) = 1 \},
\end{equation*}
and for $F = \mt$ we write
\begin{equation*}
\Eq_\be^{\mt}(\N\T(X))
\equiv
\Eq_\be^\infty(\N\T(X)) := \{\vphi \in \Eq_\be(\N\T(X)) \mid \vphi(Q_\Bi) = 0 \foral i =1, \dots, N \}.
\end{equation*}
Likewise we define the $F$-parts for the gauge-invariant equilibrium states by 
\[
\GEq_\be^F(\N\T(X)) := \{ \vphi \in \Eq_\be^F(\N\T(X)) \mid \vphi = \vphi E\}.
\]
\end{definition}

The $F$-equilibria are connected with specific ideals in $\N\T(X)$.
We provide this description in the next proposition.

\begin{proposition}\label{P:states ideals}
Let $X$ be a product system of finite rank over $A$ and fix $\be > 0$.
Then a state $\vphi$ is in $\Eq_\be^F(\N\T(X))$ if and only if it restricts to a state on $\sca{Q_F}$ that satisfies the KMS-condition and annihilates $\sca{Q_\Bi \mid i \notin F}$.
\end{proposition}

\begin{proof}
First we claim that $\{\sum_{\un{n} \leq k \cdot \un{1}_F} Q_F^{\un{n}} \mid k \in \bZ_+\}$ defines an approximate unit for $\sca{Q_F}$.
Indeed by \cite[Proposition 4.6]{DK18} and equation (\ref{eq:p-r}) we have that
\begin{align*}
\sca{Q_F} 
& = \ol{\spn}\{t(\xi_{\un{n}}) Q_F t(\eta_{\un{m}})^* \mid \un{n}, \un{m} \in \bZ_+^N\} \\
& = \ol{\spn}\{Q_F^{\un{n}_F} t(\xi_{\un{n}}) t(\eta_{\un{m}})^* Q_F^{\un{m}_F} \mid \un{n}, \un{m} \in \bZ_+^N\} \\ 
& = \ca(Q_F^{\un{n}} f Q_F^{\un{m}} \mid \un{n}, \un{m} \in F, f \in \N\T(X)).
\end{align*}
Then, by using the KMS-condition and orthogonality of $Q_F^{\un{n}}$, we have that $\vphi$ is uniquely identified by
\[
\vphi(f) 
= 
\lim_k \sum_{\un{n}, \un{m} \leq k \cdot \un{1}_F} \vphi(Q_F^{\un{n}} f Q_F^{\un{m}})
=
\lim_k \sum_{\un{n} \leq k \cdot \un{1}_F} \vphi(Q_F^{\un{n}} f Q_F^{\un{n}}).
\]

Secondly we have to show that if $\vphi(Q_\Bi) = 0$ for all $i \notin F$ then $\vphi$ annihilates the ideal the $Q_\Bi$ generate.
By the Cauchy-Schwartz inequality we have that $\vphi( Q_{\Bi} g ) = 0$ for all $g \in \N\T(X)$.
In particular we can use the KMS-condition (as all elements below are analytical) to obtain
\[
\vphi(t(\xi_{\un{n}}) t(\eta_{\un{m}})^* \cdot Q_\Bi \cdot t(\xi_{\un{k}}) t(\eta_{\un{w}})^*) 
= 
e^{- |\un{n} - \un{m}| \be} \vphi(Q_{\Bi} \cdot t(\xi_{\un{k}}) t(\eta_{\un{w}})^* t(\xi_{\un{n}}) t(\eta_{\un{m}})^*)
=
0.
\]
Therefore by continuity $\vphi(f Q_\Bi g) = 0$ for all $f, g \in \N\T(X)$.
\end{proof}

Moreover the $F$-type equilibrium states are $F$-invariant in the following sense.

\begin{proposition}\label{P:gi}
Let $X$ be a product system of finite rank over $A$ and fix $\be > 0$.
For $F \subseteq \{1, \dots, N\}$ we have that $\Eq_\be^F(\N\T(X)) \neq \mt$ if and only if $\GEq_\be^F(\N\T(X)) \neq \mt$.
Moreover for every $\vphi \in \GEq_\be^F(\N\T(X))$ we have that
\[
\vphi(t(\xi_{\un{n}}) t(\eta_{\un{m}})^*) = \de_{\un{n}_F, \un{m}_F} \vphi(t(\xi_{\un{n}}) t(\eta_{\un{m}})^*).
\]
In particular the equilibrium states of finite type are gauge-invariant.
\end{proposition}

\begin{proof}
For the first assertion, let $\vphi \in \Eq_\be(\N\T(X))$ and set $\vphi ' = \vphi E$.
Then $\vphi' E = \vphi E^2 = \vphi'$, so that $\vphi'$ is gauge-invariant.
Moreover we have that $\vphi'$ satisfies equation (\ref{eq:kms3}) since
\begin{align*}
\vphi'(t(\xi_{\un{n}}) t(\eta_{\un{m}})^*)
& =
\de_{\un{n}, \un{m}} \vphi(t(\xi_{\un{n}}) t(\eta_{\un{m}})^*) \\
& =
\de_{\un{n}, \un{m}} e^{-|\un{n}|\be} \vphi(t(\eta_{\un{m}})^* t(\xi_{\un{n}}))
=
\de_{\un{n}, \un{m}} e^{-|\un{n}|\be} \vphi'(t(\eta_{\un{m}})^* t(\xi_{\un{n}})).
\end{align*}
By Proposition \ref{P:char} we thus get that $\vphi'$ is a gauge-invariant equilibrium state at $\be$.
Since the $Q_{F}^{\un{n}}$ are gauge-invariant we get that
\[
\vphi'(Q_{F}^{\un{n}}) = \vphi E(Q_F^{\un{n}}) = \vphi(Q_F^{\un{n}})
\foral
\un{n} \in F.
\]
Therefore $\vphi E \in \Eq_\be^F(\N\T(X))$ whenever $\vphi \in \Eq_\be^F(\N\T(X))$.

For the second assertion we may use equation (\ref{eq:p-r}) and orthogonality of $Q_F^{\un{w}}$ for $\un{w} \in F$ from Lemma \ref{L:proj} to obtain
\begin{align*}
Q_F^{\un{w}} t(\xi_{\un{n}}) t(\eta_{\un{m}})^* Q_F^{\un{w}}
& =
\begin{cases}
Q_F^{\un{w}} t(\xi_{\un{n}}) Q_F^{\un{w} - \un{n}_F} Q_F^{\un{w} - \un{m}_F} t(\eta_{\un{m}})^* Q_F^{\un{w}} & \text{ if } \un{w} \geq \un{n}_F, \un{m}_F, \\
0 & \text{ otherwise},
\end{cases}
\\ & =
\de_{\un{n}_F, \un{m}_F} Q_F^{\un{w}} t(\xi_{\un{n}}) t(\eta_{\un{m}})^* Q_F^{\un{w}}.
\end{align*}
If $\vphi \in \Eq_\be^F(\N\T(X))$ then $\vphi(f) = \sum_{\un{w} \in F} \vphi(Q_F^{\un{w}} f Q_F^{\un{w}})$ for all $f \in \N\T(X)$.
Therefore
\begin{align*}
\vphi(t(\xi_{\un{n}}) t(\eta_{\un{m}})^*)
& =
\de_{\un{n}_F, \un{m}_F} \lim_k \sum_{\un{w} \leq k \cdot \un{1}_F} \vphi(Q_F^{\un{w}} t(\xi_{\un{n}}) t(\eta_{\un{m}})^* Q_F^{\un{w}})
=
\de_{\un{n}_F, \un{m}_F} \vphi(t(\xi_{\un{n}}) t(\eta_{\un{m}})^*).
\end{align*}
For the last assertion just notice that $Q_{\{1, \dots, N\}}^{\un{w}} = p_{\un{w}}$.
\end{proof}

The main theorem of this section is the following convex decomposition.

\begin{theorem}\label{T:con}
Let $X$ be a product system of finite rank over $A$ and let $\be > 0$.
Then every $\vphi$ in $\Eq_\be(\N\T(X))$ (resp. in $\GEq_\be(\N\T(X))$) admits a unique decomposition in a convex combination of $\vphi_F$ in $\Eq_\be^F(\N\T(X))$ (resp. in $\GEq_\be^F(\N\T(X))$).
Moreover the $F$-part $\vphi_F$ is non-trivial if and only if $\vphi(Q_F) \neq 0$.
\end{theorem}

The proof will follow from a number of lemmas.
The only place we require $X$ to have finite rank is to ensure that the projections we use are in $\N\T(X)$ (in fact in $\pi(A)'$).
In what follows fix $\vphi$ be a positive functional that satisfies the KMS-condition at $\be$.
For fixed $\mt \neq C \subseteq \{1, \dots, N\}$ we define the auxiliary positive functionals
\begin{equation}
\vphi_{0, C}(f) := \sum_{\un{n} \in C} \vphi(Q_C^{\un{n}} f Q_{C}^{\un{n}}) \foral f \in \N\T(X).
\end{equation}
Since the $Q_C^{\un{n}}$ are orthogonal projections by Lemma \ref{L:proj}, we have that
\[
\sum_{\textup{finite } \un{n} \in C} \vphi(Q_C^{\un{n}} f Q_{C}^{\un{n}})
\leq
\nor{f} \cdot \vphi\big( \sum_{\textup{finite } \un{n} \in C} Q_C^{\un{n}}\big)
\leq \nor{f},
\]
and so indeed $\vphi_{0,C}$ defines a positive functional.
Due to the KMS-condition we also have that
\[
\vphi_{0, C}(f) 
= 
\sum_{\un{n} \in C} \vphi(Q_C^{\un{n}} f Q_{C}^{\un{n}})
=
\sum_{\un{n} \in C} \vphi(Q_C^{\un{n}} f)
=
\sum_{\un{n} \in C} \vphi(f Q_C^{\un{n}})
=
\sum_{\un{n}, \un{m} \in C} \vphi(Q_C^{\un{n}} f Q_{C}^{\un{m}}).
\]

\begin{lemma}\label{L:00-q}
With the aforementioned notation, we have that the functional $\vphi_{0,C}$ satisfies the KMS-condition and $\vphi_{0, C}(Q_C) = \vphi(Q_C)$.
If $\vphi = \vphi E$ then $\vphi_{0, C} = \vphi_{0,C} E$ as well.
\end{lemma}

\begin{proof}
By Proposition \ref{P:char} it suffices to check that it satisfies equation (\ref{eq:kms2}).
The KMS-condition on $\vphi$ and equation (\ref{eq:p-r}) yield
\begin{align*}
\vphi_{0,C}(t(\xi_{\un{m}}) t(\xi_{\un{k}}) t(\eta_{\un{w}})^*)
& =
e^{-|\un{m}| \be} \sum_{\un{n} \in C} \vphi(t(\xi_{\un{k}}) t(\eta_{\un{w}})^* Q_C^{\un{n}} t(\xi_{\un{m}})) \\
& =
e^{-|\un{m}|\be} \sum \{ \vphi(t(\xi_{\un{k}}) t(\eta_{\un{w}})^* t(\xi_{\un{m}}) Q_C^{\un{n} - \un{m}_C}) \mid \un{n} \geq \un{m}_C, \un{n} \in C\} \\
& =
e^{-|\un{m}|\be} \sum \{ \vphi(t(\xi_{\un{k}}) t(\eta_{\un{w}})^* t(\xi_{\un{m}}) Q_C^{\un{n}}) \mid \un{n} \in C\} \\
& =
e^{-|\un{m}| \be} \vphi_{0,C}(t(\xi_{\un{k}}) t(\eta_{\un{w}})^* t(\xi_{\un{m}}) ).
\end{align*}
Moreover we have that
\[
\vphi_{0, C}(Q_C)
=
\sum_{\un{n} \in C} \vphi(Q_C^{\un{n}} Q_C Q_{C}^{\un{n}})
=
\vphi(Q_C).
\]
The last claim is immediate and the proof is complete.
\end{proof}

\begin{lemma}\label{L:0-q}
With the aforementioned notation we have that $\vphi_{0, C}(Q_D) = \vphi_{0,C \setminus D}(Q_{D})$ for every $C, D \subseteq \{1, \dots, N\}$, with the understanding that $\vphi_{0, \mt}(Q_D) = \vphi_{0, D}(Q_D) = \vphi(Q_D)$.
\end{lemma}

\begin{proof}
For $\un{n} \in C$ we directly compute
\begin{align*}
Q_{C}^{\un{n}} Q_D
& =
\sum_{\ell(\umu) = \un{n}} t(x_{\umu}) Q_{C \setminus D} Q_{D \cap C} t(x_{\umu})^* Q_D
\\
& =
\begin{cases}
\sum_{\ell(\umu) = \un{n}} t(x_{\umu}) Q_{C \setminus D} Q_{D \cap C} Q_D t(x_{\umu})^* Q_D & \text{ if } \un{n} \perp D, \\
0 & \text{ otherwise},
\end{cases} \\
& =
\begin{cases}
\sum_{\ell(\umu) = \un{n}} t(x_{\umu}) Q_{C \setminus D} Q_D t(x_{\umu})^* Q_D & \text{ if } \un{n} \perp D, \\
0 & \text{ otherwise},
\end{cases} \\
& =
\begin{cases}
\sum_{\ell(\umu) = \un{n}} t(x_{\umu}) Q_{C \setminus D} t(x_{\umu})^* Q_D & \text{ if } \un{n} \perp D, \\
0 & \text{ otherwise},
\end{cases}
\\ & =
\begin{cases}
Q_{C \setminus D}^{\un{n}} Q_D & \text{ if } \un{n} \in C \setminus D,\\
0 & \text{ otherwise}.
\end{cases}
\end{align*}
Hence we get that
\begin{align*}
\vphi_{0, C}(Q_D)
& =
\sum_{\un{n} \in C} \vphi(Q_C^{\un{n}} Q_D)
=
\sum_{\un{n} \in C \setminus D} \vphi(Q_{C \setminus D}^{\un{n}} Q_D)
=
\vphi_{0,C \setminus D}(Q_D). \qedhere
\end{align*}
\end{proof}

Now for a fixed $F \neq \mt$ we define the functionals
\begin{equation}
\vphi_F(f) := \sum_{C \supseteq F} (-1)^{|C \setminus F|} \vphi_{0,C}(f) \foral f \in \N\T(X).
\end{equation}

\begin{lemma}\label{L:F-part}
With the aforementioned notation and for $F \neq \mt$, the functional $\vphi_F$ of $\N\T(X)$ is positive, and $\vphi_F = 0$ if $\vphi(Q_F) = 0$.
If $\vphi$ is gauge-invariant then so is $\vphi_F$.
Moreover we have that $\vphi_F \leq \vphi$ and the state $\vphi_F(1)^{-1} \vphi_F$ is in $\Eq_\be^F(\N\T(X))$.
\end{lemma}

\begin{proof}
It is clear that $\vphi_F$ satisfies the KMS-condition, since so does every summand $\vphi_{0,C}$.
Likewise if in addition $\vphi = \vphi E$ then so it holds for $\vphi_F$.

If $\vphi(Q_F) = 0$ then $\vphi(Q_C) = 0$ for every $C \supseteq F$.
Thus for $\un{n} \in C$ we get that
\[
\vphi(Q_C^{\un{n}})
 =
\sum_{\ell(\umu) = \un{n}} e^{-|\un{n}| \be} \vphi(Q_C t(x_{\umu})^* t(x_{\umu}) Q_C)
\leq 
\sum_{\ell(\umu) = \un{n}} e^{-|\un{n}| \be} \vphi(Q_C)
=
0.
\]
Therefore for every $0 \leq f \in \N\T(X)$ we obtain
\[
\vphi_{0,C}(f)
=
\sum_{\un{n} \in C} \vphi(Q_C^{\un{n}} f Q_C^{\un{n}})
\leq
\nor{f} \sum_{\un{n} \in C} \vphi(Q_C^{\un{n}})
=
0.
\]
As $\N\T(X)$ is spanned by its positive elements, we would get that $\vphi_F = 0$.

To show positivity, we may use the alternating sums for every $\un{n}$-level with $\un{n} \leq k \cdot \un{1}_F$ and the alternating sums expansion for the projection
\[
\prod_{i \notin F} (1 - (1 - P_{(k+1) \cdot \Bi})) = \prod_{i \notin F} P_{(k+1) \cdot \Bi}
\]
to deduce that
\begin{align*}
\sum_{C \supseteq F} \sum_{\un{n} \leq k \cdot \un{1}_C} (-1)^{|C \setminus F|} Q_{C}^{\un{n}}
& = 
\sum_{\un{m} \leq k \cdot \un{1}_F} \bigg[ \sum \{ (-1)^{|C \setminus F|} Q_C^{\un{m} + \un{w}} \mid \un{w} \leq k \cdot \un{1}_{C \setminus F}, C \supseteq F\} \bigg] \\
& = 
\sum_{\un{m} \leq k \cdot \un{1}_F} 
\bigg[ \sum_{\ell(\umu) = \un{m}} t(x_{\umu}) \bigg[ \prod_{i \notin F} P_{(k+1) \cdot \Bi} \bigg] Q_{F} t(x_{\umu})^* \bigg]
\equiv
\sum_{\un{m} \leq k \cdot \un{1}_F} R_F^{\un{m}},
\end{align*}
where we write
\[
R_F^{\un{m}}
:=
\sum_{\ell(\umu) = \un{m}} t(x_{\umu}) \bigg[ \prod_{i \notin F} P_{(k+1) \cdot \Bi} \bigg] Q_{F} t(x_{\umu})^*
=
\sum_{\ell(\umu) = \un{m}} | \bigg[ \prod_{i \notin F} P_{(k+1) \cdot \Bi} \bigg] Q_F t(x_{\umu})^*|^2.
\]
Alternatively, for a fixed $\un{m} \leq k \cdot \un{1}_F$, we will show that
\begin{equation}\label{eq:pqr}
R_F^{\un{m}}
=
\sum \{ (-1)^{|C \setminus F|} Q_C^{\un{m} + \un{w}} \mid \un{w} \leq k \cdot \un{1}_{C \setminus F}, C \supseteq F\}.
\end{equation}
Indeed since $\supp \un{m} \subseteq F$ we have
\[
\sum_{\ell(\umu) = \un{m}} t(x_{\umu}) \bigg[ \prod_{i \notin F} P_{(k+1) \cdot \Bi} \bigg] Q_F t(x_{\umu})^* \ze_{\un{n}}
=
\begin{cases}
\ze_{\un{n}} & \text{ if } \un{n}_F = \un{m}, \un{n}_{F^c} \geq (k+1) \cdot \un{1}_{F^c}, \\
0 & \text{ otherwise}.
\end{cases}
\]
On the other hand if $\un{0} \neq \un{w} \in C \setminus F$ then $Q_C^{\un{m} + \un{w}} \ze_{\un{n}} = \ze_{\un{n}}$ exactly when $\un{n}_C = \un{m}_C + \un{w}_C = \un{m} + \un{w}$, otherwise it is zero.
Equivalently we have
\[
Q_C^{\un{m} + \un{w}} \ze_{\un{n}}
=
\begin{cases}
\ze_{\un{n}} & \text{ if } \un{n}_F = \un{m}, \un{n}_{C \setminus F} = \un{w}, \\
0 & \text{ otherwise}.
\end{cases}
\]
Therefore we have the following cases that verify equation (\ref{eq:pqr}):

\noindent
$\bullet$ Case 1, $\un{n}_F = \un{m}$ and $\un{n}_{F^c} \geq (k+1) \cdot \un{1}_{F^c}$.
In this case we have that $\un{n}_{C \setminus F} \geq (k+1) \cdot \un{1}_{C \setminus F} > \un{w}$ for every $C \supseteq F$ and $\un{w} \leq k \cdot \un{1}_{C \setminus F}$.
Hence we have
\begin{align*}
\sum \{ (-1)^{|C \setminus F|} Q_C^{\un{m} + \un{w}} \ze_{\un{n}} \mid \un{w} \leq k \cdot \un{1}_{C \setminus F}, C \supseteq F\}
=
Q_F^{\un{m}} \ze_{\un{n}}
=
\ze_{\un{n}}.
\end{align*}
$\bullet$ Case 2, $\un{n}_F \neq \un{m}$.
In this case we directly verify that
\begin{align*}
\sum \{ (-1)^{|C \setminus F|} Q_C^{\un{m} + \un{w}} \ze_{\un{n}} \mid \un{w} \leq k \cdot \un{1}_{C \setminus F}, C \supseteq F\}
=
0.
\end{align*}
$\bullet$ Case 3, $\un{n}_F = \un{m}$ and $n_i \leq k$ for some $i \notin F$.
In this case set $G := \{i \notin F \mid n_i \leq k\} \neq \mt$.
For every $C \supseteq F$ with $C \setminus F \not\subseteq G$ and $\un{w} \leq k \cdot \un{1}_{C \setminus F}$ we have that $\un{w}_{C \setminus F} \neq \un{n}_{C \setminus F}$.
On the other hand for every $C \supseteq F$ with $C \setminus F \subseteq G$ there exists a unique $\un{w} \leq k \cdot \un{1}_{C \setminus F}$ such that $\un{w} = \un{n}_{C \setminus F}$.
Therefore we deduce
\begin{align*}
\sum \{ (-1)^{|C \setminus F|} Q_C^{\un{m} + \un{w}} \ze_{\un{n}} \mid \un{w} \leq k \cdot \un{1}_{C \setminus F}, C \supseteq F\}
& =
\sum \{ (-1)^{|C \setminus F|} \ze_{\un{n}} \mid C \setminus F \subseteq G \} \\
& =
\sum_{C \subseteq G} (-1)^{|C|} \ze_{\un{n}}
=
0.
\end{align*}

Now we return to the proof.
As in Lemma \ref{L:proj} we have that $\{R_F^{\un{m}}\}_{\un{m} \leq k \cdot \un{1}_F}$ is a family of pairwise orthogonal projections.
Thus the KMS-condition yields
\begin{align*}
\vphi_F(f)
& =
\sum_{C \supseteq F} (-1)^{|C \setminus F|} \sum_{\un{n} \in C} \vphi(Q_C^{\un{n}} f) 
 =
\lim_k \sum_{C \supseteq F} (-1)^{|C \setminus F|} \sum_{\un{n} \leq k \cdot \un{1}_C} \vphi(Q_C^{\un{n}} f) \\
& =
\lim_k \sum_{\un{m} \leq k \cdot \un{1}_F} \vphi(R_F^{\un{m}} f) 
 =
\lim_k \sum_{\un{m} \leq k \cdot \un{1}_F} \vphi(R_F^{\un{m}} f R_F^{\un{m}})
=
\lim_k \sum_{\un{n}, \un{m} \leq k \cdot \un{1}_F} \vphi(R_F^{\un{n}} f R_F^{\un{m}}).
\end{align*}
Therefore $\vphi_F$ is a positive functional and moreover $\vphi_F \leq \vphi$.
In particular Lemma \ref{L:proj} yields
\begin{align*}
\sum_{\un{n} \in F} \vphi_F(Q_F^{\un{n}})
& =
\sum_{C \supseteq F} (-1)^{|C \setminus F|} \, \sum_{\un{n} \in F} \, \sum_{\un{m} \in C} \vphi(Q_F^{\un{n}} Q_C^{\un{m}}) \\
& =
\sum_{C \supseteq F} (-1)^{|C \setminus F|} \, \sum_{\un{m}'' \in C \setminus F} \, \sum_{\un{n} \in F} \, \sum_{\un{m}' \in F} \vphi(Q_F^{\un{n}} Q_C^{\un{m}' + \un{m}''}) \\
& =
\sum_{C \supseteq F} (-1)^{|C \setminus F|} \, \sum_{\un{m}'' \in C \setminus F} \, \sum_{\un{m}' \in F} \vphi(Q_C^{\un{m}' + \un{m}''}) \\
& =
\sum_{C \supseteq F} (-1)^{|C \setminus F|} \, \sum_{\un{m} \in C} \vphi(Q_C^{\un{m}})
=
\vphi_F(1).
\end{align*}
To finish the proof we have to show that $\vphi_F(Q_{\Bi}) = 0$ for every $i \notin F$.
To this end we directly compute
\begin{align*}
\vphi_F(Q_{\Bi})
& =
\sum_{C \supseteq F \cup \{i\}} (-1)^{|C \setminus F|} \vphi_{0,C}(Q_{\Bi})
+
\sum_{i \notin C \supseteq F} (-1)^{|C \setminus F|} \vphi_{0,C}(Q_{\Bi}) \\
& =
\sum_{i \notin C \supseteq F} (-1)^{|C \setminus F| + 1} \vphi_{0,C\cup \{i\}}(Q_{\Bi})
+
\sum_{i \notin C \supseteq F} (-1)^{|C \setminus F|} \vphi_{0,C}(Q_{\Bi}) \\
& =
- \sum_{i \notin C \supseteq F} (-1)^{|C \setminus F|} \vphi_{0,C}(Q_{\Bi})
+
\sum_{i \notin C \supseteq F} (-1)^{|C \setminus F|} \vphi_{0,C}(Q_{\Bi})
=
0,
\end{align*}
where we used that $\vphi_{0, C \cup \{i\}}(Q_{\Bi}) = \vphi_{0,C}(Q_{\Bi})$ when $i \notin C$ from Lemma \ref{L:0-q}.
\end{proof}

\begin{lemma}\label{L:less 0}
With the aforementioned notation and for $\mt \neq F \neq \{1, \dots, N\}$, we have that $(\vphi_C)_F = 0$ whenever $C \neq F$ with $|C| \geq |F|$.
\end{lemma}

\begin{proof}
We first show that $(\vphi_C)_{0,D} = 0$ whenever $C \not\supseteq D$.
Indeed in this case there exists an $i \in D \setminus C$ and so $\vphi_C(Q_D) = 0$.
In particular we have $\vphi_C(Q_D^{\un{n}} f) = 0$ for all $\un{n} \in D$ and $f \in \N\T(X)$.
Therefore we get that
\[
(\vphi_C)_{0,D}(f) = \sum_{\un{n} \in D} \vphi_C(Q_{D}^{\un{n}} f) = 0.
\]

We will consider two cases for $C$ and $F$.
First suppose that $C \not\supseteq F$.
If $D \supseteq F$ then $C \not\supseteq D$, and so by the above we have that
\begin{align*}
(\vphi_C)_F
& =
\sum_{D \supseteq F} (-1)^{|D \setminus F|} (\vphi_C)_{0, D}
 =
0.
\end{align*}

Secondly suppose that $C \supsetneq F$.
If $C \not\supseteq D$ then as before we have $(\vphi_C)_{0,D} = 0$.
If $C \supseteq D$ then Lemma \ref{L:proj} gives $(\vphi_{0,G})_{0,D} = \vphi_{0,G}$ for all $G \supseteq C$ since
\begin{align*}
(\vphi_{0,G})_{0,D}(f) 
& = 
\sum_{\un{n} \in D} \sum_{\un{m} \in G} \vphi(Q_G^{\un{m}} Q_D^{\un{n}} f)
= \sum_{\un{n} \in D} \sum_{\un{m}' \in D} \sum_{\un{m}'' \in G \setminus D} \vphi(Q_G^{\un{m}' + \un{m}''} Q_D^{\un{n}} f) \\
& = \sum_{\un{m}' \in D} \sum_{\un{m}'' \in G \setminus D} \vphi(Q_G^{\un{m}' + \un{m}''} f)
= \vphi_{0,G}(f).
\end{align*}
Thus we have
\begin{align*}
(\vphi_C)_F
& =
\sum_{D \supseteq F} (-1)^{|D \setminus F|} \sum_{G \supseteq C} (-1)^{|G \setminus C|} (\vphi_{0,G})_{0,D} 
=
\sum_{C \supseteq D \supseteq F} (-1)^{|D \setminus F|} \sum_{G \supseteq C} (-1)^{|G \setminus C|} (\vphi_{0,G})_{0,D} \\
& =
\sum_{C \supseteq D \supseteq F} (-1)^{|D \setminus F|} \sum_{G \supseteq C} (-1)^{|G \setminus C|} \vphi_{0,G}
=
\de_{C, \{1, \dots, N\}} \sum_{D \supseteq F} (-1)^{|D \setminus F|} \vphi_{0, \{1, \dots, N\}}
=
0,
\end{align*}
since $F \neq \{1, \dots, N\}$.
\end{proof}

\begin{proof}[{\bf Proof of Theorem \ref{T:con}}]
For every $\mt \neq F \subseteq \{1, \dots, N\}$ let the functionals $\vphi_F$ as defined above and set
\[
\vphi_\infty := \vphi - \sum \{ \vphi_F \mid \mt \neq F \subseteq \{1, \dots, N\}\}.
\]
As we have observed every $\vphi_F$ satisfies the KMS-condition and its normalization will provide the $F$-component for the convex decomposition of $\vphi$.
Consequently the normalization of $\vphi_\infty$ will give the infinite part of $\vphi$.

To see that $\vphi_\infty$ is positive we proceed inductively.
First consider $F = \{1, \dots, N\}$.
By Lemma \ref{L:0-q} we get that
\[
\vphi_{\{1, \dots, N\}}(Q_{\{1, \dots, N\}}) = \vphi_{0, \{1, \dots, N\}}(Q_{\{1, \dots, N\}}) =  \vphi(Q_{\{1, \dots, N\}}). 
\]
The functional $\vphi - \vphi_{\{1, \dots, N\}}$ satisfies the KMS-condition and it is positive by Lemma \ref{L:F-part}.
It annihilates $Q_{\{1, \dots, N\}}$ and thus as in Proposition \ref{P:states ideals}, it annihilates the ideal that $Q_{\{1, \dots, N\}}$ generates.
Hence $\vphi - \vphi_{\{1, \dots, N\}}$ induces a positive functional on the quotient of $\N\T(X)$ by the ideal $\sca{Q_{\{1, \dots, N\}}}$.

Now let $F = \{2, \dots, N\}$.
By Lemma \ref{L:less 0} and Lemma \ref{L:F-part} we also have that
\[
\vphi_{\{2, \dots, N\}} = (\vphi - \vphi_{\{1, \dots, N\}})_{\{2, \dots, N\}} \leq \vphi - \vphi_{\{1, \dots, N\}}
\]
and therefore the functional 
\[
\vphi  - \vphi_{\{1, \dots, N\}} - \vphi_{\{2, \dots, N\}}
\]
is positive.
By Lemma \ref{L:0-q} we have that
\begin{align*}
\vphi_{\{1, \dots, N\}} (Q_{\{2, \dots, N\}}) + \vphi_{\{2, \dots, N\}} (Q_{\{2, \dots, N\}}) 
& = \\
& \hspace{-3.5cm} = 
\vphi_{0, \{1, \dots, N\}} (Q_{\{2, \dots, N\}})
+ \vphi_{0, \{2, \dots, N\}} (Q_{\{2, \dots, N\}})
- \vphi_{0, \{1, \dots, N\}} (Q_{\{2, \dots, N\}}) \\
& \hspace{-3.5cm} = \vphi (Q_{\{2, \dots, N\}}).
\end{align*}
Moreover $\vphi_{\{2, \dots, N\}}(Q_{\bo{1}}) = 0$ by Lemma \ref{L:F-part}, and so
\[
\vphi_{\{2, \dots, N\}}(Q_{\{1, \dots, N\}}) = 0.
\] 
Therefore the positive functional $\vphi  - \vphi_{\{1, \dots, N\}} - \vphi_{\{2, \dots, N\}}$ annihilates both $Q_{\{1, \dots, N\}}$ and $Q_{\{2, \dots, N\}}$ and satisfies the KMS-condition.
Hence, as in Proposition \ref{P:states ideals}, it defines a positive functional on the quotient of $\N\T(X)$ by the ideal $\sca{Q_{\{1, \dots, N\}}, Q_{\{2, \dots, N\}}}$.

For the inductive hypothesis let $|F_1| = \cdots = |F_k| = n$ and suppose that
\[
\vphi - \sum_{|F| \geq n +1} \vphi_F - \sum_{l=1}^k \vphi_{F_l}
\]
defines a positive functional on the ideal
\[
\sca{Q_F, Q_{F_l} \mid |F| \geq n+1, l=1, \dots, k}.
\]
Let $D \neq F_1, \dots, F_k$ such that $|D| = n$.
We wish to show that the functional
\[
\vphi - \sum_{|F| \geq n+1} \vphi_F - \sum_{l=1}^{k} \vphi_{F_l} - \vphi_D
\]
is positive and annihilates the projections
\[
Q_D
\qand
\{Q_F, Q_{F_l} \mid |F| \geq n+1, l=1, \dots, k\}.
\]
An application of Lemma \ref{L:less 0} and Lemma \ref{L:F-part} gives positivity since
\[
\vphi_D = (\vphi - \sum_{|F| \geq n+1} \vphi_F - \sum_{l=1}^{k} \vphi_{F_l})_D \leq \vphi - \sum_{|F| \geq n+1} \vphi_F - \sum_{l=1}^{k} \vphi_{F_l}.
\]
For $G$ such that $D \cap G^c \neq \mt$ and $i \in D \setminus G$ we have
\[
\vphi_G(Q_{D}) \leq \vphi_G(Q_\Bi) = 0.
\]
In particular $\vphi_{F_l}(Q_D) = 0$ for all $l=1, \dots, k$.
Therefore we obtain
\begin{align*}
\sum_{|F| \geq n+1} \vphi_F(Q_D) + \sum_{l=1}^{k} \vphi_{F_l}(Q_D) + \vphi_D(Q_D)
& = \\
& \hspace{-5cm} =
\sum_{F : D \subseteq F} \bigg[ \sum_{C : F \subseteq C} (-1)^{|C \setminus F|} \vphi_{0, C}(Q_D) \bigg]
 =
\sum_{C : D \subseteq C} \bigg[ \sum_{F : D \subseteq F \subseteq C} (-1)^{|C \setminus F|} \vphi_{0, C}(Q_D) \bigg] \\
& \hspace{-5cm} =
\sum_{C : D \subseteq C} \vphi_{0, C}(Q_D) \bigg[ \sum_{F : D \subseteq F \subseteq C} (-1)^{|C \setminus F|} \bigg]
 =
\vphi_{0, D}(Q_D)
=
\vphi(Q_D),
\end{align*}
since $\sum_{F : D \subseteq F \subseteq C} (-1)^{|C \setminus F|} = 0$ unless $D = C$.
Furthermore we have that $F_l \cap D^c\neq \mt$ as $D \neq F_l$ and $F \cap D^c \neq \mt$ as $|D| < |F|$.
Therefore $G \cap D^c \neq \mt$ for any
\[
G \in \{F, F_l \mid |F| \geq k+1, l=1, \dots, k\}.
\]
By choosing $i \in G \setminus D$ we get that
\[
\vphi_D(Q_{G}) \leq \vphi_D(Q_{\Bi}) = 0.
\]
Finally note that the positive functional
\[
\vphi - \sum_{|F| \geq n+1} \vphi_F - \sum_{l=1}^{k} \vphi_{F_l} - \vphi_D
\]
satisfies the KMS condition as so does every summand.
Therefore it defines a positive functional on the quotient 
\[
\N\T(X) / \sca{Q_F, Q_{F_l}, Q_D \mid |F| \geq n +1, l=1, \dots k}.
\]

Inducting on $|F| = N, N-1, \dots, 1$ then gives that $\vphi_\infty$ is positive on the quotient of $\N\T(X)$ by the ideal $\sca{Q_F \mid \mt \neq F} = \sca{Q_{\Bi} \mid i=1, \dots, N}$ which is exactly $\N\O(A,X)$.
This includes that $\vphi_\infty(Q_{\Bi}) = 0$ for all $i \in \{1, \dots, N\}$.

To show uniqueness of the decomposition suppose that there are $\vphi'_F \in \Eq_\be^F(\N\T(X))$ with
\[
\vphi_\infty + \sum_{\mt \neq F \subseteq \{1, \dots, N\} } \vphi_F 
=
\vphi'_\infty + \sum_{\mt \neq F \subseteq \{1, \dots, N\} } \vphi'_F.
\]
By Proposition \ref{P:states ideals} every $\vphi_F, \vphi'_F \in \Eq_\be^F(\N\T(X))$ is uniquely identified by its restriction on $\sca{Q_F}$ while
\[
\vphi_F|_{\sca{Q_D}} = 0= \vphi'_F|_{\sca{Q_D}}
\qif
D \cap F^c \neq \mt.
\]
Applying on $\sca{Q_{\{1, \dots, N\}}}$ yields
\[
\vphi_{\{1, \dots, N\}}|_{\sca{Q_{\{1, \dots, N\}}}} = \vphi'_{\{1, \dots, N\}}|_{\sca{Q_{\{1, \dots, N\}}}}
\]
and thus so do their unique extensions on $\N\T(X)$.
We may thus remove those from the sums and proceed with $\{2, \dots, N\}$.
Inducting on the sets $|F| = N -1$ identifies $\vphi_F$ with $\vphi'_F$ whenever $|F| = N, N-1$.
Inducting on $|F|$ concludes that so is true for all $F$ leaving $\vphi_\infty = \vphi'_\infty$.
\end{proof}

\section{Parametrization of the gauge-invariant equilibria}\label{S:para}

We now proceed to the parametrization of the $F$-components of the gauge-invariant equilibrium states.
We need to identify the corresponding $F$-parts in the simplex $\Tr(A)$ of tracial states of $A$.

\begin{definition}
Let $X$ be a product system over $A$ with a unit decomposition $x = \{x_{i,j} \mid j=1, \dots, d_i, i=1, \dots, N\}$ and let $\be >0$.
For every $\mt \neq F \subseteq \{1, \dots, N\}$ and $\tau \in \Tr(A)$ we define
\begin{equation*}
c_{\tau, \be}^F := \sum \{ e^{- |\umu| \be} \tau(\sca{x_{\umu}, x_{\umu}}) \mid \ell(\umu) \in F \}.
\end{equation*}
Moreover we define the $F$-set of tracial states of $A$ by
\begin{equation*}
\Tr_{\be}^F(A)
:=
\{ \tau \in \Tr(A) \mid c_{\tau,\be}^F < \infty
\text{ and }
e^{\be}\tau(a) = \sum_{j=1}^{d_i} \sca{x_{i,j}, a x_{i,j}} \foral i \notin F\}.
\end{equation*}
In particular for $F = \{1, \dots, N\}$ we write
\begin{equation*}
\Tr_\be^{\fty}(A) := \{\tau \in \Tr(A) \mid c_{\tau, \be}^{\{1, \dots, N\}} = \sum_{|\umu| = k} e^{- k \be} \tau(\sca{x_{\umu}, x_{\umu}}) < \infty \}.
\end{equation*}
The case of $F = \mt$ is recaptured in the averaging traces, namely
\begin{equation*}
\Avt_\be(A) := \{\tau \in \Tr(A) \mid e^{\be}\tau(a) = \sum_{j=1}^{d_i} \tau(\sca{x_{i,j}, a x_{i,j}}) \foral i = 1, \dots, N\}.
\end{equation*}
\end{definition}

Due to Remark \ref{R:qizero} that will follow, the above definitions are independent of the choice of the unit decomposition.

\begin{proposition}\label{P:cap}
Let $X$ be a product system of finite rank over $A$ and let $\be > 0$.
Then $\Tr_\be^F(A) \cap \Tr_\be^{F'}(A) = \mt$ whenever $F \neq F'$.
\end{proposition}

\begin{proof}
Without loss of generality let $i \in F' \setminus F$ and suppose there is a $\tau \in \Tr_\be^F(A) \cap \Tr_\be^{F'}(A)$.
Then $e^{k\be} = \sum_{|\mu_i| = k} \tau(\sca{x_{i, \mu_i}, x_{i,\mu_i}})$ as $i \notin F$ and so we reach the contradiction
\[
\infty > c_{\tau, \be}^{F'} \geq \sum_{k=0}^\infty \sum_{|\mu_i| = k} e^{-k\be} \tau(\sca{x_{i,\mu_i}, x_{i,\mu_i}}) = \sum_{k=0}^\infty 1 = \infty.
\qedhere
\]
\end{proof}

We wish to establish a parametrization of $\GEq_\be^F(\N\T(X))$ by an appropriate sub-simplex of $\Tr_{\be}^F(A)$.
In order to achieve this we need a characterization of the ideals given by $\fI_F = \ker\{ A \to \N\O(F, A, X) \}$ from Definition \ref{D:rel}.

\begin{proposition}\label{P:ideals}
Let $X$ be a product system of finite rank over $A$ and fix a set $\mt \neq F \subseteq \{1, \dots, N\}$.
Then
\begin{equation*}
\fI_F  = \{a \in A \mid \lim_{k} \vphi_{k \cdot \un{1}_F}(a) = 0 \}.
\end{equation*}
\end{proposition}

\begin{proof}
Recall that by definition we have $\N\O(F, A, X) = \N\T(X)/ \sca{Q_\Bi \mid i \in F}$ and \cite[Proposition 4.6]{DK18} yields
\[
\sca{Q_\Bi \mid i \in F} = \ol{\spn} \{t(X_{\un{n}}) Q_C t(X_{\un{m}})^* \mid \mt \neq C \subseteq F,  \un{n}, \un{m} \in \bZ_+^N\}.
\]
Equation (\ref{eq:p-r}) then implies
\begin{align*}
Q_{F^c} \sca{Q_\Bi \mid i \in F} Q_{F^c}
=
\ol{\spn} \{t(X_{\un{n}}) Q_{C \cup F^c} t(X_{\un{m}})^* \mid \mt \neq C \subseteq F, \un{n}, \un{m} \in F\}.
\end{align*}
Let the projections $p(k):= p_{k \cdot \un{1}_F}$ for $k \in \bZ_+$.
It is clear that if $\un{n}, \un{m} \in F$ and $k$ is large enough so that $k \cdot \un{1}_C > \un{n}_C, \un{m}_C$ then
\[
p(k) t(\xi_{\un{n}}) Q_{C \cup F^c} t(\xi_{\un{m}})^* p(k) = 0.
\]
Therefore if $\pi(a) \in \fI_{F}$ then $Q_{F^c} \pi(a) Q_{F^c} \in Q_{F^c} \sca{Q_\Bi \mid i \in F} Q_{F^c}$ and so
\[
\lim_k \phi_{k \cdot \un{1}_{F}}(a) = \lim_k p(k) \pi(a) p(k) = \lim_k p(k) Q_{F^c} \pi(a) Q_{F^c} p(k) = 0.
\]
Conversely, fix $\eps >0$ and let $k \in \bZ_+$ such that $\| \phi_{k \cdot \un{1}_{F}}(a) \| < \eps$.
If $i \in F$ then
\[
1 - P_{k \cdot \bo{i}} = \sum\{p_{\un{n}} \mid n_i=0, \dots, k-1\} = \sum_{l=0}^{k - 1} Q_{\Bi}^{l \cdot \Bi} \in \fI_F,
\]
and therefore
\[
f := \sum_{\mt \neq C \subseteq F} (-1)^{|C|} \pi(a) \prod_{i \in C} (1 - P_{k \cdot \bo{i}}) \in \fI_F.
\]
Then we get
\begin{align*}
\| \pi(a) + \fI_F \|
& \leq
\|\pi(a) + f\| 
 =
\| \sum_{C \subseteq F} (-1)^{|C|} \pi(a) \prod_{i \in C} (1 - P_{k \cdot \bo{i}}) \| \\
& =
\| \pi(a) \prod_{i \in F} (1 - (1-P_{k \cdot \Bi})) \| 
 =
\| \pi(a) \prod_{i \in F} P_{k \cdot \bo{i}} \| \\
& = 
\| \sumoplus_{\un{n} \geq k \cdot \un{1}_F} \phi_{\un{n}}(a) \| 
= \sup_{\un{n} \in \bZ_+^N} \| \phi_{k \cdot \un{1}_F}(a) \otimes \id_{X_{\un{n}}} \| 
\leq
\|\phi_{k \cdot \un{1}_F} \|
< \eps.
\end{align*}
As $\eps$ was arbitrary we derive that $\pi(a) \in \fI_F$.
\end{proof}

\begin{theorem}\label{T:para}
Let $X$ be a product system of finite rank over $A$ and $\be > 0$.
Then we have the following parametrization:

\noindent
$(1)$ For $F = \mt$ there is a bijection
\[
\Phi^\infty \colon \{\tau \in \Avt_\be(A) \mid \tau|_{\fI_{\{1, \dots, N\}}} = 0\}
\to
\GEq_{\be}^\infty(\N\T(X)),
\]
such that
\[
\Phi^\infty_{\tau}(\pi(a)) = \tau(a) \foral a \in A.
\]
$(2)$ For $F \neq \mt$ there is a bijection
\[
\Phi^F \colon \{ \tau \in \Tr_\be^F(A) \mid \tau|_{\fI_{F^c}} = 0 \} \to \GEq_\be^F(\N\T(X)),
\]
such that
\[
\Phi^F_{\tau}(Q_F) \cdot c_{\tau, \be}^F = 1
\qand
\Phi^F_\tau(Q_F \pi(a) Q_F) = \Phi^F_\tau(Q_F) \cdot \tau(a) \foral a \in A.
\]
If $x = \{x_{i,j} \mid j = 1, \dots, d_i, i=1, \dots, N\}$ is a unit decomposition for $X$ then
\begin{align*}
\Phi^F_{\tau}(t(\xi_{\un{n}}) t(\eta_{\un{m}})^*)
& = 
\de_{\un{n}, \un{m}}  \; (c_{\tau, \be}^F)^{-1} \sum_{\ell(\umu) \in F}  e^{-(|\un{n}| + |\umu|) \be} \tau(\sca{\eta_{\un{m}} \otimes x_{\umu}, \xi_{\un{n}} \otimes x_{\umu}}),
\end{align*}
This description is independent of the choice of the decomposition. \\
$(3)$ Every $\Phi^F$ for $F\neq \mt$ respects convex combinations in the sense that if $\la \in (0,1)$, then
\[
\Phi^F(\tau)
=
\la \frac{c_{\tau_1, \be}^F}{c_{\tau, \be}^F} \Phi^F(\tau_1)
+
(1 - \la) \frac{c_{\tau_2, \be}^F}{c_{\tau, \be}^F} \Phi^F(\tau_2)
\]
for $\tau = \la \tau_1 + (1-\la) \tau_2$ with $\tau_1, \tau_2 \in \Tr_\be^F(A)$, and 
\[
(\Phi^F)^{-1}(\vphi)
=
\la \frac{\vphi_1(Q_F)}{\vphi(Q_F)} (\Phi^F)^{-1}(\vphi_1)
+
(1 - \la) \frac{\vphi_2(Q_F)}{\vphi(Q_F)} (\Phi^F)^{-1}(\vphi_2)
\]
for $\vphi = \la \vphi_1 + (1-\la) \vphi_2$ with $\vphi_1, \vphi_2 \in \GEq_\be^F(\N\T(X))$.
Therefore the parametrizations preserve the extreme points of the simplices.
\end{theorem}

The proof follows from a number of steps.
Henceforth we write $\Phi \equiv \Phi^F$.

\begin{lemma}
Let $\tau \in \Tr_\be^F(A)$ with $\tau|_{\fI_{F^c}} = 0$.
Consider the C*-subalgebra
\begin{equation*}
\B_{F^c} := \ol{\spn}\{t(X_{\un{k}}) t(X_{\un{w}})^* \mid \un{k}, \un{w} \perp F\}
\end{equation*}
of the fixed point algebra $\N\T(X)^\ga$.
Then $\tau$ extends to a gauge invariant state $\wt\tau$ on $\B_{F^c}$ that satisfies the KMS-condition (\ref{eq:kms3}).
\end{lemma}

\begin{proof}
It suffices to extend $\tau$ on $E(\B_{F^c})$.
Let $q \colon \N\T(X) \to \N\O(F^c, A, X)$ be the canonical quotient map.
For convenience set 
\begin{equation*}
\si := q \pi, \; s := q t \qand \psi' = q \psi.
\end{equation*}
Let $B = \si(A)$ and $Y_{\un{n}} = \ol{s(X_{\un{n}})}$.
Now we see that $Y = \{Y_{\un{n}}\}_{\un{n} \perp F}$ is a product system of finite rank.
Moreover it is injective.
Indeed for $i \notin F$ the covariance on $\N\O(F^c,A,X)$ gives that
\[
\sum_{j=1}^{d_i} s(x_{i,j}) s(x_{i,j})^* = \psi_{\Bi}'(\phi_X(1_A)) = \si(1_A) = 1.
\]
Therefore if $\si(a) \in \ker \phi_{Y, \Bi}$ then $\si(a) = \si(a) \sum_{j=1}^{d_i} s(x_{i,j}) s(x_{i,j})^* = 0$.

We claim that the identity representation $\id \colon Y \to \N\O(F^c, A, X)$ gives the inclusion $\N\O(B, Y) \subseteq \N\O(F^c, A ,X)$.
As $Y$ is regular we have that $\N\O(B, Y) = \N\O(Y)$.
The identity representation is trivially injective on $B$ and admits a gauge action.
Moreover for every $i \perp F$ we have
\[
\id(b)(I - \sum_{j=1}^{d_i} \id(s(x_{i,j})) \id(s(x_{i,j}))^*) = b q(Q_\Bi) = 0.
\]
By \cite{DK18} then $(\id_B, \id_Y)$ is covariant along all directions and thus by \cite{SY11} it lifts to a faithful representation of $\N\O(B, Y)$.

Similar to the one-variable case for regular C*-correspondences \cite{Pim97} we obtain that the fixed point algebra of $\N\O(Y)$ can be written as a direct limit.
Namely, we have that
\[
\varinjlim (\K Y_{\un{n}}, \otimes \id)
\simeq
\N\O(Y)^\ga
=
\ol{\spn} \{\psi'_{\un{n}}(\K X_{\un{n}}) \mid \un{n} \perp F \}
=
E(\B_{F^c}),
\]
where
\[
\otimes \id_{Y_{\un{m}}} \colon \K Y_{\un{n}} \to \K Y_{\un{n} + \un{m}} : 
\theta_{\xi_{\un{n}}, \eta_{\un{n}}}^{X_{\un{n}}} \mapsto \theta_{\xi_{\un{n}}, \eta_{\un{n}}}^{X_{\un{n}}} \otimes \id_{Y_{\un{m}}}
=
\sum_{\ell(\umu) = \un{m}} \theta_{\xi_{\un{n} x_{\umu}}, \eta_{\un{n}} x_{\umu}}^{X_{\un{n}}}.
\]
Therefore we obtain the diagram
\[
\xymatrix@C=1.5cm{
\K Y_{\un{n}} \ar[rr]^{\otimes \id_{Y_{\un{m}}}} & & \K Y_{\un{n} + \un{m}} \ar[d]^{\simeq} \\
\psi_{\un{n}}'(\K X_{\un{n}}) \ar@{-->}[rr]^{\iota_{\un{n}}^{\un{n} + \un{m}}} \ar[u]^{\simeq} & & \psi_{\un{n} + \un{m}}'(\K X_{\un{n} + \un{m}})
}
\]
where the induced map $\iota_{\un{n}}^{\un{n} + \un{m}}$ is given by
\[
\iota_{\un{n}}^{\un{n} + \un{m}} (s(\xi_{\un{n}}) s(\eta_{\un{n}})^*)
=
\sum_{\ell(\umu) = \un{m}} s(\xi_{\un{n}}) s(x_{\umu}) s(x_{\umu})^* s(\eta_{\un{n}})^*.
\]
In order to extend $\tau$ to a state on $E(\B_{F^c})$ we have to find states $\tau'_{\un{n}} \colon \psi_{\un{n}}'(\K X_{\un{n}}) \to \bC$ for every $\un{n} \perp F$ that are compatible with the direct limit connecting maps, i.e., that they satisfy 
\[
\tau'_{\un{n} + \Bi} \iota_{\un{n}}^{\un{n} + \Bi} = \tau'_{\un{n}}
\foral 
i \notin F.
\]

For the first step $\tau$ defines a tracial state $\tau'_{\un{0}}$ on $\si(A)$ with $\tau = \tau'_{\un{0}} \si$ as it factors through $q$.
Moreover it is clear that if $\un{n} \perp F$ then
\[
\tau'_{\un{0}} \si(a) = e^{- |\un{n}| \be} \sum_{\ell(\umu) = \un{n}} \tau(\sca{x_{\umu}, a x_{\umu}}).
\]
For every $\un{n} \perp F$ we define the functional $\tau'_{\un{n}}$ on $\psi_{\un{n}}'(\K X_{\un{n}})$ by
\begin{equation*}
\tau'_{\un{n}} \psi'_{\un{n}}(k_{\un{n}}) := 
e^{-|\un{n}|\be} \sum_{\ell(\umu) = \un{n}} \tau(\sca{x_{\umu}, k_{\un{n}} x_{\umu}}) \FOR k_{\un{n}} \in \K X_{\un{n}}.
\end{equation*}
If $\psi_{\un{n}}(k_{\un{n}}) \in \ker q$ then $\pi(\sca{x_{\umu}, k_{\un{n}} x_{\umu}}) = t(x_{\umu})^* \psi_{\un{n}}(k_{\un{n}}) t(x_{\umu}) \in \ker q \cap \pi(A)$ and thus $\tau'_{\un{n}}$ is well defined.
Note also that $\tau'_{\un{n}}$ does not depend on the choice of the decomposition.
For if $\{y_{\unu} \mid \ell(\unu) = \un{n}\}$ is another decomposition of $X_{\un{n}}$ then
\begin{align*}
\sum_{\ell(\umu) = \un{n}} \tau(\sca{x_{\umu}, k_{\un{n}} x_{\umu}})
& =
\sum_{\ell(\umu) = \un{n}} \sum_{\ell(\unu) = \un{n}} \tau(\sca{x_{\umu}, k_{\un{n}} y_{\unu}} \sca{y_{\unu}, x_{\umu}} ) \\
& =
\sum_{\ell(\unu) = \un{n}} \sum_{\ell(\umu) = \un{n}} \tau(\sca{y_{\unu}, x_{\umu}} \sca{x_{\umu}, k_{\un{n}} y_{\unu}} )
 =
\sum_{\ell(\unu) = \un{n}} \tau(\sca{y_{\unu}, k_{\un{n}} y_{\unu}} ).
\end{align*}
Every $\tau'_{\un{n}}$ is a state since for every positive contraction $k_{\un{n}}$ we obtain
\[
e^{- |\un{n}| \be} \sum_{\ell(\umu) = \un{n}} \tau(\sca{x_{\umu}, k_{\un{n}} x_{\umu}})
\leq
e^{- |\un{n}| \be} \sum_{\ell(\mu) = \un{n}} \tau(\sca{x_{\umu}, x_{\umu}})
=
\tau(1)
=
1.
\]
Now we see that every $\tau'_{\un{n}}$ satisfies 
\begin{align*}
\tau'_{\un{n}}(s(\xi_{\un{n}}) s(\eta_{\un{n}})^*)
& =
e^{-|\un{n}| \be} \sum_{\ell(\umu) = \un{n}} \tau(\sca{x_{\umu}, \theta^{X_{\un{n}}}_{\xi_{\un{n}}, \eta_{\un{n}}} x_{\umu}}) \\
& =
e^{-|\un{n}| \be} \sum_{\ell(\umu) = \un{n}} \tau(\sca{\eta_{\un{n}}, x_{\umu}} \sca{x_{\umu}, \xi_{\un{n}}}) 
 =
e^{-|\un{n}| \be} \tau(\sca{\eta_{\un{n}}, \xi_{\un{n}}}).
\end{align*}
Now we can verify that $\tau'_{\un{n} + \Bi} \iota_{\un{n}}^{\un{n} + \Bi} = \tau'_{\un{n}}$ for $i \notin F$ by computing on rank one operators:
\begin{align*}
\tau'_{\un{n} + \Bi} \iota_{\un{n}}^{\un{n} + \Bi} (s(\xi_{\un{n}}) s(\eta_{\un{n}})^*)
& =
\sum_{j =1}^{d_i} \tau'_{\un{n} + \Bi}(s(\xi_{\un{n}}) s(x_{i,j}) s(x_{i,j})^* s(\eta_{\un{n}})^*) \\
& =
e^{-(|\un{n}| + 1) \be} \sum_{j=1}^{d_i} \tau(\sca{\eta_{\un{n}} \otimes x_{i,j}, \xi_{\un{n}} \otimes x_{i,j}}) \\
& =
e^{-(|\un{n}| + 1) \be} \sum_{j=1}^{d_i} \tau(\sca{x_{i,j}, \sca{\eta_{\un{n}}, \xi_{\un{n}}} x_{i,j}}) \\
& =
e^{-|\un{n}| \be} \tau(\sca{\eta_{\un{n}}, \xi_{\un{n}}}) 
=
\tau'_{\un{n}}(s(\xi_{\un{n}}) s(\eta_{\un{n}})^*).
\end{align*}

Write $\wt\tau' := \varinjlim \tau'_{\un{n}}$ for the induced state on $E(\B_{F^c})$ that extends every $\tau'_{\un{n}}$.
We see that $\wt\tau'$ satisfies the KMS-condition since
\begin{align*}
\wt\tau'(s(\xi_{\un{n}}) s(\eta_{\un{n}})^*)
& =
\tau'_{\un{n}} (s(\xi_{\un{n}}) s(\eta_{\un{n}})^*)
=
e^{-|\un{n}| \be} \tau(\sca{\eta_{\un{n}}, \xi_{\un{n}}})
=
e^{-|\un{n}| \be} \wt\tau'( s(\eta_{\un{n}})^* s(\xi_{\un{n}})).
\end{align*}
Therefore we get that the functional
\begin{equation}
\wt\tau := \wt\tau' q E
\end{equation}
defines a gauge invariant state on $\B_{F^c}$ that satisfies the KMS-condition, and clearly $\wt \tau \pi = \tau$.
\end{proof}

\begin{proof}[\bf Proof of Theorem \ref{T:para} (1)]
It is clear that $\wt\tau(Q_{\Bi}) = 0$ for every $i \notin F$.
If $F = \mt$ we stop the construction here and deduce the weak*-homeomorphism
\[
\Phi^\infty \colon \Avt_\be(A) \cap \{\tau \in \Tr(A) \mid \tau|_{\fI_{\{1, \dots, N\}}} = 0\}
\to
\GEq_\be^\infty(\N\T(X))
: \tau \mapsto \wt\tau. \qedhere
\]
\end{proof}
If $F \neq \mt$ then we consider the projection $Q_F = \prod_{i \in F} (1 - P_{\Bi})$ and we will construct the equilibrium state by using the statistical approximations on $Q_F$.
For every $\un{n} \in F$ we define the C*-correspondence
\[
Z_{\un{n}} := \ol{t(X_{\un{n}}) \B_{F^c}}
\text{ over }
\B_{F^c} := \ol{\spn}\{t(X_{\un{k}}) t(X_{\un{w}})^* \mid \un{k}, \un{w} \perp F\},
\]
where the bimodule structure is induced from $\N\T(X)$.
Since $\B_{F^c}$ is unital we have that $t(X_{\un{n}}) \subseteq t(X_{\un{n}}) \B_{F^c}$ and thus for $\un{n}, \un{m} \in F$ we derive
\[
t(X_{\un{n}}) t(X_{\un{m}}) \B_{F^c}
\subseteq
t(X_{\un{n}}) \B_{F^c} t(X_{\un{m}}) \B_{F^c}
\subseteq
t(X_{\un{n} + \un{m}}) \B_{F^c}.
\]
Thus $\{Z_{\un{n}}\}_{\un{n} \in F}$ defines a product system.

By the Gauge-Invariant-Uniqueness-Theorem we have $\N\T(X) = \N\T(Z)$.
Indeed first notice that the map $\id \colon Z \to \N\T(X)$ defines a Nica-covariant representation with a gauge action.
Moreover $\ca(\id_{\B_{F^c}}, \id_Z) = \N\T(X)$ since
\[
t(X_{\un{n}}) t(X_{\un{m}})^* \subseteq t(X_{\un{n}_F}) \B_{F^c} t(X_{\un{m}_F})^*
\foral 
\un{n}, \un{m} \in \bZ_+^N.
\]
Finally we need to verify that $\B_{F^c} \cap \B_{(\un{0}, \infty]}^Z = (0)$ where
\[
\B_{(\un{0}, \infty]}^Z := \ol{\spn}\{ t(\xi_{\un{n}}) b t(\eta_{\un{n}})^* \mid \un{0} \neq \un{n} \in F, b \in \B_{F^c} \}.
\]
To reach contradiction let $g \in \B_{F^c} \cap \B_{(\un{0}, \infty]}^Z$ so that
\[
Q_F g Q_F \in Q_F \B_{(\un{0}, \infty]}^Z Q_F = (0).
\]
Let $\un{k}, \un{w} \perp F$ such that $g = \psi(k_{\un{k}, \un{w}}) + g'$ and $\un{k}$ or $\un{w}$ is minimal with $0 \neq k_{\un{k}, \un{w}} \in \K(X_{\un{w}}, X_{\un{k}})$.
We then compute
\begin{align*}
\K(X_{\un{w}}, X_{\un{k}})
\ni
k_{\un{k}, \un{w}}
=
p_{\un{k}} Q_F \psi(k_{\un{k}, \un{w}}) Q_F p_{\un{w}}
=
p_{\un{k}} Q_F g Q_F p_{\un{w}}
=
0.
\end{align*}
However this gives the contradiction
\[
\|\psi(k_{\un{k}, \un{w}})\| =
\sup_{\un{m} \in \bZ_+^N} \| k_{\un{k}, \un{w}} \otimes \id_{X_\un{m}} \| = 0.
\]

Consequently $\N\T(X)$ acts on the Fock space $\F Z = \sumoplus \{ Z_{\un{n}} \mid \un{n} \in F \}$.
Let us keep track of the inclusions by writing
\[
\iota_{\un{n}} \colon Z_{\un{n}} \to \F Z,
\]
and write $(\wt\pi, \wt t)$ for the representation of $\N\T(X)$ with
\[
\wt\pi(a) \iota_{\un{m}}( t(\xi_{\un{m}}) b') = \iota_{\un{m}}(\pi(a) t(\xi_{\un{m}}) b')
\qand
\wt{t}(t(\xi_{\un{n}})) \iota_{\un{m}}( t(\xi_{\un{m}}) b )
=
\iota_{\un{n}_F + \un{m}}(t(\xi_{\un{n}}) t(\xi_{\un{m}}) b).
\]
Consider the GNS-representation $(H_{\wt\tau}, \rho_{\wt\tau}, x_{\wt\tau})$ related to the constructed $\wt\tau$ and define
\[
(\rho, v) : = (\wt\pi \otimes I, \wt{t} \otimes I) \text{ acting on } \H := \F Z \otimes_{\rho_{\wt\tau}} H_{\wt\tau}.
\]
For every $\umu$ with $\ell(\umu) \in F$ let the vector state $\vphi_{\umu}$ given by
\[
\vphi_{\umu}(f) : = \sca{\iota_{\ell(\umu)}( t(x_{\umu})) \otimes x_{\wt\tau}, (\rho \times v)(f) \big[\iota_{\ell(\umu)}( t(x_{\umu})) \otimes x_{\wt\tau}) \big]}_{\H}
\qfor
f \in \N\T(X).
\]
We use the $\vphi_{\umu}$ to define the functional
\begin{equation}
\Phi_\tau(f) := (c_{\tau, \be}^F)^{-1} \sum \{ e^{- |\umu| \be} \vphi_{\umu}(f) \mid \ell(\umu) \in F \}. 
\end{equation}
For $f = 1_A$ we have
\[
\vphi_{\umu}(\pi(1_A)) = \wt\tau(t(x_{\umu})^* t(x_{\umu})) = \tau(\sca{x_{\umu}, x_{\umu}}),
\]
and therefore
\[
\Phi_\tau(\pi(1_A)) = (c_{\tau, \be}^F)^{-1} \sum \{ e^{- |\umu| \be} \tau(\sca{x_{\umu}, x_{\umu}}) \mid \ell(\umu) \in F \} = 1.
\]
This guarantees that $\Phi_\tau$ is a well defined state on $\N\T(X)$.
In order to show that $\Phi_\tau \in \GEq_\be^F(\N\T(X))$ we require the following properties for the $\vphi_{\umu}$:

\smallskip
\noindent
$\bullet$ (i) For every $\un{m} \in F$ and $\ell(\umu) \in F$ we have that
\begin{align*}
\wt{t}(t(\xi_{\un{m}})^*) \iota_{\ell(\umu)}(t(x_{\umu}))
=
\de_{\un{m}, \un{m} \wedge \ell(\umu)}
\cdot \iota_{\ell(\umu) - \un{m}} (t(\xi_{\un{m}})^* t(x_{\umu})).
\end{align*}

\smallskip
\noindent
$\bullet$ (ii) For every $\un{0} \neq \un{m} \in F$ we have
\[
(\wt\pi \times \wt{t})(Q_F) \iota_{\un{m}}(t(\ze_{\un{m}}) b) = 0,
\]
while $(\wt\pi \times \wt{t})(Q_F) \iota_{\un{0}}(\pi(1)) = \iota_{\un{0}}(\pi(1))$.
Therefore we have
\[
\vphi_{\umu}(Q_F b Q_F)
=
\de_{\ell(\umu), \un{0}} \wt\tau(b) 
\foral b \in \B_{F^c}.
\]

\smallskip
\noindent
$\bullet$ (iii) Recall that $\wt \tau = \wt\tau E$ on $\B_{F^c}$.
For every $\un{n}, \un{m} \in \bZ_+^N$ we get that
\begin{align*}
\vphi_{\umu}(t(\xi_{\un{n}}) t(\eta_{\un{m}})^*)
& = \\
& \hspace{-1.5cm} =
\vphi_{\umu}(t(\xi_{\un{n}_F}) \big[ t(\xi_{\un{n}_{F^c}}) t(\eta_{\un{m}_{F^c}})^* \big] t(\eta_{\un{m}_{F}})^*) \\
& \hspace{-1.5cm} =
\de_{\un{n}_F, \un{n}_F \wedge \ell(\umu)} \cdot \de_{\un{m}_F, \un{m}_F \wedge \ell(\umu)} \cdot \wt \tau E(t(x_{\umu})^* t(\xi_{\un{n}_F}) \big[ t(\xi_{\un{n}_{F^c}}) t(\eta_{\un{m}_{F^c}})^* \big] t(\eta_{\un{m}_{F}})^* t(x_{\umu})) \\
& \hspace{-1.5cm} =
\de_{\un{n}, \un{m}} \cdot \de_{\un{n}_F, \un{n}_F \wedge \ell(\umu)}
\wt\tau(t(x_{\umu})^* t(\xi_{\un{n}}) t(\eta_{\un{m}})^* t(x_{\umu})).
\end{align*}

\begin{lemma}\label{L:technical}
With the aforementioned notation, if $\un{r} \in F$ and $\un{n}, \un{m} \in \bZ_+^N$  then
\begin{equation}
\sum_{\ell(\umu) = \un{r}} \vphi_{\umu}(t(\xi_{\un{n}}) t(\eta_{\un{m}})^*)
=
\de_{\un{n}, \un{m}} \cdot \de_{\un{n}_F, \un{n}_F \wedge \un{r}} \cdot e^{-|\un{n}_{F^c}|\be} \sum_{\ell(\unu) = \un{r} - \un{n}_F} 
\vphi_{\unu} (t(\eta_{\un{m}})^* t(\xi_{\un{n}})).
\end{equation}
Consequently $\Phi_\tau$ attains the stated form with respect to the unit decomposition $x = \{x_{i,j} \mid j=1, \dots, d_i, i=1, \dots, N\}$ and thus it satisfies the KMS-condition.
\end{lemma}

\begin{proof}
Property (iii) for $\vphi_{\umu}$ implies that $\vphi_{\umu} = \vphi_{\umu} E$ and so $\Phi_\tau = \Phi_\tau E$.
Therefore let us consider the case where $\un{n} = \un{m}$.
If $\un{n}_F \not\leq \ell(\umu) \in F$ then again property (iii) yields $\vphi_{\umu}(t(\xi_{\un{n}}) t(\eta_{\un{n}})^*) = 0$.
Now suppose that $\un{n}_F \leq \ell(\umu) \in F$ with $\ell(\umu) = \un{r}$.
Since $\un{r} - \un{n}_F \in F$ we get $(\un{r} - \un{n}_F) \wedge \un{n}_{F^c} = \un{0}$ and so  
\[
t(\eta_{\un{n}})^* t(x_{\umu})
\in
\ol{t(X_{\un{n}_{F^c}})^* t(X_{\un{r} - \un{n}_{F}})}
\subseteq
\ol{t(X_{\un{r} - \un{n}_F}) t(X_{\un{n}_{F^c}})^*}.
\]
Therefore by using the approximate identity of $\psi(\K X_{\un{r} - \un{n}_F})$ on $t(X_{\un{r} - \un{n}_F})$ we get
\[
\sum_{\ell(\unu) = \un{r} - \un{n}_F} t(x_{\unu}) t(x_{\unu})^*  t(\eta_{\un{n}})^* t(x_{\umu})
=
t(\eta_{\un{n}})^* t(x_{\umu}).
\]
Now for $\ell(\unu) = \un{r} - \un{n}_F$ we have $t(\xi_{\un{n}}) t(x_{\unu}) \in t(X_{\un{n} + \un{r} - \un{n}_F}) = \ol{t(X_{\un{r}}) t(X_{\un{n}_{F^c}})}$ and so
\[
\sum_{\ell(\umu) = \un{r}} t(x_{\umu}) t(x_{\umu})^* t(\xi_{\un{n}}) t(x_{\unu})
=
t(\xi_{\un{n}}) t(x_{\unu}).
\]
Finally we have that
\begin{align*}
t(x_{\umu})^* t(\xi_{\un{n}}) t(x_{\unu})
& \in
\ol{t(X_{\un{r} - \un{n}_F})^* t(X_{\un{n}_{F^c}}) t(X_{\un{r} - \un{n}_F})}
\subseteq
\ol{t(X_{\un{n}_{F^c}})} \subseteq \B_{F^c}.
\end{align*}
Therefore we can use the KMS-condition for $\wt\tau$ and $\un{n}_{F^c}$ to obtain
\begin{align*}
\sum_{\ell(\umu) = \un{r}} \vphi_{\umu}(t(\xi_{\un{n}}) t(\eta_{\un{n}})^*)
& = 
\sum_{\ell(\umu) = \un{r}} \sum_{\ell(\unu) = \un{r} - \un{n}_F} 
\wt\tau ( \underbrace{t(x_{\umu})^* t(\xi_{\un{n}}) t(x_{\unu})}_{\in \, t(X_{\un{n}_{F^c}})} \cdot \underbrace{t(x_{\unu})^* t(\eta_{\un{n}})^* t(x_{\umu})}_{\in \, t(X_{\un{n}_{F^c}})^*}) \\
& = 
e^{-|\un{n}_{F^c}|\be} \sum_{\ell(\unu) = \un{r} - \un{n}_F} \sum_{\ell(\umu) = \un{r}} 
\wt\tau (t(x_{\unu})^* t(\eta_{\un{n}})^* \cdot t(x_{\umu}) t(x_{\umu})^* t(\xi_{\un{n}}) t(x_{\unu})) \\
& = 
e^{-|\un{n}_{F^c}|\be} \sum_{\ell(\unu) = \un{r} - \un{n}_F} 
\wt\tau (t(x_{\unu})^* t(\eta_{\un{n}})^* t(\xi_{\un{n}}) t(x_{\unu})) \\
& = 
e^{-|\un{n}_{F^c}|\be} \sum_{\ell(\unu) = \un{r} - \un{n}_F} 
\vphi_{\unu} (t(\eta_{\un{n}})^* t(\xi_{\un{n}})).
\end{align*} 
Consequently we derive
\begin{align*}
c_{\tau, \be}^F \cdot \Phi_\tau(t(\xi_{\un{n}}) t(\eta_{\un{n}})^*)
& = \\
& \hspace{-1.5cm} = 
\sum \{ e^{- |\umu| \be} \vphi_{\umu}(t(\xi_{\un{n}}) t(\eta_{\un{n}})^*) \mid \ell(\umu) \in F \} \\
& \hspace{-1.5cm} =
e^{-|\un{n}_{F^c}| \be} \sum \{ e^{- |\umu| \be} \vphi_{\unu}(t(\eta_{\un{n}})^* t(\xi_{\un{n}})) \mid \ell(\unu) = \ell(\umu) - \un{n}_F, \un{n}_F \leq \ell(\umu) \in F \} \\
& \hspace{-1.5cm} =
e^{-|\un{n}_{F^c}| \be} \sum \{ e^{-(|\umu| + |\un{n}_F|) \be} \vphi_{\umu}(t(\eta_{\un{n}})^* t(\xi_{\un{n}})) \mid \ell(\umu) \in F \} \\
& \hspace{-1.5cm} =
c_{\tau, \be}^F \cdot e^{-|\un{n}| \be} \Phi_\tau(t(\eta_{\un{n}})^* t(\xi_{\un{n}})).
\end{align*}
It follows now that $\Phi_\tau$ both has the required form and that it satisfies the KMS-condition.
\end{proof}

\begin{remark}\label{R:qizero}
We note that the form of $\Phi_\tau$ does not depend on the choice of the decompositions.
Indeed if $y = \{y_{i,j} \mid j=1, \dots, d_i', i=1, \dots, N\}$ defines a second decomposition for $X$ then for any $\un{n} \in \bZ_+^N$ (and not just for $\un{n} \in F$) we get
\begin{align*}
\sum_{\ell(\umu) = \un{n}} \tau(\sca{\eta_{\un{n}} \otimes x_{\umu}, \xi_{\un{n}} \otimes x_{\umu}})  
& = 
\sum_{\ell(\umu) = \un{n}} 
\sum_{\ell(\unu) = \un{n}} 
\tau(\sca{\eta_{\un{n}} \otimes x_{\umu}, \xi_{\un{n}} \otimes y_{\unu}} \sca{y_{\unu}, x_{\umu}}) \\
& = 
\sum_{\ell(\unu) = \un{n}} 
\sum_{\ell(\umu) = \un{n}} 
\tau(\sca{y_{\unu}, x_{\umu}} \sca{x_{\umu}, \sca{\eta_{\un{n}}, \xi_{\un{n}}} y_{\unu}}) \\
& = 
\sum_{\ell(\unu) = \un{n}} 
\tau(\sca{y_{\unu}, \sca{\eta_{\un{n}}, \xi_{\un{n}}} y_{\unu}}) 
 =
\sum_{\ell(\unu) = \un{n}} \tau(\sca{\eta_{\un{n}} \otimes y_{\unu}, \xi_{\un{n}} \otimes y_{\unu}}).
\end{align*}
In particular we have that $\sum_{\ell(\umu) = \un{n}} \tau(\sca{x_{\umu}, x_{\umu}}) = \sum_{\ell(\unu) = \un{n}} \tau(\sca{y_{\unu}, y_{\unu}})$.
Moving one step further let $i \notin F$ and $\un{n} \in F$ so that both families 
\[
\{x_{\umu} x_{i,j} \mid \ell(\umu) = \un{n}, j=1, \dots, d_i\}
\qand
\{x_{i,j} x_{\umu} \mid \ell(\umu) = \un{n}, j=1, \dots, d_i\}
\]
define unit decompositions for $X_{\un{n} + \Bi}$.
Since $\tau \in \Tr_\be^F(A)$ we conclude that
\begin{align*}
\sum_{j=1}^{d_i} \sum_{\ell(\umu) = \un{n}} \tau(\sca{x_{i,j} \otimes x_{\umu}, x_{i,j} \otimes x_{\umu}})
& =
\sum_{j=1}^{d_i} \sum_{\ell(\umu) = \un{n}} \tau(\sca{x_{i,j} x_{\umu}, x_{i,j} x_{\umu}}) \\
& =
\sum_{j=1}^{d_i} \sum_{\ell(\umu) = \un{n}} \tau(\sca{x_{\umu} x_{i,j}, x_{\umu} x_{i,j}}) \\
& =
\sum_{\ell(\umu) = \un{n}} \sum_{j=1}^{d_i} \tau(\sca{x_{i,j}, \sca{x_{\umu}, x_{\umu}} x_{i,j}}) 
 =
\sum_{\ell(\umu) = \un{n}} e^{\be} \tau(\sca{x_{\umu}, x_{\umu}}).
\end{align*}
\end{remark}

\begin{proof}[\bf Proof of Theorem \ref{T:para} (2--4)]
First we show that $\Phi_\tau$ is indeed in $\GEq_\be^F(\N\T(X))$.
By using the KMS-condition and property (ii), for $\ell(\unu) = \un{n} \in F$ we have that
\begin{align*}
c_{\tau, \be}^F \cdot \Phi_\tau(Q_F^{\un{n}})
& =
c_{\tau, \be}^F \sum_{\ell(\unu) = \un{n}} e^{-|\un{n}| \be} \Phi_\tau(Q_F t(x_{\unu})^* t(x_{\unu}) Q_F) \\
& =
\sum_{\ell(\unu) = \un{n}} e^{-|\un{n}| \be} \sum \{ e^{-|\umu| \be} \vphi_{\umu}(Q_F \pi(\sca{x_{\unu}, x_{\unu}}) Q_F) \mid \ell(\umu) \in F \} \\
& =
\sum_{\ell(\unu) = \un{n}} e^{-|\un{n}| \be} \tau(\sca{x_{\unu}, x_{\unu}}).
\end{align*}
Applying in particular for $\un{n} = \un{0}$ we get $\Phi_\tau(Q_F)^{-1} = c_{\tau, \be}^F$.
Summing over all $\un{n} \in F$ we conclude
\[
\sum_{\un{n} \in F} \Phi_\tau(Q_F^{\un{n}}) 
= (c_{\tau, \be}^F)^{-1} \sum \{ e^{-|\unu| \be} \tau(\sca{x_{\unu}, x_{\unu}}) \mid \ell(\unu) \in F\}
=
\Phi_\tau(\pi(1_A))
=
1.
\]
Now let $i \notin F$ and we will show that $\Phi_\tau(P_\Bi) = 1$.
By Remark \ref{R:qizero} we derive that
\begin{align*}
\Phi_\tau(P_\Bi)
& =
\sum_{j=1}^{d_i} \Phi_\tau(t(x_{i,j}) t(x_{i,j})^*) \\
& =
(c_{\tau,\be}^F)^{-1} \sum_{\un{w} \in \bZ_+^N} e^{-(|\un{w}| + 1) \be} \sum_{j=1}^{d_i} \sum_{\ell(\umu) = \un{w}} \tau (\sca{x_{i,j} \otimes x_{\umu}, x_{i,j} \otimes x_{\umu}}) \\
& =
(c_{\tau,\be}^F)^{-1} \sum_{\un{w} \in \bZ_+^N} e^{-(|\un{w}| + 1) \be} e^{\be} \sum_{\ell(\umu) = \un{w}} \tau (\sca{x_{\umu}, x_{\umu}})  
 =
\Phi_\tau(1) = 1.
\end{align*}

Next we show that $\Phi$ is a bijection.
For injectivity notice that property (ii) yields
\[
\Phi_\tau(Q_F \pi(a) Q_F)
=
(c_{\tau, \be}^F)^{-1} \wt\tau(\pi(a))
=
\Phi_\tau(Q_F) \tau(a).
\]
Therefore $\tau$ is uniquely determined by $\Phi_\tau$.
For surjectivity let $\vphi \in \GEq_\be^F(\N\T(X))$.
If $\vphi(Q_F) = 0$ then we would have that
\begin{align*}
\vphi(Q_F^{\un{n}}) 
= e^{-|\un{n}| \be} \sum_{\ell(\umu) = \un{n}} \vphi(Q_F t(x_{\unu})^* t(x_{\unu}) Q_F) 
\leq e^{-|\un{n}| \be} \sum_{\ell(\umu) = \un{n}} \vphi(Q_F) = 0
\end{align*}
which contradicts that $\sum_{\un{n} \in F} \vphi(Q_F^{\un{n}}) = 1$.
Hence we can define the state $\tau_{\vphi}$ on $A$ given by
\begin{equation*}
\tau_\vphi(a) := \vphi(Q_F)^{-1} \vphi(Q_F \pi(a) Q_F) \foral a \in A.
\end{equation*}
It is clear that $\tau_\vphi \in \Tr(A)$ as $\vphi$ is an equilibrium state and $Q_F \in \pi(A)'$.
Since $\vphi(Q_{\Bi}) = 0$ for all $i \notin F$ we get $\vphi|_{\fI_{F^c}} = 0$ and thus $\tau_\vphi|_{\fI_{F^c}} = 0$.
Furthermore we use the KMS-condition on $\vphi$ to get
\begin{align*}
\vphi(Q_F) \sum_{\ell(\umu) = \un{n}} \tau_{\vphi}(\sca{x_{\umu}, x_{\umu}})
& =
e^{|\un{n}| \be} \sum_{\ell(\umu) = \un{n}} \vphi(t(x_{\umu}) Q_F t(x_{\umu})^*)
=
e^{|\un{n}|\be} \vphi(Q_F^{\un{n}}),
\end{align*}
and therefore
\[
c_{\tau, \be}^F
=
\sum \{e^{- |\umu| \be} \tau_{\vphi}(\sca{x_{\umu}, x_{\umu}}) \mid \ell(\umu) \in F \}
=
\vphi(Q_F)^{-1} \sum_{\un{n} \in F} \vphi(Q_F^{\un{n}}) = \vphi(Q_F)^{-1}.
\]
Furthermore as $\vphi(Q_\Bi) = 0$ for $i \notin F$ we get that $\vphi(P_\Bi f) = \vphi(f)$ for all $f \in \N\T(X)$.
Hence by the KMS-condition on $\vphi$ and equation (\ref{eq:p-r}) we derive
\begin{align*}
e^{\be} \tau_\vphi(\pi(a))
& =
e^{\be} \vphi(Q_F)^{-1} \vphi(Q_F \pi(a) Q_F) 
 =
e^{\be} \vphi(Q_F)^{-1} \vphi(P_\Bi Q_F \pi(a) Q_F) \\
& =
\vphi(Q_F)^{-1} \sum_{j=1}^{d_i} \vphi(t(x_{i,j})^* Q_F \pi(a) Q_F t(x_{i,j})) \\
& =
\sum_{j=1}^{d_i} \vphi(Q_F)^{-1} \vphi(Q_F t(x_{i,j})^* \pi(a) t(x_{i,j}) Q_F) 
 =
\sum_{j=1}^{d_i} \tau_\vphi(\sca{x_{i,j}, a x_{i,j}}).
\end{align*}
Consequently $\tau_{\vphi} \in \Tr_\be^F(A)$ and we can now form the induced state $\Phi_{\tau_{\vphi}}$.
We wish to show that $\vphi = \Phi_{\tau_{\vphi}}$ and conclude surjectivity.
Since they are both gauge-invariant, Proposition \ref{P:char} asserts that it suffices to show that they agree on $\pi(A)$.
Since $\sum_{\un{n} \in F} \vphi(Q_F^{\un{n}}) = 1$ we get that $\vphi(\pi(a)) = \sum_{\un{n} \in F} \vphi(Q_F^{\un{n}} \pi(a))$ and we can deduce
\begin{align*}
\vphi(\pi(a))
& =
\lim_m \sum_{k=0}^m \bigg[ \sum \{ \vphi(t(x_{\umu}) Q_F t(x_{\umu})^* \pi(a)) \mid \ell(\umu) \in F, |\umu| = k \} \bigg] \\
& =
\lim_m \sum_{k=0}^m \bigg[ \sum \{ e^{-k \be} \vphi(Q_F t(x_{\umu})^* \pi(a) t(x_{\umu}) Q_F ) \mid \ell(\umu) \in F, |\umu| = k \} \bigg]\\
& =
\vphi(Q_F) \lim_m \sum_{k=0}^m \bigg[ \sum \{ e^{- k \be} \tau_{\vphi}(\sca{x_{\umu}, a x_{\umu}}) \mid \ell(\umu) \in F, |\umu| = k \} \bigg]\\
& =
(c_{\tau, \be}^F)^{-1} \sum \{e^{- |\umu| \be} \tau_\vphi(\sca{x_{\umu}, a x_{\umu}}) \mid \ell(\umu) \in F \}
=
\Phi_{\tau_{\vphi}}(\pi(a)).
\end{align*}

The last part on convexity follows in the same way as in \cite[Corollary 6.3]{Kak17} and it is omitted.
\end{proof}

\begin{remark}
The inverse correspondence $(\Phi^F)^{-1} \colon \GEq_\be^F(\N\T(X)) \to \Tr_\be^F(A)$ is continuous for all $F$.
This follows by the way we retrieve $\tau_\vphi$ when $F \neq \mt$, and because $\Phi(\tau)$ is an extension of $\tau$ when $F = \mt$.
Therefore $\Phi$ is a homeomoprhism when $\Tr_\be^F(A)$ is closed in $\Tr(A)$.
On the other hand $\Phi^\infty$ is always a weak*-homeomorphism as $\Avt_\be(A)$ is weak*-closed.
\end{remark}

\section{Relative Nica-Pimnser algebras}\label{S:rel}

We next apply the obtained parametrization to relative Nica-Pimsner algebras.
We say that a family $\{I_F \mid \mt \neq F \subseteq \{1, \dots, N\}\}$ of ideals of $A$ is a \emph{lattice of $\perp$-invariant ideals} if:
\begin{enumerate}
\item $I_F \subseteq I_{F'}$ when $F \subseteq F'$; and
\item $\sca{X_{\un{n}}, I_F X_{\un{n}}} \subseteq I_F$ when $\un{n} \perp F$.
\end{enumerate}
We then define the $\bT^N$-equivariant ideal
\begin{align*}
\K_I 
& := 
\sca{ \pi(a) Q_F \mid a \in I_\F, F \subseteq \{1, \dots, N\} } \\
&=
\ol{\spn}\{ t(X_{\un{n}}) \pi(a) Q_F t(X_{\un{m}})^* \mid \un{n}, \un{m} \in \bZ_+^N, a \in I_F,\mt \neq F \subseteq \{1, \dots, N\} \}.
\end{align*}
The fact that the ideal $\K_I$ attains this form as a linear space follows in the same way as in \cite[Proposition 4.6]{DK18}.
The only that is required is $\perp$-invariance and that $I_F \subseteq I_F'$ when $F \subseteq F'$.

\begin{definition}
Let $X$ be a product system of finite rank.
Let $\{I_F \mid \mt \neq F \subseteq \{1, \dots, N\}\}$ be a lattice of $\perp$-invariant ideals of $A$.
We define the \emph{relative Nica-Pimsner algebra} $\N\O(I, X)$ be the quotient of $\N\T(X)$ by the $\bT^N$-equivariant ideal $\K_I$.
\end{definition}

Similar to $\N\T(X)$, we use the projections $Q_F + \K_I$ to make sense of the $\Eq_\be^F(\N\O(I, X))$ simplices.
Since $\N\O(A,X)$ is a quotient of $\N\O(I, X)$, we have the same simplex of infinite-type states for $\N\O(I, X)$. 
The only restriction for the $F$-parametrization is to have the states in $\Tr_\be^F(A)$ to annihilate \emph{both} $\fI_{F^c}$ and $I_{F}$.

\begin{theorem}\label{T:para 2}
Let $X$ be a product system of finite rank over $A$ and $\be > 0$.
Let $\{I_F \mid \mt \neq F \subseteq \{1, \dots, N\}\}$ be a lattice of $\perp$-invariant ideals of $A$.
Then:

\noindent
$(1)$ Every state in $\GEq_\be(\N\O(I, X))$ admits a unique convex decomposition over the $F$-simplices $\GEq_\be^F(\N\O(I, X))$.

\noindent
$(2)$ For $F = \mt$ there is a bijection
\[
\Phi^\infty \colon \{\tau \in \Avt_\be(A) \mid \tau|_{\fI_{\{1, \dots, N\}}} = 0\}
\to
\GEq_{\be}^\infty(\N\O(I, X))
\, \text{ such that } \,
\Phi^\infty_{\tau} \pi = \tau.
\]
$(3)$ For $F \neq \mt$ there is a bijection
\[
\Phi^F \colon \{ \tau \in \Tr_\be^F(A) \mid \tau|_{\fI_{F^c} + I_{F}} = 0 \} \to \GEq_\be^F(\N\O(I, X)),
\]
that factors through $q_{\K_I} \colon \N\T(X) \to \N\O(I, X)$, arising from Theorem \ref{T:para}.\\
$(4)$ Every $\Phi^F$ for $F\neq \mt$ respects convex combinations and thus preserves the extreme points of the simplices.
\end{theorem}

\begin{proof}
Let us write $\K \equiv \K_I$ for simplicity.
Since $q_\K \colon \N\T(X) \to \N\O(I, X)$ is gauge-invariant we have that $\vphi' \in \GEq_\be^F(\N\O(I, X))$ if and only if there exists a $\vphi \in \GEq_\be^F(\N\T(X))$ such that $\vphi' q_\K = \vphi$.
The case $F = \mt$ gives the same infinite-type states.
Hence we consider the case where $F \neq \mt$.

If $\vphi' \in \GEq_\be^F(\N\O(I, X))$ then $\vphi = \vphi' q_\K \in \GEq_\be^F(\N\T(X))$ defines a unique state $\tau_{\vphi} \in \Tr_\be^F(A)$ such that $\tau_{\vphi}|_{\fI_{F^c}} = 0$ from Theorem \ref{T:para}.
We need to show that $\tau_{\vphi}|_{I_{F}} = 0$.
Indeed if $a \in I_{F}$ then $\pi(a)Q_F \in \K$ and so
\[
\tau_{\vphi}(a) = [\vphi' q_\K(Q_F)]^{-1} \vphi' q_\K(Q_F \pi(a) Q_F) = 0.
\]
Conversely let $\tau \in \Tr_\be^F(A)$ that annihilates both $\fI_{F^c}$ and $I_{F}$.
Construct the induced $\vphi \equiv \Phi_\tau^F$ from Theorem \ref{T:para}.
We need to show that $\vphi|_{\K} = 0$ and this will induce the required $\vphi' \in \GEq_\be^F(\N\O(I, X))$.
By using the KMS-condition as in Proposition \ref{P:states ideals} it suffices to show that
\[
\vphi(\pi(a) Q_C) = 0 \foral a \in I_{C} \textup{ and } \mt \neq C \subseteq \{1, \dots, N\}.
\]
First suppose that $C \cap F^c \neq \mt$.
By construction we already have that $\vphi(Q_{\Bi}) = 0$ for any $i \in C \setminus F$.
Hence $\vphi(Q_C) = 0$ and therefore 
\[
\vphi(\pi(a) Q_C) = 0 \foral a \in I_C.
\]
Next let $C \subseteq F$.
For any $D \subseteq C$ we get
\begin{align*}
\vphi(\pi(a) P_D)
& =
\sum_{\ell(\unu) = \un{1}_D} e^{-|D| \be} \vphi(t(x_{\unu})^* \pi(a) t(x_{\unu})) \\
& =
(c_{\tau, \be}^{F})^{-1} \sum \{ e^{-(|\umu| + |D|) \be} \tau (\sca{x_{\unu} \otimes x_{\umu}, a x_{\unu} \otimes x_{\umu}}) \mid \ell(\unu) = \un{1}_D, \ell(\umu) \in F \}.
\end{align*}
Thus by using the alternating sums we get
\begin{align*}
\vphi(\pi(a) Q_C)
& = 
\sum_{D \subseteq C} (-1)^{|D|} \vphi(\pi(a) P_D) \\
& =
(c_{\tau, \be}^{F})^{-1} \sum_{D \subseteq C} (-1)^{|D|} \bigg[ \sum \{ e^{-|\umu| \be} \tau(\sca{x_{\umu}, a x_{\umu}}) \mid \un{1}_D \leq \ell(\umu) \in F \} \bigg] \\
& =
(c_{\tau, \be}^F)^{-1} \sum \{ e^{-|\umu| \be} \tau(\sca{x_{\umu}, a x_{\umu}}) \mid \ell(\umu) \perp C \}. 
\end{align*}
For the last equality, if $\ell(\umu) \neq \un{0}$ with $\ell(\umu) \cap C \neq \mt$ then
\[
\{D \subseteq C \mid \un{1}_D \leq \ell(\umu)\}
=
\{D \mid D \subseteq \supp \ell(\umu) \cap C\}.
\]
Thus the part corresponding to this $\ell(\umu)$ contributes zero to the sum, namely
\[
\sum_{D \subseteq C, \un{1}_D \leq \supp \ell(\umu)} (-1)^{|D|} e^{-|\umu| \be} \tau(\sca{x_{\umu}, a x_{\umu}})
=
e^{-|\umu| \be} \tau(\sca{x_{\umu}, a x_{\umu}}) \sum_{D \subseteq \supp \ell(\umu) \cap C} (-1)^{|D|} 
=
0.
\]
Hence the only part that survives is for $D = \mt$ and $\ell(\umu) \perp C$.
However $a \in I_{C}$ and so for $\ell(\umu) \perp C$ we get that
\[
\sca{x_{\umu}, a x_{\umu}} \in I_{C} \subseteq I_{F},
\]
giving $\tau(\sca{x_{\umu}, a x_{\umu}}) = 0$.
Therefore we conclude that $\vphi(\pi(a) Q_C) = 0$ also in this case, and the proof is complete.
\end{proof}

\section{Extending to the full parametrizations}\label{S:non gi}

We may now obtain the parametrizations of the full simplices of the $F$-equilibrium states.
First we note that $\vphi$ is of infinite type if and only if $\vphi(Q_{\Bi} f) = 0$ for all $f \in \N\T(X)$ and $i = 1, \dots, N$, if and only if
\[
\vphi(f) = \vphi(P_{\Bi} f) = e^{-\be} \sum_{j=1}^{d_i} \vphi(t(x_{i,j})^* f t(x_{i,j})) 
\foral
f \in \N\T(X),
i=1, \dots, N.
\]
This the furthest we can go at this generality and thus we now pass to the non-infinite types.
Now each $F$-part will depend on the product system we constructed for the proof of Theorem \ref{T:para}.

\begin{theorem}\label{T:para 3}
Let $X$ be a product system of finite rank over $A$ and $\be > 0$.
For $\mt \neq F \subseteq \{1, \dots, N\}$ let the $F$-product system $Z$ over $\B_{F^c}$ constructed concretely in $\N\T(X)$ by
\[
Z_{\un{n}} = \ol{t(X_{\un{n}}) \B_{F^c}}
\; \text{ for } \;
\un{n} \in F
\qand
\B_{F^c} = \ol{\spn}\{t(X_{\un{k}}) t(X_{\un{w}})^* \mid \un{k}, \un{w} \perp F\}.
\]
Then the parametrization of Theorem \ref{T:para} lifts to a parametrization
\[
\{\wt\tau \in \Eq_\be(\B_{F^c}) \mid \wt\tau\pi \in \Tr_\be^F(A), \wt\tau|_{\fI_{F^c}'} = 0\}
\to 
\Eq_\be^F(\N\T(X)),
\]
where
\[
\fI_{F^c}' := \ker\{ \B_{F^c} \to \N\O(F^c, A,X)\}.
\]
Likewise the parametrization of the relative Cuntz-Nica-Pimsner algebras from Theorem \ref{T:para 2} lifts to
\[
\{\wt\tau \in \Eq_\be(\B_{F^c}) \mid \wt\tau\pi \in \Tr_\be^F(a), \wt\tau|_{\fI_{F^c}' + I_F'} = 0\}
\to 
\Eq_\be^F(\N\O(I, X)),
\]
where
\[
I_{F}' := \ol{\spn}\{t(X_{\un{k}}) \pi(I_F) t(X_{\un{w}})^* \mid \un{k}, \un{w} \perp F\}.
\]
\end{theorem}

\begin{proof}
The construction is the same with that of Theorem \ref{T:para} and Theorem \ref{T:para 2} as long as we substitute $A$ with $\B_{F^c}$ and we verify the relevant points.
By Proposition \ref{P:gi} we can show that
\[
\vphi(t(\xi_{\un{n}}) t(\xi_{\un{k}}) t(\eta_{\un{w}})^* t(\eta_{\un{m}})^*) = \de_{\un{n}, \un{m}} \vphi(t(\xi_{\un{n}}) t(\xi_{\un{k}}) t(\eta_{\un{w}})^* t(\eta_{\un{m}})^*)
\; \text{ for all } \;
\un{n}, \un{m} \in F,
\un{k}, \un{w} \perp F.
\]
Hence by taking linear combinations and limits we deduce that
\[
\vphi(t(\xi_{\un{n}}) b t(\eta_{\un{m}})^*) = \de_{\un{n}, \un{m}} \vphi(t(\xi_{\un{n}}) b t(\eta_{\un{m}})^*)
\; \text{ for all } \;
\un{n}, \un{m} \in F 
\; \text{ and } \; 
b \in \B_{F^c}.
\]

Let $\wt\tau$ such that $\wt\tau\pi \in \Tr_\be^F(A)$ and $\wt\tau|_{\fI_{F^c}'} = 0$.
We need to check that the constructed $\vphi := \Phi_{\wt\tau}$ of Theorem \ref{T:para} is an equilibrium state.
For convenience let us say that an element $b \in \B_{F^c}$ is \emph{simple} if $b \in t(X_{\un{k}}) t(X_{\un{w}})^*$ for some $\un{k}, \un{w} \perp F$.
In this case we say that $b$ has degree $\deg(b) = |\un{k} - \un{w}|$.
We may proceed in the same way as in Lemma \ref{L:technical} for $\un{n}, \un{m} \in F$ and $b_1, b_2$ simple elements to obtain
\begin{equation}\label{eq:b}
\vphi(t(\xi_{\un{n}}) b_1 b_2 t(\eta_{\un{m}})^*)
=
\de_{\un{n}, \un{m}} e^{-|\un{n}_F| \be} e^{-\deg(b_1) \be} \vphi(b_2 t(\eta_{\un{m}})^* t(\xi_{\un{n}}) b_1).
\end{equation}
In particular $\vphi$ inherits the KMS condition on $\B_{F^c}$ from $\wt\tau$.
We may extend this relation by taking linear combinations and limits.

We shall show that $\vphi$ satisfies the KMS condition.
Let $\un{n}, \un{k}, \un{w} \in \bZ_+^N$ and we can apply for
\[
t(\xi_{\un{n}}) t(\xi_{\un{k}})
\in
\ol{t(X_{\un{n}_F + \un{k}_F}) \cdot t(X_{\un{n}_{F^c} + \un{k}_{F^c}})}
\qand
t(\eta_{\un{w}}) \in \ol{t(X_{\un{w}_F}) \cdot t(X_{\un{w}_{F^c}})},
\]
to get that
\begin{align*}
\vphi(t(\xi_{\un{n}}) t(\xi_{\un{k}}) t(\eta_{\un{w}})^*)
& =
\de_{\un{n}_F + \un{k}_F, \un{w}_F} e^{-|\un{n} + \un{k}| \be} \vphi(t(\eta_{\un{w}})^* t(\xi_{\un{n}}) t(\xi_{\un{k}})).
\end{align*}
As long as $\un{w}_F = \un{n}_F + \un{k}_F \geq \un{n}_F$ we have
\begin{align*}
t(\eta_{\un{w}})^* t(\xi_{\un{n}})
& \in
\ol{t(X_{\un{w}_F - \un{n}_F})^* t(X_{\un{w}_{F^c}})^* t(X_{\un{n}_{F^c}})} \\
& \subseteq
\ol{t(X_{\un{w}_F - \un{n}_F})^* t(X_{\un{n}_{F^c} - \un{n}_{F^c} \wedge \un{w}_{F^c}}) t(X_{\un{w}_{F^c} - \un{n}_{F^c} \wedge \un{w}_{F^c}})^*} \\
& \subseteq
\ol{t(X_{\un{n}_{F^c} - \un{n}_{F^c} \wedge \un{w}_{F^c}}) t(X_{\un{w}_{F^c} - \un{n}_{F^c} \wedge \un{w}_{F^c}})^* t(X_{\un{w}_F - \un{n}_F})^*}.
\end{align*}
Now we apply to the adjoint of equation (\ref{eq:b}) and thus get
\begin{align*}
\vphi(t(\xi_{\un{n}}) t(\xi_{\un{k}}) t(\eta_{\un{w}})^*)
& =
\de_{\un{n}_F + \un{k}_F, \un{w}_F} e^{-|\un{n} + \un{k}| \be} \vphi(t(\eta_{\un{w}})^* t(\xi_{\un{n}}) \cdot t(\xi_{\un{k}})) \\
& =
\de_{\un{n}_F + \un{k}_F, \un{w}_F} e^{-|\un{n} + \un{k}| \be} e^{-|\un{k}| \be} \vphi(t(\xi_{\un{k}}) t(\eta_{\un{w}})^* t(\xi_{\un{n}})) \\
& =
\de_{\un{n}_F + \un{k}_F, \un{w}_F} e^{-|\un{n}| \be} \vphi(t(\xi_{\un{k}}) t(\eta_{\un{w}})^* t(\xi_{\un{n}})).
\end{align*}
Notice that
\[
t(\xi_{\un{k}}) t(\eta_{\un{w}})^* t(\xi_{\un{n}})
\in
\ol{t(X_{\un{k} + \un{n} - \un{n} \wedge \un{w}}) t(X_{\un{w} - \un{n} \wedge \un{w}})^*}
\]
and hence
\begin{align*}
\vphi(t(\xi_{\un{k}}) t(\eta_{\un{w}})^* t(\xi_{\un{n}}))
& =
\de_{(\un{k} + \un{n} - \un{n} \wedge \un{w})_F, (\un{w} - \un{n} \wedge \un{w})_F} \vphi(t(\xi_{\un{k}}) t(\eta_{\un{w}})^* t(\xi_{\un{n}})) \\
& =
\de_{\un{n}_F + \un{k}_F, \un{w}_F} \vphi(t(\xi_{\un{k}}) t(\eta_{\un{w}})^* t(\xi_{\un{n}})),
\end{align*}
giving that
\[
\vphi(t(\xi_{\un{n}}) t(\xi_{\un{k}}) t(\eta_{\un{w}})^*)
=
e^{-|\un{n}| \be} \vphi(t(\xi_{\un{k}}) t(\eta_{\un{w}})^* t(\xi_{\un{n}})).
\]
Proposition \ref{P:char} then implies that $\vphi$ satisfies the KMS condition on $\N\T(X)$.

Injectivity of the parametrization is immediate as $\tau$ is completely identified by $\Phi(\tau)$.
Notice that $Q_F$ commutes with $\B_{F^c}$ due to equation (\ref{eq:p-r}).
Equation (\ref{eq:b}) implies that two equilibrium states that are gauge invariant along $F$ coincide if and only if they coincide on $\B_{F^c}$.
This yields surjectivity of the parametrization.

To see that the parametrization is inherited by $\N\O(I,X)$ we just need to check that $I_F'$ is the ideal generated by $\pi(I_F)$ in $\B_{F^c}$ and that it is $\perp$-invariant.
We refer to \cite[Proposition 4.6]{DK18} which gives that $\pi(I_F) t(X_{\un{m}}) \subseteq t(X_{\un{m}}) \pi(I_F)$ whenever $\un{m} \perp F$.
Hence indeed $I_F' = \ol{\B_{F^c} \pi(I_F) \B_{F^c}}$.
For invariance let $\un{m}, \un{k}, \un{w} \perp F$.
Then
\begin{align*}
t(X_{\un{m}})^* t(X_{\un{k}}) \pi(I_F) t(X_{\un{w}})^* t(X_{\un{m}})
& \subseteq
t(X_{\un{k} - \un{k} \wedge \un{m}}) t(X_{\un{m} - \un{k} \wedge \un{m}})^* \pi(I_F) t(X_{\un{m} - \un{m} \wedge \un{w}}) t(X_{\un{w} - \un{m} \wedge \un{w}})^* \\
& \subseteq
t(X_{\un{k} - \un{k} \wedge \un{m}}) t(X_{\un{x}})^* \pi(I_F) t(X_{\un{y}}) t(X_{\un{w} - \un{m} \wedge \un{w}})^* 
\subseteq I_F',
\end{align*}
where we have set
\[
\un{x} = \un{m} - \un{k} \wedge \un{m} - \un{r}, \quad
\un{y} = \un{m} - \un{m} \wedge \un{w} - \un{r}, \quad
\un{r} = (\un{m} - \un{k} \wedge \un{m}) \wedge (\un{m} - \un{m} \wedge \un{w}),
\]
and we have used invariance of $I_F$ for $\un{r} \perp F$, so that $t(X_{\un{r}})^* \pi(I_F) t(X_{\un{r}}) \subseteq \pi(I_F)$.
\end{proof}

\begin{remark}\label{R:ngi}
Recently, Christensen \cite{Chr18} parametrized the equilibrium states for higher rank graphs, without using any of the assumptions of \cite{FHR18}.
The decomposition/parametriza\-tion for the Toeplitz-Nica-Pimsner algebra that he obtains has the same form with Theorem \ref{T:con} and Theorem \ref{T:para}.
Nevertheless he achieves more by considering the action to be weighted on different fibers.

We shall show here that our main theorems accommodate this setting with a small calibration.
Fix $\un{s} = (s_1, \dots, s_N) \geq \un{0}$ and let the action
\[
\si^{(s)} \colon r \to \Aut(\N\T(X)) : r \mapsto \ga_{(\exp(i s_1 r), \dots, \exp(i s_N r))}.
\]
Then we obtain
\[
\si_r^{(s)}(t(\xi_{\un{n}}) t(\eta_{\un{m}})^*) = e^{- i \sca{\un{n} - \un{m}, \un{s}} r} t(\xi_{\un{n}}) t(\eta_{\un{m}})^*
\qfor
\sca{\un{w}, \un{s}} = \sum_{i=1}^N w_i s_i \foral \un{w} \in \bZ^N.
\]
In particular we have $|\umu| = \sca{\ell(\umu), \un{1}}$.
Notice that the projections $P_F, Q_F$ are invariant under this new action $\si^{(s)}$.

If $s_i \neq 0$ for all $i \in \{1, \dots, N\}$ then our arguments apply verbatim by substituting $|\cdot |$ with $\sca{\cdot, \un{s}}$; the only that is required is linearity of $\sca{\cdot, \un{s}}$.
Now suppose that $s_i = 0$ for some $i$, and without loss of generality let $G = \{i \mid s_i = 0\}$.
By using the Fock representation let
\[
\B_G := \ol{\spn}\{t(\xi_{\un{k}}) t(\eta_{\un{w}})^* \mid \un{k}, \un{w} \in G\}
\qand
X'_{\un{n}} := \ol{t(X_{\un{n}}) \B_G}
\foral
\un{n} \perp G,
\]
and consider the product system $X' = \{X_\un{n}'\}_{\un{n} \perp G}$.
Now we may apply Theorem \ref{T:con} and Theorem \ref{T:para} to obtain the $F$-parts for $\N\T(X')$.
However as shown in the proof of Theorem \ref{T:para} we get that $\N\T(X') = \N\T(X)$ in a canonical way, i.e.,  the $F$-projections are the same for any $F \perp G$.
Consequently if $F \perp G$ then $\N\O(F^c, \B_G, X') = \N\O(F^c, A, X)$.
Thus the required parametrization for $\N\T(X)$ is the same with Theorem \ref{T:para} with the only difference that we need to consider traces on $\B_G$ rather than just on $A$.
From there we can achieve the full parametrizations as in Theorem \ref{T:para 3}.
\end{remark}

\section{Entropy}\label{S:entropy}

Now we turn our attention to the entropy of the product system.
This will be connected to existence of equilibrium states and tracial states in $\Tr_\be^F(A)$.

\begin{definition}
Let $X$ be a product system of finite rank over $A$.
Let $x = \{x_{i,j} \mid j=1, \dots, d_i, i=1, \dots, N\}$ be a unit decomposition of $X$.
For every $\tau \in \Tr(A)$ we define the \emph{tracial entropy}
\begin{equation*}
h_X^\tau := \limsup_k \frac{1}{k} \log \bigg[ \sum_{|\umu| = k} \tau(\sca{x_{\umu}, x_{\umu}}) \bigg].
\end{equation*}
In particular for $F \subseteq \{1, \dots, N\}$ we define the \emph{$F$-tracial entropy}
\begin{equation*}
h_X^{\tau, F} := \limsup_k \frac{1}{k} \log \bigg[ \sum_{|\umu| = k, \ell(\umu) \in F} \tau(\sca{x_{\umu}, x_{\umu}}) \bigg].
\end{equation*}
We define the \emph{entropy} of the unit decomposition $x$ by
\begin{equation}
h_X^x := \limsup_k \frac{1}{k} \log \| \sum_{|\umu| = k} \sca{x_{\umu}, x_{\umu}}\|_{A},
\end{equation}
and then the \emph{strong entropy} of $X$ by
\begin{equation}
h_X^s := \inf\{ h_X^x \mid \text{ $x$ is a unit decomposition of $X$}\}.
\end{equation}
Likewise for $F \subseteq \{1, \dots, N\}$ we define the \emph{$F$-entropy} of a unit decomposition $x$ by
\begin{equation}
h_X^{x, F} := \limsup_k \frac{1}{k} \log \| \sum_{|\umu| = k, \ell(\umu) \in F} \sca{x_{\umu}, x_{\umu}}\|_{A},
\end{equation}
and the \emph{$F$-strong entropy} of $X$ by
\begin{equation}
h_X^{s, F} := \inf\{ h_X^{x,F} \mid \text{ $x$ is a unit decomposition of $X$}\}.
\end{equation}
Moreover we define the \emph{entropy} of $X$ by
\begin{equation*}
h_X := \inf\{\be >0 \mid \Eq_\be(\N\T(X)) \neq \mt \}.
\end{equation*}
\end{definition}

\begin{remark}
Using Remark \ref{R:qizero} we can see that the tracial entropies do not depend on the choice of the decomposition.
Hence, if $A$ is abelian then $h_X^{s,F} = h_X^{x,F}$ for every unit decomposition $x$.
It is also clear that if $\tau \in \Tr(A)$ with $h_X^{\tau, F} < \be$ then $c_{\tau, \be}^F <\infty$.

Furthermore the $\limsup$ for $h_X^x$ is actually a limit of a decreasing sequence.
Indeed for every $k' \leq k$ we get that
\begin{align*}
\sum_{|\umu| = k} \sca{x_{\umu}, x_{\umu}}
& =
\sum_{|\umu| = k - k'} \sca{x_{\umu}, \big( \sum_{|\unu| = k'} \sca{x_{\unu}, x_{\unu}} \big) x_{\umu}} 
 \leq
\| \sum_{|\unu| = k'} \sca{x_{\unu}, x_{\unu}} \|_A \sum_{|\umu| = k - k'} \sca{x_{\umu}, x_{\umu}},
\end{align*}
showing that the sequence $\big(\| \sum_{|\umu| = k} \sca{x_{\umu}, x_{\umu}}\|_A\big)_k$ is submultiplicative.
\end{remark}

\begin{proposition}\label{P:entropy1}
Let $X$ be a product system of finite rank over $A$.
Let $x = \{x_{i,j} \mid i=1, \dots, N, j=1, \dots, d_i\}$ be a unit decomposition of $X$.
Then for any $\tau \in \Tr(A)$ we have
\[
h_X^\tau \leq h_X^x = \max \{ h_X^{x, i} \mid i =1, \dots, N \} \leq \max \{ \log d_i \mid i =1, \dots, N \},
\]
so that
\[
h_X^\tau \leq h_X^s = \max \{ h_X^{s, i} \mid i =1, \dots, N \} \leq \max \{ \log d_i \mid i =1, \dots, N \},
\]
and likewise for their $F$-analogues.
\end{proposition}

\begin{proof}
It suffices to show the first part.
It is clear that $h_X^\tau \leq h_X^x$ since
\[
\sum_{\ell(\umu) = \un{n}} \tau(\sca{x_{\umu}, x_{\umu}}) \leq \| \sum_{\ell(\umu) = \un{n}} \sca{x_{\umu}, x_{\umu}} \|_A.
\]
Moreover it follows directly that
\[
h_X^{x, i} = \limsup_k \frac{1}{k} \log \|\sum_{|\mu_i| = k} \sca{x_{i, \mu_i}, x_{i, \mu_i}} \|_A
\leq
\limsup_k \frac{1}{k} \log \sum_{|\mu_i| = k} 1
= \log d_i.
\]
At the same time we have that
\[
h_X^{x, F}
=
\limsup_k \frac{1}{k} \log \|\sum_{|\umu| = k, \ell(\umu) \in F} \sca{x_{\umu}, x_{\umu}} \|_A
\leq
\limsup_k \frac{1}{k} \log \|\sum_{|\umu| = k} \sca{x_{\umu}, x_{\umu}} \|_A
=
h_X^x.
\]
Now set $c > 0$ such that
\[
\log c:= \max \{ h_X^{x,i} \mid i=1, \dots, N \} \leq h_X^x.
\]
For $\eps >0$ let $k_i$ such that
\[
\| \sum_{|\mu_i| = k} \sca{x_{i, \mu_i}, x_{i,\mu_i}} \|_A \leq (c + \eps)^k
\foral
k \geq k_i.
\]
Set $k_0 := \max\{k_1, \dots, k_N\}$.
We claim that for every $F$ there exists a polynomial $p_F$ (of degree $|F| - 1$) with positive co-efficients such that
\begin{equation}\label{eq:claim}
\| \sum_{|\umu| = k, \ell(\umu) \in F} \sca{x_{\umu}, x_{\umu}} \|_{A} \leq p_F(k) (c + \eps)^k, 
\foral k \geq |N| \cdot k_0.
\end{equation}
Applying then for $F = \{1, \dots, N\}$ and the induced polynomial
\[
p(k) \equiv p_{\{1, \dots, N\}}(k) = a_{N-1} k^{N-1} + \cdots + a_0
\]
will give that
\[
\log c \leq
h_X^x
\leq 
\lim_k \frac{1}{k} \log[p(k) (c+\eps)^k]
=
\lim_k \frac{1}{k} \log[k^{N-1} (c+\eps)^k]
=
\log(c+\eps),
\]
since
\[
a_{N-1} k^{N-1} \leq p(k) \leq (a_{N-1} + \dots + a_0) k^{N-1}.
\]
Taking $\eps \to 0$ yields the required $h_X^x = \log c$.
To prove the claim, recall that every $h_X^{x,F}$ is a limit of a decreasing sequence.
Set
\[
M := \sup \{ \|\sum_{|\umu| = k, \ell(\umu) \in F} \sca{x_{\umu}, x_{\umu}} \|_{A}^{1/k} \mid F \subseteq \{1, \dots, N\}, k \in \bZ_+ \}.
\]
We have already chosen $k_0$ so that the claim holds for $F = \{i\}$ with $p_{\{i\}} = 1$.
Let $F = \{i, j\}$ and $k \geq |N| \cdot k_0$.
Then $k - n \geq k_0$ for every $n \leq k_0$ and we get 
\begin{align*}
\sum_{n=0}^{k_0 - 1} 
\| \sum_{|\mu_i| = n} \sca{x_{i,\mu_i}, x_{i, \mu_i}} \|_{A} \cdot 
\| \sum_{|\mu_j| = k - n} \sca{x_{j,\mu_j}, x_{j, \mu_j}} \|_{A}
& \leq \\
& \hspace{-3cm} \leq
\sum_{n=0}^{k_0 - 1} M^n (c + \eps)^{k - n}
=
\big[\sum_{n=0}^{k_0 - 1} M^n (c + \eps)^{-n} \big] (c + \eps)^k.
\end{align*}
By a change of variable we derive the symmetrical
\begin{align*}
\sum_{n=k - k_0 + 1}^{k} 
\| \sum_{|\mu_i| = n} \sca{x_{i,\mu_i}, x_{i, \mu_i}} \|_{A} \cdot 
\| \sum_{|\mu_j| = k - n} \sca{x_{j,\mu_j}, x_{j, \mu_j}} \|_{A}
& \leq \\
& \hspace{-4cm} \leq
\sum_{n=0}^{k_0 - 1} 
\| \sum_{|\mu_i| = k - n} \sca{x_{i,\mu_i}, x_{i, \mu_i}} \|_{A} \cdot 
\| \sum_{|\mu_j| = n} \sca{x_{j,\mu_j}, x_{j, \mu_j}} \|_{A} \\
& \hspace{-4cm} \leq
\big[\sum_{n= 0}^{k_0 - 1} M^n (c + \eps)^{-n}\big] (c + \eps)^k.
\end{align*}
Moreover for the part in-between we have
\begin{align*}
\sum_{n=k_0}^{k - k_0} 
\| \sum_{|\mu_i| = n} \sca{x_{i,\mu_i}, x_{i, \mu_i}} \|_{A} \cdot 
\| \sum_{|\mu_j| = k - n} \sca{x_{j,\mu_j}, x_{j, \mu_j}} \|_{A}
& \leq \\
& \hspace{-4cm} \leq
\sum_{n=k_0}^{k - k_0} (c + \eps)^n (c + \eps)^{k-n} 
\leq 2k (c + \eps)^k.
\end{align*}
Putting these together and using submultiplicativity for the terms of $h_X^{x, \{i,j\}}$ we derive the required
\begin{align*}
\|\sum_{|\mu_i| + |\mu_j| = k} \sca{x_{i,\mu_i} \otimes x_{j, \mu_j}, x_{i,\mu_i}\otimes x_{j, \mu_j}} \|_{A}
& \leq \\
& \hspace{-4cm} \leq
\sum_{n=0}^k \|\sum_{|\mu_i| = n, |\mu_j| = k - n} \sca{x_{i,\mu_i} \otimes x_{j, \mu_j}, x_{i,\mu_i}\otimes x_{j, \mu_j}} \|_{A} \\
& \hspace{-4cm} \leq
\sum_{n=0}^{k_0 - 1} 
\| \sum_{|\mu_i| = n} \sca{x_{i,\mu_i}, x_{i, \mu_i}} \|_{A} \cdot 
\| \sum_{|\mu_j| = k - n} \sca{x_{j,\mu_j}, x_{j, \mu_j}} \|_{A} + \\
& \hspace{-3cm} +
\sum_{n=k_0}^{k - k_0} 
\| \sum_{|\mu_i| = n} \sca{x_{i,\mu_i}, x_{i, \mu_i}} \|_{A} \cdot 
\| \sum_{|\mu_j| = k - n} \sca{x_{j,\mu_j}, x_{j, \mu_j}} \|_{A} + \\
& \hspace{-2cm} +
\sum_{n= k - k_0 + 1}^{k} 
\| \sum_{|\mu_i| = n} \sca{x_{i,\mu_i}, x_{i, \mu_i}} \|_{A} \cdot 
\| \sum_{|\mu_j| = k - n} \sca{x_{j,\mu_j}, x_{j, \mu_j}} \|_{A} \\
& \hspace{-4cm} \leq
\big[ 2 k + 2 \sum_{n=0}^{k_0 - 1} M^n (c + \eps)^{-n} \big] (c + \eps)^k.
\end{align*}
Now suppose that equation (\ref{eq:claim}) holds for some $F$ and we will show that it holds for $F \cup \{i\}$ for $i \notin F$.
As before for $k \geq |N| \cdot k_0$ and $n \leq k_0$ we have
\begin{align*}
\sum_{n=0}^{k_0 - 1} 
\| \sum_{|\mu_i| = n} \sca{x_{i,\mu_i}, x_{i, \mu_i}} \|_{A} \cdot 
\| \sum_{|\umu| = k - n, \ell(\umu) \in F} \sca{x_{\umu}, x_{\umu}} \|_{A}
& \leq \\
& \hspace{-6cm} \leq
\sum_{n=0}^{k_0 - 1} M^n p_F(k - n) (c + \eps)^{k-n} 
\leq
\big[\sum_{n=0}^{k_0 - 1} M^n (c + \eps)^{-n} \big] p_F(k) (c + \eps)^k.
\end{align*}
Here we use that $p_F$ has positive co-efficients so that $p_F(k-n) \leq p_F(k)$ for $n \leq k_0 \leq k$.
Its symmetrical is given by
\begin{align*}
\sum_{n=k - k_0 + 1}^{k} 
\| \sum_{|\mu_i| = n} \sca{x_{i,\mu_i}, x_{i, \mu_i}} \|_{A} \cdot 
\| \sum_{|\umu| = k - n, \ell(\umu) \in F} \sca{x_{\umu}, x_{\umu}} \|_{A}
& \leq \\
& \hspace{-7cm} \leq
\sum_{n=0}^{k_0 - 1} 
\| \sum_{|\mu_i| = k - n} \sca{x_{i,\mu_i}, x_{i, \mu_i}} \|_{A} \cdot 
\| \sum_{|\umu| = n, \ell(\umu) \in F} \sca{x_{\umu}, x_{\umu}} \|_{A} \\
& \hspace{-7cm} \leq
\big[\sum_{n= 0}^{k_0 - 1} M^n (c + \eps)^{-n}\big] (c + \eps)^k.
\end{align*}
Moreover we have
\begin{align*}
\sum_{n=k_0}^{k - k_0} 
\| \sum_{|\mu_i| = n} \sca{x_{i,\mu_i}, x_{i, \mu_i}} \|_{A} \cdot 
\| \sum_{|\umu| = k - n, \ell(\umu) \in F} \sca{x_{\umu}, x_{\umu}} \|_{A}
& \leq \\
& \hspace{-7cm} \leq
\sum_{n=k_0}^{k - k_0} (c + \eps)^n p_F(k-n) (c + \eps)^{k-n} 
\leq
\sum_{n= k_0}^{k - k_0} p_F(k) (c + \eps)^k
\leq 2 k p_F(k) (c + \eps)^k.
\end{align*}
We thus conclude
\begin{align*}
\|\sum_{|\mu_i| + |\umu| = k, \ell(\umu) \in F} \sca{x_{i,\mu_i} \otimes x_{\umu}, x_{i,\mu_i}\otimes x_{\umu}} \|_{A}
& \leq \\
& \hspace{-3cm} \leq
\bigg[2 k p_F(k) + (1 + p_F(k)) \sum_{n= 0}^{k_0 - 1} M^n (c + \eps)^{-n} \bigg] (c + \eps)^k.
\end{align*}
We have that the polynomial
\[
p_{F \cup \{i\}}(k) := 2 k p_F(k) + (1 + p_F(k)) \sum_{n= 0}^{k_0 - 1} M^n (c + \eps)^{-n}
\]
has positive co-efficients (and degree $1 + \deg p_F = |F|$).
This proves the inductive step and the proof is complete.
\end{proof}

\begin{proposition}\label{P:entropy2}
Let $X$ be a product system of finite rank over $A$ and let $\be > 0$.
If $\tau \in \Tr_\be^F(A)$ for some $F \subseteq \{1, \dots, N\}$ then $h_X^{\tau, F} \leq h_X^{\tau} \leq \be$.
\end{proposition}

\begin{proof}
Let $\tau \in \Tr_\be^F(A)$ for $\be >0 $.
As we are adding positive reals it follows that $h_X^{\tau, F} \leq h_X^{\tau}$.
If $F = \mt$ then it is immediate that
\[
h_X^\tau = \limsup_k \frac{1}{k} \log \bigg[ \sum_{|\umu| = k} \tau(\sca{x_{\umu}, x_{\umu}}) \bigg]
= \limsup_k \frac{1}{k} \log e^{k \be} = \be.
\]
Let $F \neq \mt$ and fix $\eps >0$.
Since $c_{\tau, \be}^F <\infty$ there exists a $k_0 > 0$ such that 
\[
\bigg[ e^{-k\be} \sum_{|\umu| = k, \ell(\umu) \in F} \tau(\sca{x_{\umu}, x_{\umu}}) \bigg]^{1/k}
\leq (1 + \eps)
\foral k \geq k_0.
\]
Set also
\[
M := \max \{1, \bigg[ e^{-k\be} \sum_{|\umu| = k, \ell(\umu) \in F} \tau(\sca{x_{\umu}, x_{\umu}}) \bigg]^{1/k} \mid k = 0, \dots, k_0-1 \}.
\]
For $k \geq k_0$ and $m \leq k_0 - 1$ we have
\begin{align*}
\sum_{|\umu| = k, |\ell(\umu)_{F}| = m} \tau(\sca{x_{\umu}, x_{\umu}})
& =
\sum_{|\umu| = m, \ell(\umu) \in F} \bigg[ \sum_{|\unu| = k - m, \ell(\unu) \perp F} \tau(\sca{x_{\unu},  \sca{x_{\umu}, x_{\umu}} x_{\unu}}) \bigg] \\
& =
e^{(k - m)\be} \sum_{|\umu| = m, \ell(\umu) \in F} \tau(\sca{x_{\umu}, x_{\umu}}) 
\leq
e^{k \be} M^m
\leq
e^{k \be} M^{k_0}.
\end{align*}
A similar computation as above (with $1+\eps$ in place of $M$) for $k \geq m \geq k_0$ yields
\begin{align*}
\sum_{|\umu| = k, |\ell(\umu)_{F}| = m} \tau(\sca{x_{\umu}, x_{\umu}})
\leq
e^{k \be} (1 + \eps)^{m}
\leq
e^{k \be} (1 + \eps)^{k}.
\end{align*}
Therefore for $k \geq k_0$ we get
\begin{align*}
\bigg[ \sum_{|\umu| = k} \tau(\sca{x_{\umu}, x_{\umu}}) \bigg]^{1/k}
& =
\bigg[\sum_{m=0}^{k_0-1} \sum_{|\umu| = k, |\ell(\umu)_F| = m} \tau(\sca{x_{\umu}, x_{\umu}}) +
\sum_{m=k_0}^{k} \sum_{|\umu| = k, |\ell(\umu)_F| = m } \tau(\sca{x_{\umu}, x_{\umu}}) \bigg]^{1/k} \\
& \leq
\big[ k_0 M^{k_0} + 2k(1 + \eps)^k \big]^{1/k} e^{\be}.
\end{align*}
The dominating sequence converges to $(1+ \eps) e^\be$.
As $\eps >0$ was arbitrary we obtain the required $h_X^\tau \leq \be$.
\end{proof}

\begin{theorem}\label{T:entropy}
Let $X$ be a product system of finite rank over $A$.
Then
\begin{align*}
h_X 
 = \max\{0, \inf \{h_X^{\tau} \mid \tau \in \Tr(A)\} \}.
\end{align*}
If in addition $h_X > 0$ then there exists a $\tau \in \Tr(A)$ such that $h_X = h_\tau$.
\end{theorem}

\begin{proof}
Let $\be > 0$ such that $\Eq_\be(\N\T(X)) \neq \mt$.
By Proposition \ref{P:gi} we moreover have that $\GEq_\be(\N\T(X)) \neq \mt$.
Then by Theorem \ref{T:con} and Theorem \ref{T:para} there exists a $\tau \in \Tr_\be^F(A)$ for some $F$.
Proposition \ref{P:entropy2} then implies that $h_X^{\tau} \leq \be$. 
Therefore 
\[
\inf \{h_X^{\tau} \mid \tau \in \Tr(A)\} \leq h_X.
\]

If $h_X = 0$ then there is nothing to show.
Suppose that $h_X > 0$ and assume that there exists a $\tau \in \Tr(A)$ with $0 \leq h_X^\tau < h_X$.
Then for a positive $\be \in (h_X^\tau, h_X)$, the root test implies that $c_{\tau, \be}^{\{1, \dots, N\}} <\infty$.
Hence $\tau \in \Tr_\be^{\fty}(A)$ and by Theorem \ref{T:para} it induces a $\Phi_\tau$ in $\Eq_{\be}^{\fty}(\N\T(X))$.
This is a contradiction as it should be that $\Eq_\be(\N\T(X)) = \mt$ by the choice of $\be$.
Hence $h_X$ is the infimum of the tracial entropies.

Furthermore, by weak*-continuity there exists an equilibrium state at $\be = h_X$ for $\N\T(X)$.
Assuming $h_X >0$, by Theorem \ref{T:con} and Theorem \ref{T:para} there is a $\tau \in \Tr_{h_X}^{F}(A)$ for some $F$, and Proposition \ref{P:entropy2} gives $h_X^{\tau} \leq h_X$.
However we also have $h_X \leq h_X^{\tau}$ by the first part, thus obtaining equality.
\end{proof}

We now turn our attention to ground states and KMS${}_\infty$-states.
First let us show that KMS${}_\infty$-states do exist and can only come as limits of equilibrium states of finite type.

\begin{proposition}\label{P:strong entropy}
Let $X$ be a product system of finite rank over $A$ and let $F \subseteq \{1, \dots, N\}$.
Then
\[
\Eq_\be^C(\N\T(X)) = \mt \foral C \subsetneq F \text{ and } \be > h_X^{s,F}.
\]
Consequently, we have that
\[
\Eq_\be(\N\T(X)) = \Eq_\be^{\fty}(\N\T(X)) \simeq \Tr(A) \foral \be > h_X^s.
\]
\end{proposition}

\begin{proof}
First consider the case for $C = \mt$ and let $\vphi \in \Eq_\be^{\infty}(\N\T(X))$.
By Proposition \ref{P:gi} we can substitute $\vphi$ with a gauge-invariant one.
Hence without loss of generality assume that $\vphi = \vphi E$ and let $\tau \in \Avt_\be(A)$ such that $\vphi = \Phi^{\infty}(\tau) \in \Eq_\be^{\infty}(\N\T(X))$.
Then for any $F \subseteq \{1, \dots, N\}$ and a unit decomposition $x = \{x_{i,j} \mid j=1, \dots, d_i, i=1, \dots, N\}$ we have
\[
\be = \lim\sup_k \frac{1}{k} \sum_{|\umu| = k, \ell(\umu) \in F} \tau(\sca{x_{\umu}, x_{\umu}}) \leq \limsup_k \frac{1}{k} \| \sum_{|\umu| = k, \ell(\umu) \in F} \sca{x_{\umu}, x_{\umu}} \|_{A}
=
h_X^{x, F}.
\]
Taking infimum over $x$ shows that $\be \leq h_X^{s,F}$.

Now, fix $\mt \neq C \subsetneq F$ and let $\vphi \in \Eq_\be^C(\N\T(X))$.
Once more by Proposition \ref{P:gi} we may take $\vphi$ to be gauge-invariant.
For $\un{n} \in F \setminus C$ take the projection
\begin{align*}
P_{\un{n}} 
& : =
\sum \{p_{\un{m}} \mid \un{m} \geq \un{n}\}
=
\prod_{i \in  \supp \un{n}} P_{n_i \cdot \Bi} \\
& =
\prod_{i \in \supp \un{n}} (1 - (1 - P_{n_i \cdot \Bi}))
=
\sum_{D \subseteq \supp \un{n}} (-1)^{|D|} \prod_{j \in D} (1 - P_{n_j \cdot \bo{j}}).
\end{align*}
However we have that
\[
1 - P_{n_j \cdot \bo{j}} = \sum_{k=0}^{n_j - 1} \sum_{\ell(\umu) = k \cdot \bo{j}} t(x_{\umu}) Q_{\bo{j}} t(x_{\umu})^*.
\]
Since $j \in \supp \un{n} \perp C$ then $\vphi(Q_{\bo{j}}) = 0$ and thus $\vphi(1 - P_{n_j \cdot \bo{j}}) = 0$.
The Cauchy-Schwartz inequality then yields
\[
\vphi(\prod_{j \in D} (1 - P_{n_j \cdot \bo{j}}))
=
\begin{cases}
1 & \text{ if } D = \mt, \\
0 & \text{ if } \mt \neq D \subseteq \supp \un{n},
\end{cases}
\]
and thus $\vphi(P_{\un{n}}) = 1$.
However $P_{\un{n}} = \sum_{\ell(\umu) = \un{n}} t(x_{\umu}) t(x_{\umu})^*$ and therefore for $|\un{n}| = k$ we get
\begin{align*}
1
&
= \vphi(P_{\un{n}})
=
e^{-k \be} \sum_{\ell(\umu) = \un{n}} \vphi \pi(\sca{x_{\umu}, x_{\umu}}) \\
& \leq
e^{-k\be} \sum_{|\umu| = k, \ell(\umu) \in F} \vphi \pi (\sca{x_{\umu}, x_{\umu}}) 
 \leq
e^{-k \be} \| \sum_{|\umu| = k, \ell(\umu) \in F} \sca{x_{\umu}, x_{\umu}} \|_{A},
\end{align*}
where we used that $\vphi \pi \in \Tr(A)$.
From the latter it follows that $\be \leq h_X^{x, F}$ and consequently that $\be \leq h_X^{s,F}$.

In particular we have $\Eq_\be(\N\T(X)) = \Eq_\be^{\fty}(\N\T(X))$ for all $\be > h_X^s$.
It remains to show that $\Tr(A) = \Tr_\be^{\{1, \dots, N\}}(A)$ when $\be > h_X^s$.
Let $x = \{x_{i,j} \mid j=1, \dots, d_i, i=1, \dots, N\}$ be a unit decomposition of $X$ such that $\be > h_X^x \geq h_X^s$.
For every $\tau \in \Tr(A)$ we have $h_X^\tau \leq h_X^s \leq h_X^x$ and thus
\[
\limsup_k \bigg[ e^{-k\be} \sum_{|\umu| = k} \tau(\sca{x_{\umu}, x_{\umu}}) \bigg]^{1/k}
\leq
e^{- \be} e^{h_X^x} < 1.
\]
Therefore $c_{\tau, \be}^{\{1, \dots, N\}} < \infty$, and the proof is complete.
\end{proof}

Now we can provide the characterization of the limit states.

\begin{theorem}\label{T:ground}
Let $X$ be a product system of finite rank over $A$.
Then there exists an affine weak*-homeomoprhism $\Psi$ between the states $\tau \in \S(A)$ and the ground states of $\N\T(X)$ such that
\begin{equation*} 
\Psi_\tau(\pi(a)) = \tau(a) \foral a \in A
\qand
\Psi_\tau(t(\xi_{\un{n}}) t(\eta_{\un{m}})^*) = 0 \text{ when } \un{n} + \un{m} \neq \un{0}.
\end{equation*}
The restriction of $\Psi$ to the tracial states $\Tr(A)$ induces a weak*-homeomorphism onto the KMS${}_\infty$-states of $\N\T(X)$.

If $\{I_F \mid \mt \neq F \subseteq \{1, \dots, N\}\}$ is a lattice of $\perp$-invariant ideals of $A$ then the corresponding weak*-homeomorphisms for $\N\O(I, X)$ arise by restricting on states that annihilate the ideal $I_{\{1, \dots, N\}}$.
\end{theorem}

\begin{proof}
For $\N\T(X)$ let $\tau \in \S(A)$ and consider its GNS-representations $(x_\tau, H_\tau, \rho_\tau)$.
Let $(\rho, v) = (\pi \otimes I, t \otimes I)$ be the induced representation on $H = \F X \otimes_{\rho_\tau} H_{\tau}$ and define the vector state
\[
\Psi_\tau(f) := \sca{1_A \otimes x_\tau, (\rho \times v)(f) 1_A \otimes x_\tau}_{H} = \tau(Q f Q)
\qfor
Q := \prod_{i=1}^N Q_{\Bi}.
\]
This map is clearly continuous and injective.
Surjectivity and restriction to KMS${}_\infty$-states follows verbatim from \cite[Proof of Theorem 9.1]{Kak17}.
Here we need to use Proposition \ref{P:strong entropy} so that every KMS${}_\infty$-state is a limit of finite-type states.

For $\N\O(I, X)$ we notice that
\[
\Psi_\tau(\pi(a) Q_F) = \tau(Q \pi(a) Q_F Q) = \tau(Q \pi(a) Q) = \tau(a).
\]
Therefore $\Psi_\tau$ factors through $q_{\K_I} \colon \N\T(X) \to \N\O(I, X)$ if and only if $\tau|_{I_{F}}$ for all $F$, if and only if $\tau|_{I_{\{1, \dots, N\}}}$, since $I_F \subseteq I_{\{1, \dots, N\}}$ for all $F$.
\end{proof}

\section{Examples}\label{S:examples}

\subsection{$\bZ_+^N$-dynamics}

When every $X_i$ admits an orthonormal basis then it is easy to deduce that $h_X = h_X^\tau = h_X^s = \max_i \log d_i$.
A moment's thought also indicates that if $\Tr_\be^F(A) \neq \mt$ then $\be = \log d_i$ for all $i \notin F$.
Hence the only possible points of breaking symmetry is where some of the $d_i$'s coincide.

An example in this respect is studied in \cite{Kak17} for $\bZ_+^N$-dynamical systems in the sense of \cite{DFK14}.
In this case let $\al \colon \bZ_+^N \to \End(A)$ be an action by unital endomorphisms on a C*-algebra $A$.
The associated product system is given by
\[
X_{\un{n}} = A_A
\text{ with }
a \cdot \xi_{\un{n}} = \al_{\un{n}}(a) \xi_{\un{n}}
\]
where the identification $\xi_{\un{n}} \otimes \xi_{\un{m}} \to \xi_{\un{n}} \xi_{\un{m}}$ is given by multiplication in $A$.
Then $\N\T(X)$ is the universal C*-algebra with respect to $(\pi, V_1, \dots, V_N)$ such that the $V_i$ are doubly commuting isometries and $\pi(a) V_{\Bi} = V_{\Bi} \pi\al_{\Bi}(a)$.
By \cite{DFK14} we have in particular that $\N\O(X)$ is the quotient by the relations
\[
\pi(a) \prod_{i \in F} (I - V_\Bi V_{\Bi}^*) \foral a \in \I_F = \bigcap_{\un{n} \perp F} \al_{\un{n}}^{-1}( (\bigcap_{i \in F} \ker \al_{\Bi})^\perp ).
\]

In this case $d_i = 1$ for all $i$ and all equilibrium states are gauge-invariant.
Therefore $\Eq_\be(\N\T(X)) = \Eq_\be^{\fty}(\N\T(X))$ for all $\be > 0$.
The equilibrium states of infinite type correspond to tracial states of the universal C*-algebra subject to the relations above for (doubly) commuting unitaries $U_\Bi$.
We may intrinsically identify $\N\O(A, X)$ with the C*-crossed product of the system given by $\ad_{U_{\Bi}} \in \Aut(q\pi(A))$ where $q \colon \N\T(X) \to \N\O(A, X)$.

\subsection{Higher-rank graphs}

Let $X$ arise from an irreducible higher-rank graph $\La$, so that the single-coloured subgraphs are irreducible.
We will not repeat the construction of higher-rank graphs here and the reader may refer to \cite{RSY04} for the details.
By \cite[Theorem 8.2]{Kak17} and Proposition \ref{P:entropy1} we have that
\[
\log \la = h_X^{\tau, i} \leq h_X^\tau \leq h_X^s = \max_i h_X^{s,i} = \log \la,
\]
where $\la$ is the common Perron-Frobenius eigenvalue of the commuting adjacency matrices.
Thus $h_X = \log \la$.
Moreover this shows that all states above $\log \la$ are of finite type and thus gauge-invariant.
Now at $\be = \log \la$ there exists a unique infinite-type state arising from the common Perron-Frobenius eigenvector $w_{\textup{PF}}$.
The proof is similar to that of \cite[Proof of Theorem 8.2]{Kak17}.
Hence the $\Eq$-structure for $\N\T(X)$ is
\[
\Eq_\be(\N\T(X))
=
\begin{cases}
\Eq_\be^{\fty}(\N\T(X)) & \text{ if } \be > \log \la, \\
\Eq_\be^{\infty}(\N\T(X)) & \text{ if } \be = \log \la,
\end{cases}
\simeq
\begin{cases}
\P_n & \text{ if } \be > \log \la, \\
\{w_{\textup{PF}}\} & \text{ if } \be = \log \la,
\end{cases}
\]
where $\P_n$ is the probability simplex on the $n$ number of vertices (with dimension $n-1$).
In this way we fully recover the results of \cite{HLRS14}.

\subsection{Multivariable factorial languages}

In \cite{DK18} we consider product systems that arise from multivariable factorial languages (m-FL).
Here we show that the notion of entropy corresponds to that of allowable words.
But first let us recall some elements of the construction from \cite{DK18}.

Fix $d$ symbols and $N$ colours.
The elements of $\fdn$ are denoted by $\umu$ and consist of $N$-tuples of sequences on $d$ elements.
We write $\de_i(k)$ for the element that has the symbol $k$ at the $i$-th position and the empty word in all other places.
We use the operation of co-ordinatewise concatenation.
Then $((\bF^d)^N, \fdn)$ becomes a quasi-lattice by using the partial order coordinate-wise.

A set $\La^* \subseteq \fdn$ is said to be a \emph{multivariable factorial language} if 
\begin{enumerate}
\item for every $i \in [N]$ there exists at least one $k \in [d]$ such that $\de_i(k) \in \La^*$;
\item if $\umu \in \La^*$ and $\unu \in \umu$ then $\unu \in \La^*$.
\end{enumerate}
As with the one variable case, there is here an alternative definition via forbidden words.

We can use the quantization on $\ell^2(\La^*)$ given by
\[
T_{\de_i(k)} e_{\umu}
=
\begin{cases}
e_{\de_i(k) \ast \umu} & \text{ if } \de_i(k) \ast \umu \in \La^*, \\
0 & \text{ otherwise}.
\end{cases}
\]
The product system related to $\La^*$ is constructed concretely and is given by
\[
X_{\un{n}} = \spn\{T_{\umu} a \mid a \in A, \ell(\umu) = \un{n} \}
\qfor
A = \ca(T_{\umu}^* T_{\umu} \mid \umu \in \La^*).
\]
In particular we get an anti-homomorphism $\al \colon \La^* \to \End(A)$ given by $\al_{\de_i(k)} = \ad_{T_{\de_i(k)}^*}$.
Then $\N\T(X)$ is the universal C*-algebra with respect to $(\pi,V)$ such that:
\begin{enumerate}
\item $V \colon \La^* \to \B(H)$ is a Nica-covariant representation for $\La^*$, in the sense that
\[
V_{\umu} V_{\umu}^* V_{\unu} V_{\unu}^*
=
\begin{cases}
V_{\umu \vee \unu} V_{\umu \vee \unu}^* & \text{ if } \umu \vee \unu \in \La^*, \\
0 & \text{ otherwise},
\end{cases}
\]
\item $\pi(a) V_{\de_i(k)} = V_{\de_i(k)} \pi \al_{\de_i(k)}(a)$ for all $(i, k) \in [N] \times [d]$ and $a \in A$,
\item $V_{\de_i(k)}^* V_{\de_i(l)} = \de_{k, l} \pi(T_{\de_i(k)}^* T_{\de_i(k)})$ for all $i \in [N]$ and $k, l \in [d]$.
\end{enumerate}

We wish to show that all equilibrium states in this case are gauge-invariant.
To this end let $\vphi \in \Eq_\be(\N\T(X)$ and fix $a\in A$ and $\umu, \unu \in \La^*$.
Without loss of generality we assume that $\ell(\umu) \perp \ell(\unu)$.
If $\umu \vee \unu = \infty$ then
\begin{align*}
\vphi(V_{\umu} \pi(a) V_{\unu}^*)
=
e^{-|\umu| \be} \vphi(V_{\unu}^* V_{\umu} \pi(a))
= 0.
\end{align*}
If $\umu \vee \unu < \infty$ and is in $\La^*$ then we have
\begin{align*}
\vphi(V_{\umu} \pi(a) V_{\unu}^*)
& =
e^{-|\umu| \be} \vphi(V_{\unu}^* V_{\umu} \pi(a))
=
e^{-|\umu| \be} \vphi(V_{\umu} V_{\unu}^* \pi( T_{\umu \vee \unu}^* T_{\umu \vee \unu} a)) \\
& =
e^{-2 |\umu| \be} \vphi(V_{\unu}^* V_{\umu} \pi( T_{\umu}^* T_{\umu \vee \unu}^* T_{\umu \vee \unu} a) T_{\umu}).
\end{align*}
Applying repeatedly we thus get that $\vphi(V_{\umu} \pi(a) V_{\unu}^*) \leq e^{-n |\umu| \be} \nor{a}$ for all $n \in \bZ_+$.
This shows that $\vphi(V_{\umu} \pi(a) V_{\unu}^*) = 0$.
A similar argument applies when $\ell(\umu) \wedge \ell(\unu) \neq \un{0}$.

It is clear that $X_{\La^*}$ is of finite rank with the decomposition given by $\{T_{\de_i(k)} \mid k=1, \dots, d\}$ for every $i=1, \dots, N$.
Fix now $F \subseteq \{1, \dots, N\}$ and let the $F$-projections of the allowable words by
\[
c_F(\umu) := \ast_{i \in F} \mu_i.
\]
By definition of $\La^*$ we have that $c_F(\La^*) \subseteq \La^*$ as $c_F(\umu) \in \umu$.
Notice that the $T_{\umu}^* T_{\umu}$ are pairwise orthogonal and therefore the $F$-strong entropy is given by
\[
h_X^{s, F} = \lim_k \frac{1}{k} \log\left( |B_k^F(\La^*)| \right)
\; \text{ for } \;
B_{k}^F(\La^*):= \{\umu \in \La^* \mid c_F(\umu) = \umu, |\umu| = k\}.
\]
Hence $h_X^{s, F}$ measures the entropy of the allowable part supported on $F$.
Likewise the strong entropy $h_X^s$ coincides with the entropy of the allowable words.

\begin{acknow}
The author would like to thank Sergey Neshveyev for pointing out an error in a preprint of this paper.

The author would like to dedicate this paper to Aimilia, on the occasion of her \emph{first} xx-th birthday.
Many happy returns.
\end{acknow}


\end{document}